\documentclass{article}

\usepackage{olo}
\usepackage{footmisc}
\usepackage{enumitem}
\usepackage{booktabs}
\usepackage{adjustbox}

\hypersetup{
    colorlinks=true,   
    urlcolor=blue,
    citecolor=blue,
    }

\newcommand{\vwo}{\vw_\ast}
\newcommand{\hvw}{\hat\vw}
\newcommand{\tvw}{\tilde\vw}
\newcommand{\vwt}{\vw_t}
\newcommand{\vwn}{\vw_{t+1}}

\newcommand{\vui}{\vu^t_i}
\newcommand{\vun}{\vu^t_{i+1}}
\newcommand{\vup}{\vu^t_{i-1}}

\newcommand{\lrelu}{\varrho}
\newcommand{\ts}{\tilde s}
\newcommand{\tsigma}{\tilde\sigma}
\newcommand{\taumax}{{\tau_{\max}}}

\newtheorem{assumption}[theorem]{Assumption}

\newcommand{\alg}{\textsf{AGGLIO}\xspace}
\newcommand{\mytitle}{\alg: Global Optimization for Locally Convex Functions}

\newcommand{\suppcite}{appendices\xspace}
\newcommand{\suppurlcite}{\suppcite}
\newcommand{\suppciteshort}{appendices\xspace}

\newcommand{\codeurl}{https://github.com/purushottamkar/agglio/}

\newcommand{\codecite}{repository \href{\codeurl}{[link]}\xspace}

\graphicspath{{figs/}}

\title{\mytitle}

\author{Debojyoti Dey \and Bhaskar Mukhoty \and Purushottam Kar\\IIT Kanpur\\\texttt{\{debojyot,bhaskarm,purushot\}@cse.iitk.ac.in}}

\begin{document}

\maketitle

\begin{abstract}
This paper presents \alg (Accelerated Graduated Generalized LInear-model Optimization), a stage-wise, graduated optimization technique that offers global convergence guarantees for non-convex optimization problems whose objectives offer only local convexity and may fail to be even quasi-convex at a global scale. In particular, this includes learning problems that utilize popular activation functions such as sigmoid, softplus and SiLU that yield non-convex training objectives. \alg can be readily implemented using point as well as mini-batch SGD updates and offers provable convergence to the global optimum in general conditions. In experiments, \alg outperformed several recently proposed optimization techniques for non-convex and locally convex objectives in terms of convergence rate as well as convergent accuracy. \alg relies on a graduation technique for generalized linear models, as well as a novel proof strategy, both of which may be of independent interest. Code for \alg is available at the following \codecite.
\end{abstract}

\section{Introduction}
\label{sec:intro}

The Generalized Linear Model (GLM) \cite{NelderWedderburn1972} provides a convenient extension to the linear regression model by allowing responses to be transformed using an \emph{inverse link} function. More specifically, given a covariate $\vx \in \bR^d$ and a gold linear model $\vwo \in \bR^d$, the corresponding response $y \in \bR$ is generated such that
\[
\E{y} = g^{-1}(\vx^\top\vwo),
\]
where $g^{-1}: \bR \rightarrow \bR$ is the inverse link function. GLMs play a central role in machine learning, allowing predictions to be made across a variety of label spaces: using the identity link retrieves least squares regression for real-valued labels, the logit link offers logistic regression for binary-valued labels, whereas the logarithmic link yields Poisson regression for count-valued labels.

The classical GLM is canonically associated with a unique likelihood distribution of the exponential family that allows efficient model estimation via likelihood maximization (e.g., using the IRLS method). The consistency properties of such \emph{M-estimators} is well studied \cite{BalakrishnanWY2017,McCullaghNelder1989}. However, contemporary applications frequently forgo this canonical association in favor of a direct modelling as
\begin{equation}
y = \phi(\vx^\top\vwo) + \epsilon,
\label{eq:gen-main}
\end{equation}
where $\vwo$ is the \emph{gold} model, $\phi: \bR \rightarrow \bR$ is an \emph{activation} function and $\epsilon$ captures model mis-specification and statistical noise. This is especially true when GLMs are recursively applied to yield deep ``neural'' models which popularly use activation functions such as the sigmoid, the ReLU \cite{NairHinton2010} and its variants e.g. \emph{leaky} ReLU \cite{MaasHN2013}, softplus \cite{GlorotBB2011}, GeLU and SiLU \cite{HendrycksGimpel2016}. Indeed, no nice exponential family likelihood distribution exists corresponding to the ReLU-style activation functions such as GeLU, SiLU etc.

In such situations, given data $\bc{(\vx_i,y_i)}_{i=1}^n$ generated using the mechanism specified in \eqref{eq:gen-main}, instead of performing likelihood maximization to recover the gold model $\vwo$, it is common to use a task-specific loss function to perform model recovery. For instance, the squared loss could be used for regression problems,
\begin{equation}
\hat\vw = \argmin_{\vw\in\bR^d}\ \frac1n\sum_{i=1}^n(y_i - \phi(\vx_i^\top\vw))^2,
\label{eq:obj-main}
\end{equation}
binary cross entropy could be used for classification, etc. However, objective functions of the form in \eqref{eq:obj-main} are non-convex for most activation functions, including sigmoid, ReLU, and its variants. This problem is magnified for recursive GLMs giving rise to the challenge of training deep networks provably accurately.

\noindent\textbf{Key Idea.} This paper shows that for a wide range of popular activation functions, objectives such as those presented in \eqref{eq:obj-main} are actually \emph{locally} convex around their global optimum, in fact strongly so. However, this strong convexity is restricted to a tiny neighborhood around the optimum and thus, cannot be directly exploited in practice (since the objective function is decidedly non-convex at a global scale). The paper first establishes a more robust version of this result that shows that by \emph{graduating} the activation function appropriately, this coveted neighborhood of strong convexity can be expanded, allowing standard GD to offer a linear rate of convergence to the global optimum. However, this expansion comes at a cost of diminished gradients that slow down the progress of descent algorithms such as GD/SGD. Balancing these two opposing effects using an adaptive graduation schedule yields \alg.

\noindent\textbf{Contributions.} This paper develops the \alg method that offers convergence to the global optimum at a linear rate for several GLM problems of the kind presented in \eqref{eq:obj-main}. The paper introduces novel concepts of \emph{Extendable Local Strong Convexity} (ELSC) and proof techniques that may find independent application to other problems. In experiments, \alg outperforms several recently proposed optimization techniques for locally convex objectives in terms of convergence rate as well as the accuracy of the final model.

\section{Related Works}
\label{sec:related}

\noindent\textbf{Non-convex Optimization.} The area of non-convex optimization is too vast to be surveyed here in any detail. We instead refer the reader to surveys and monographs such as \cite{DanilovaDGGGKS2020,HastieTW2015,JainKar2017} for a more expansive discussion. The problem areas of compressive sensing, sparse recovery, and matrix completion were among the first non-convex problems to be studied in great detail, giving rise to rich literature focusing on relaxation techniques such as LASSO \cite{HastieTW2015} as well as iterative techniques such as IHT \cite{BlumensathD2009}. More recent work has broadened the scope to general-purpose non-convex optimization, partly due to the need to understand the training of deep models better. Consequently, the focus has shifted to descent techniques \cite{DanilovaDGGGKS2020} that offer stationary solutions and hopefully escape spurious saddle points. Of particular interest is the recent work of \cite{SoltanolkotabiJL2019} that initiates a discussion on the landscape offered by objective functions used to train shallow neural networks. However, the core result in this work assumes a quadratic activation function, which is not popularly used in practice. The work also considers popularly used objective functions such as sigmoid and softplus but for those, it only guarantees convergence if initialization is done sufficiently close to the global optimum. In contrast, the method \alg described in this paper presents an effective way to train GLMs with activation functions by minimizing non-convex objectives such as those presented in \eqref{eq:obj-main}. Specifically, \alg guarantees global convergence without requiring special initialization.

\noindent\textbf{Learning with Structured Objectives.} Analyzing cases where the objective function being minimized has special structure has been a successful approach to developing provably non-convex optimization routines. Starting with the Restricted Isometry Property (RIP) originally proposed for compressed sensing \cite{BlumensathD2009}, various extensions suitable for different applications have been proposed, such as restricted strong convexity (RSC) \cite{JainTK2014}, local strong convexity (LSC) \cite{BalakrishnanWY2017,SoltanolkotabiJL2019}, weighted strong convexity (WSC) \cite{MukhotyGJK2019} and strict local quasi convexity (SLQC) \cite{HazanLS-S2015}. Several of these notions were developed in the context of a particular application, e.g., LSC for the EM algorithm, WSC for robust regression, etc., whereas the notion of SLQC was intended for more general usage and has indeed found applications to several problem areas. The analysis of the \alg algorithm in this paper too shall make use of a novel structural concept, namely \emph{Extendable Local Strong Convexity} (ELSC). The paper will also establish that this structure provably holds for objectives arising in natural learning problems.

\noindent\textbf{Graduated/Reweighted Optimization.} As discussed earlier, classical GLM training is typically done using the IRLS algorithm which is a form of reweighted optimization requiring weighted least-squares problems to be solved repeatedly. \cite{MukhotyGJK2019} adapts the IRLS algorithm to solve non-convex robust regression problems instead. For more general non-convex objectives, \cite{HazanLS-S2016} proposes a technique GradOpt to iteratively \emph{smoothen} the objective function allowing a descent procedure to slip into a valley containing the global optimum. Smoothing is done aggressively at first and moderated later. In experiments, \alg was found to consistently outperform GradOpt, possibly since the objective smoothing step in GradOpt is cumbersome to perform exactly and the stochastic alternative used seems to converge slowly. \alg on the other hand, relies on a graduation technique that can be performed exactly and efficiently.

\noindent\textbf{Acceleration.} Recent works applied acceleration techniques such as momentum \cite{CutkoskyMehta2020,WilsonMW2019}, adaptive gradient \cite{ReddiZSKK2018} and variance reduction \cite{ReddiHSPS2016} to non-convex optimization problems. Applying acceleration \alg is an interesting direction for future work but this paper focuses on establishing the core adaptive graduation technique.

\begin{figure*}[t]
	\centering
	\begin{tabular}{@{\hskip 0.04\textwidth}c@{\hskip 0.04\textwidth}c@{\hskip 0.04\textwidth}c@{\hskip 0.04\textwidth}c@{\hskip 0.04\textwidth}}
		\includegraphics[width=0.2\textwidth]{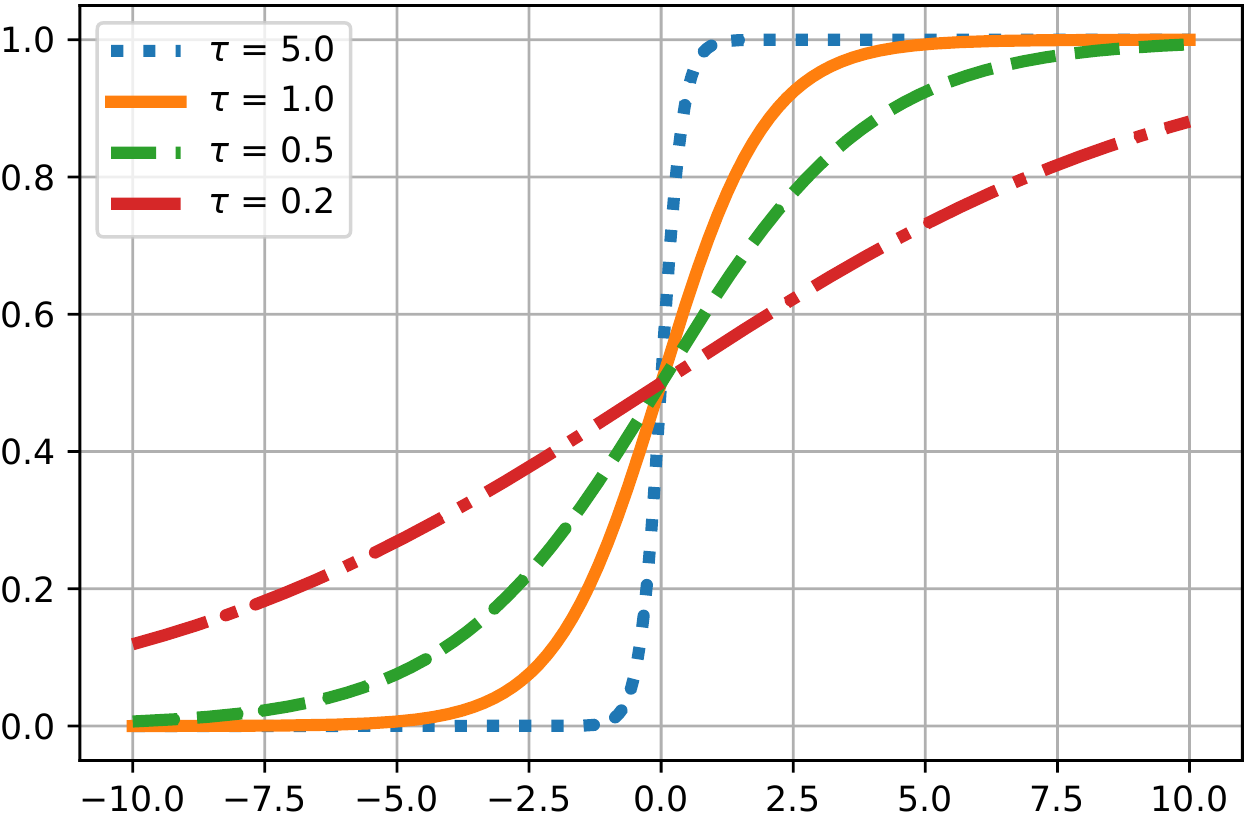} & \includegraphics[width=0.2\textwidth]{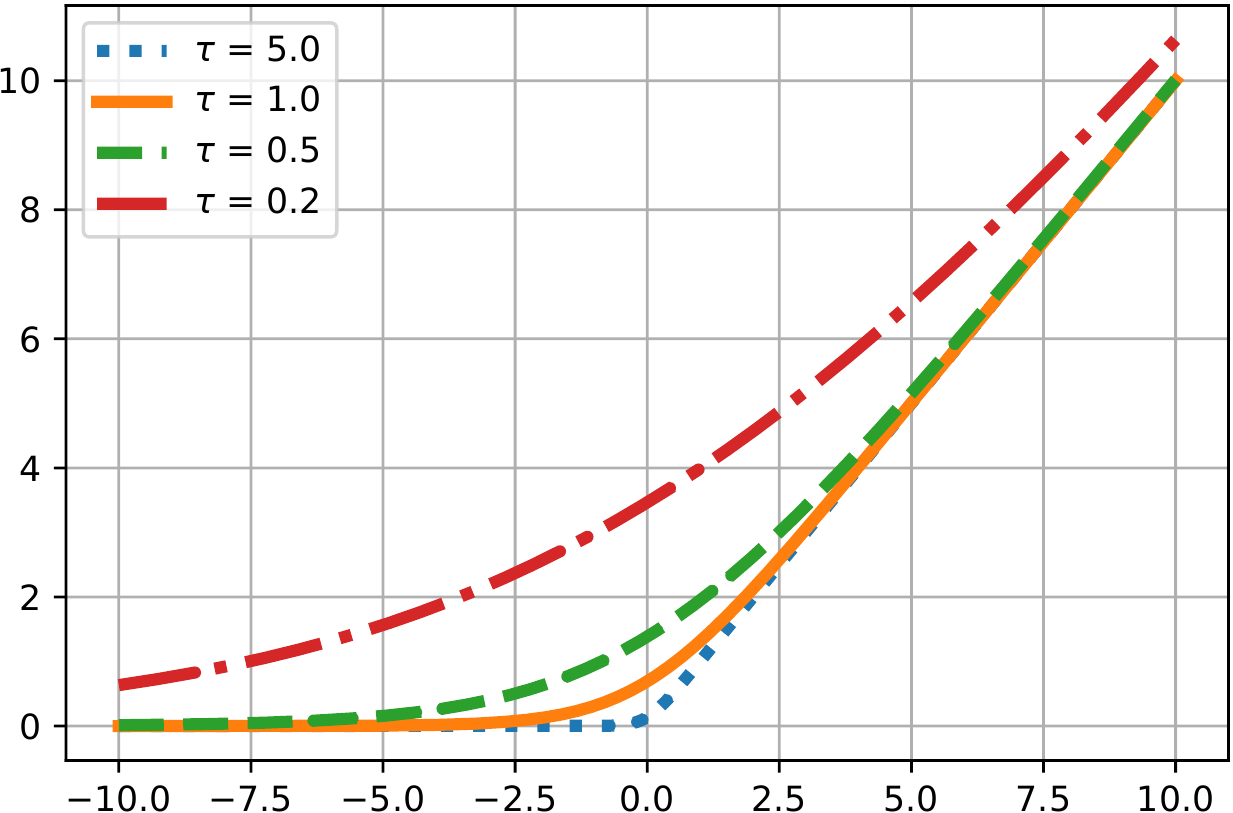} & \includegraphics[width=0.2\textwidth]{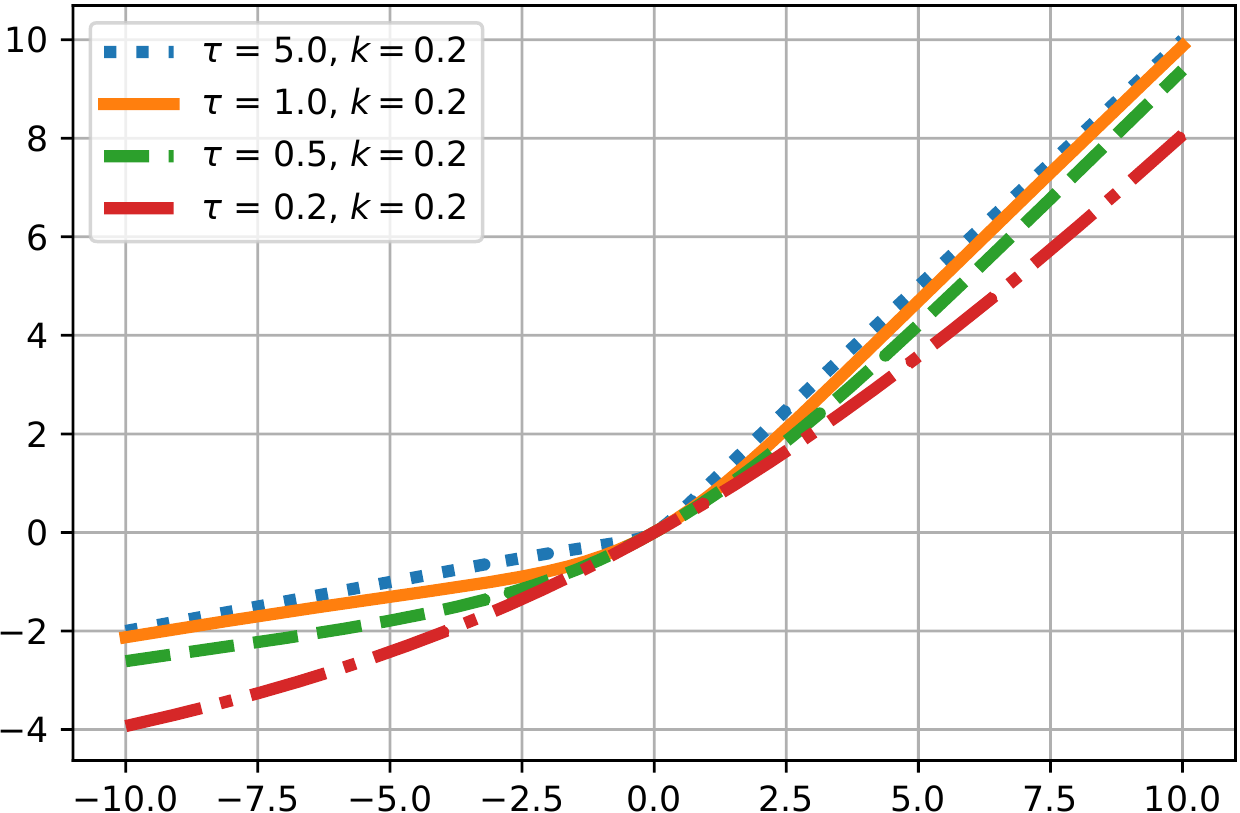} & \includegraphics[width=0.2\textwidth]{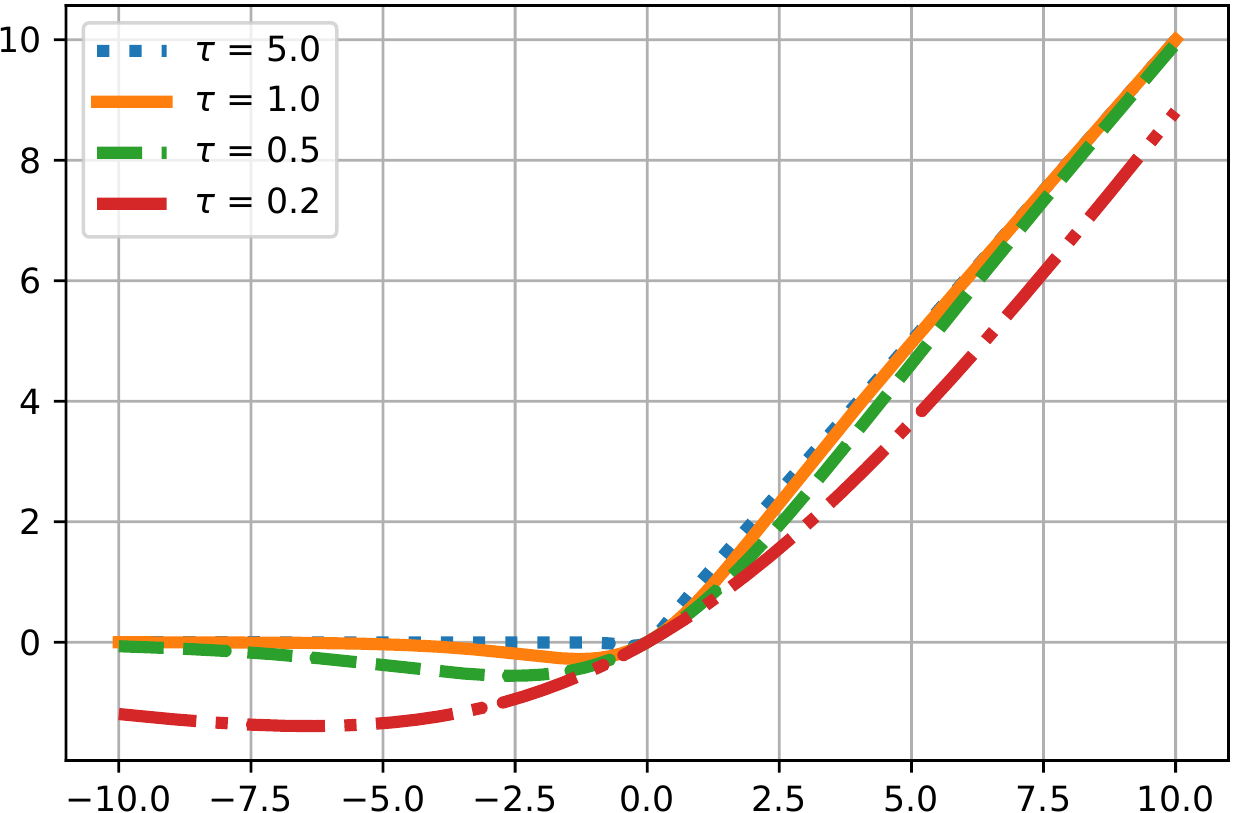} \\
		\includegraphics[width=0.2\textwidth]{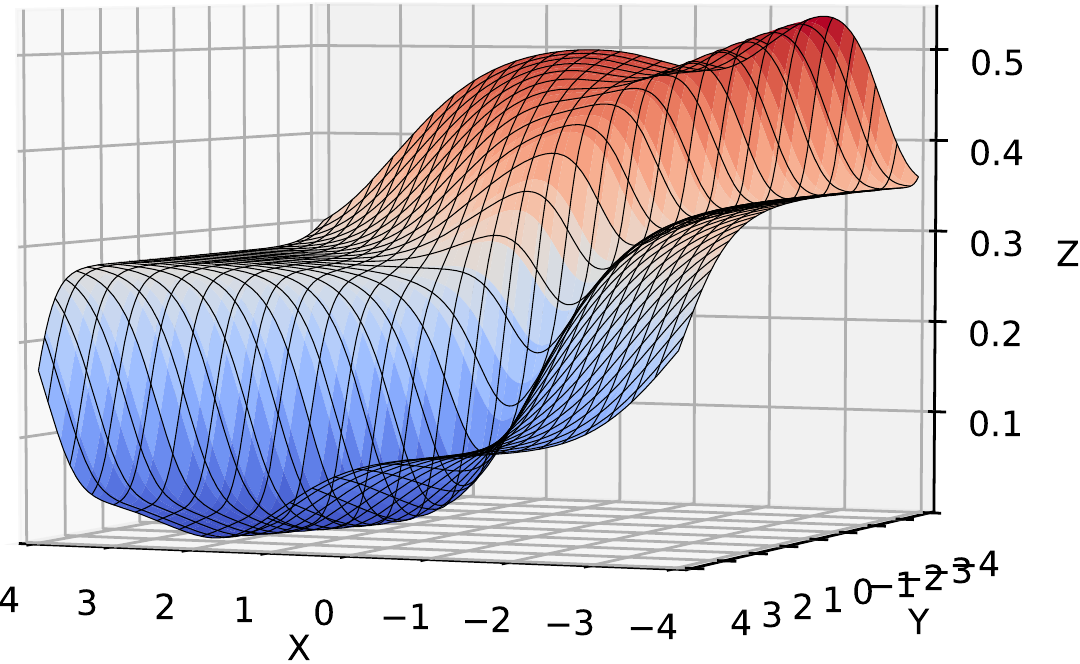} & \includegraphics[width=0.2\textwidth]{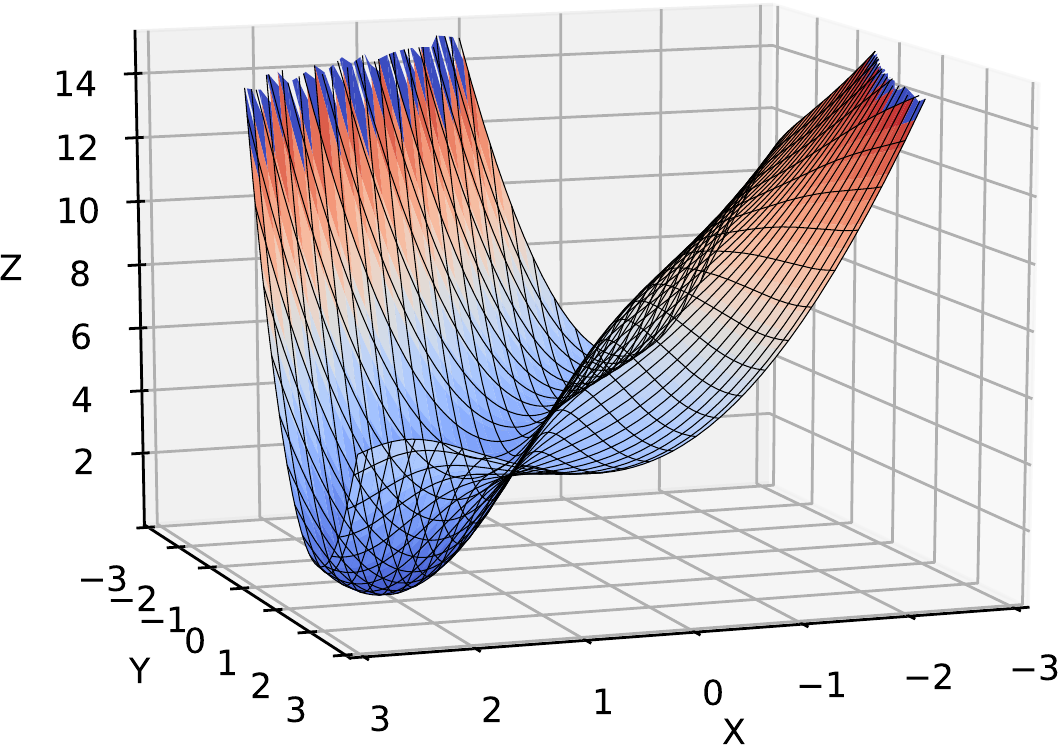} & \includegraphics[width=0.2\textwidth]{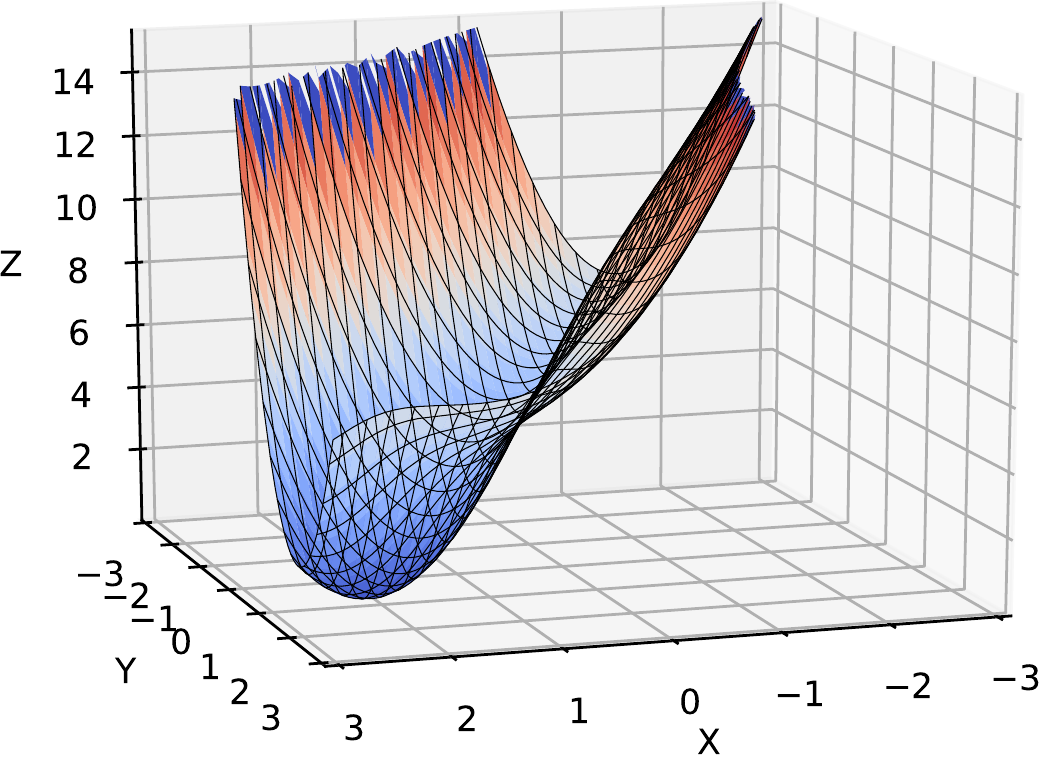} & \includegraphics[width=0.2\textwidth]{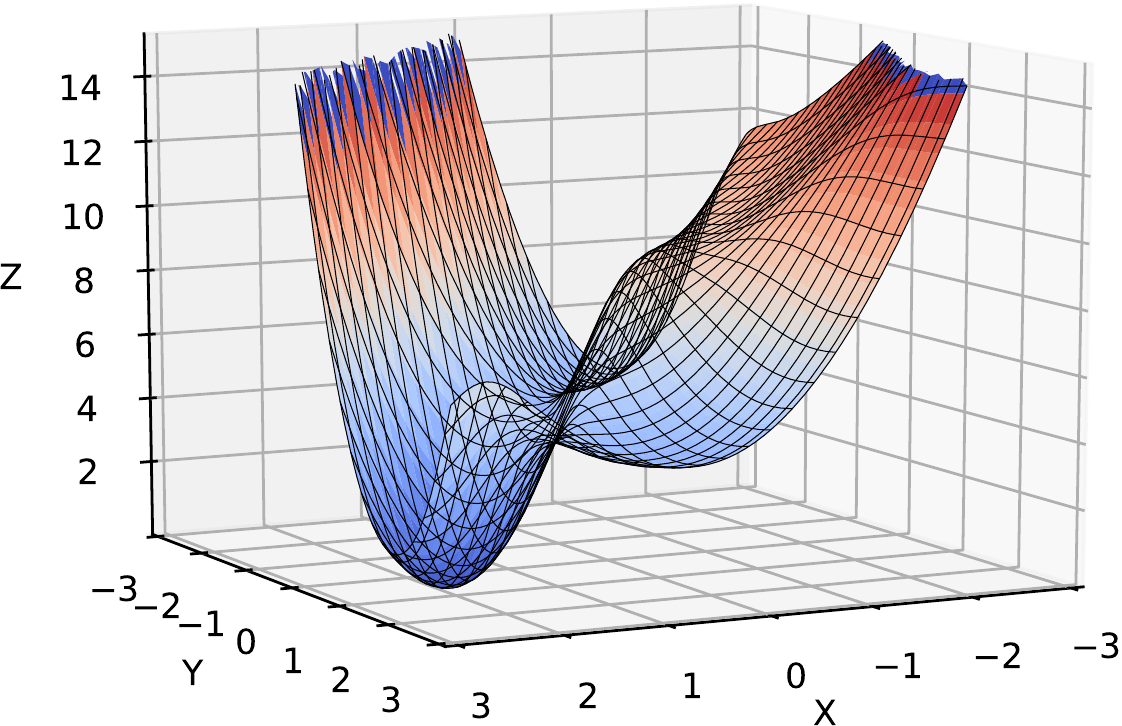}\\
		\includegraphics[width=0.2\textwidth]{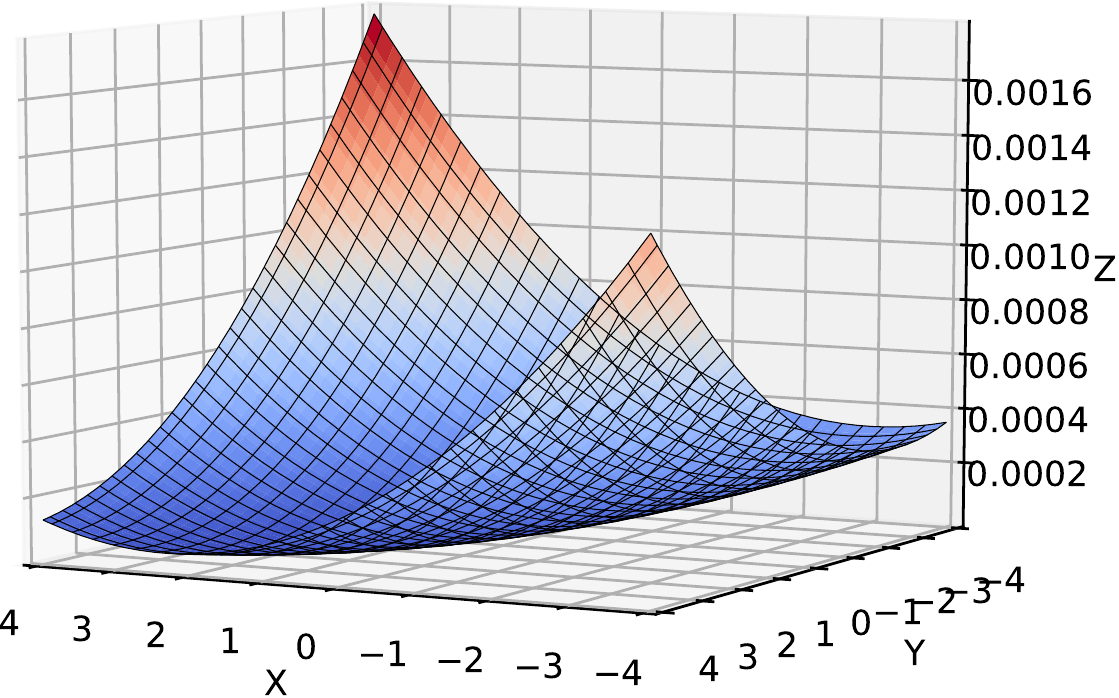} & \includegraphics[width=0.2\textwidth]{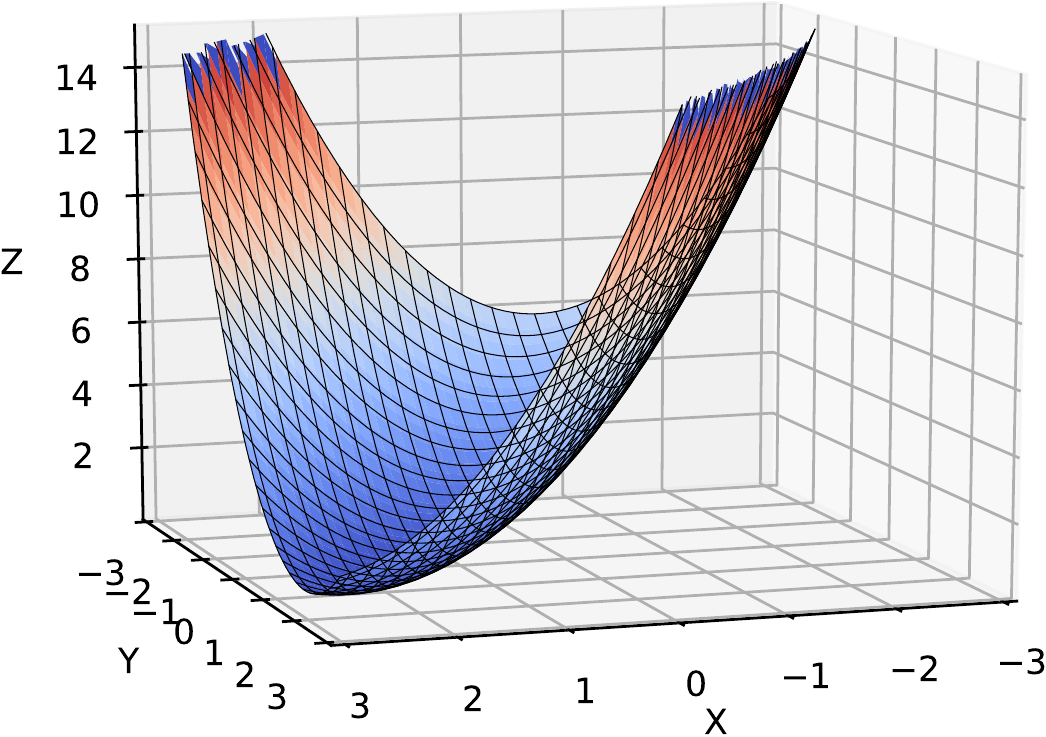} & \includegraphics[width=0.2\textwidth]{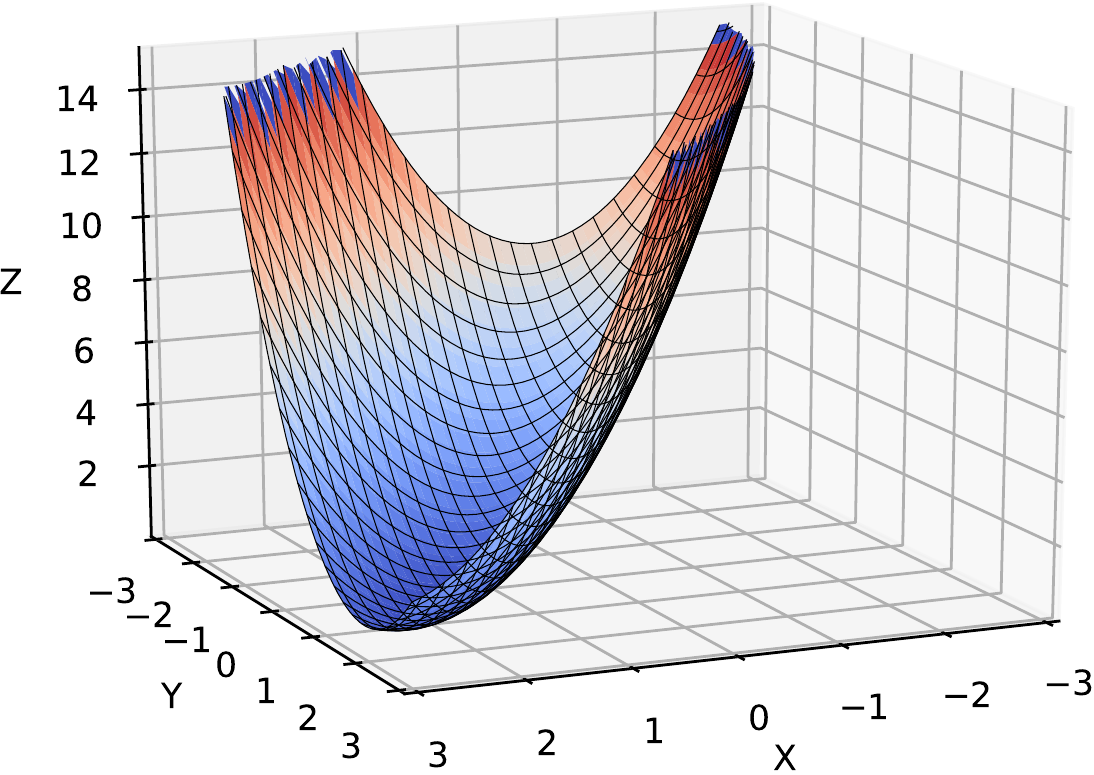} & \includegraphics[width=0.2\textwidth]{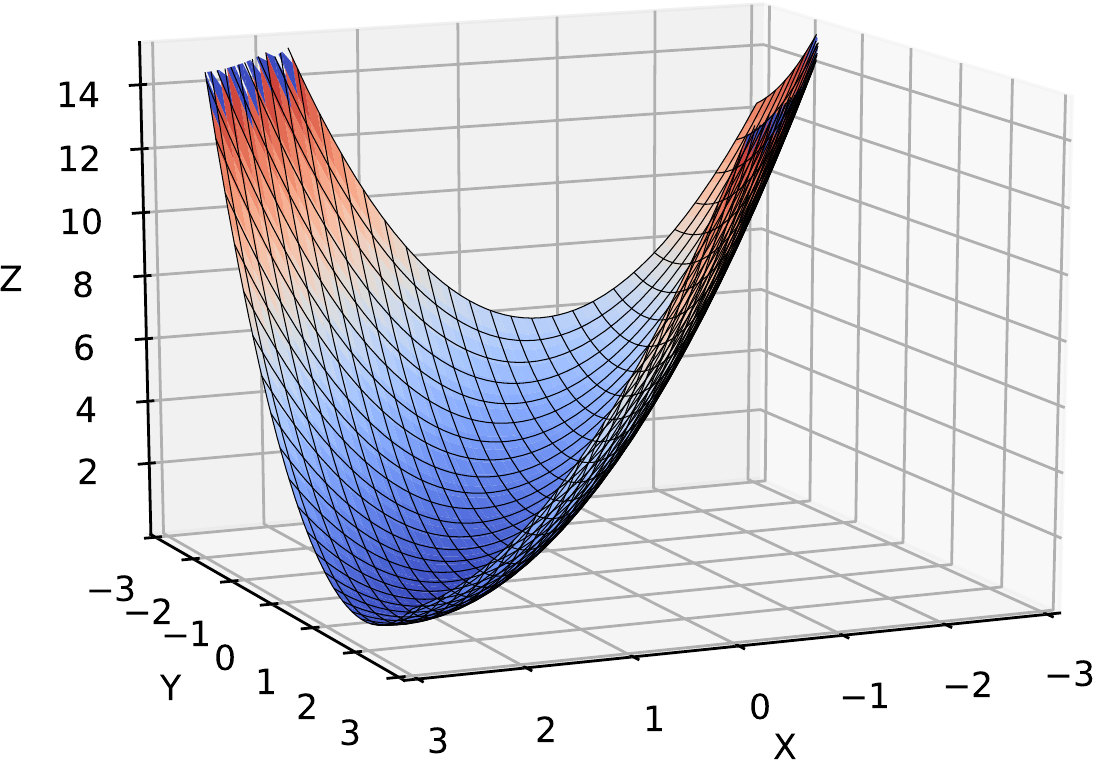}\\[\abovecaptionskip]
		\small (a) Sigmoid & \small (b) Softplus & \small (c) Leaky Softplus & \small (d) SiLU
  \end{tabular}
\caption{Effect of graduation on the local strong convexity of training objectives using 4 different activation functions. (Top Row) Effect of different graduation levels. (Middle Row) The training objective from \eqref{eq:obj-main} in a region around the global optimum $\vwo$ when using un-graduated activations i.e. $\tau = 1$ and $\cL \equiv \cL_1$ is being plotted. (Bottom Row) The objective function in the same neighborhood but with strong graduation $\tau = 0.01$ i.e. $\cL_{0.001}$ is being plotted. Graduation can flatten the objective (e.g. for sigmoid, $z$ axis values are in the range $0 - 0.0016$ after graduation but were $0 - 0.5$ before). Excessive graduation can slow down descent techniques like SGD but \alg's graduation schedule provides a fine balance.}
\label{fig:surfaces}
\end{figure*}

\begin{table*}[t]%
\caption{ELSC/ELSS constants for various activation functions. For any $r > 0$, the shorthand $m_r$ denotes $\max_{\vw \in \cB(\vwo,r)}\ \norm\vw_2$. Note that $m_r \leq \norm\vwo_2 + r$ by triangle inequality. $\tau$ denotes the temperature. The ELSC/ELSS expressions presented here omit universal constants to avoid clutter (exact expressions are available in the \suppcite).}
\label{tab:elsc-elss}
\begin{minipage}{\textwidth}
\centering
\begin{adjustbox}{max width=\textwidth}
\begin{tabular}{lccc}
\toprule
Activation Function & Graduated Expression $\phi_\tau(v)$ & ELSC constant $\lambda_r$ for $\cL_\tau$ & ELSS constant $\Lambda_r$ for $\cL_\tau$\\
\midrule
sigmoid\footnote{To preserve the canonical range of the sigmoid activation i.e. $[0,1]$, it is not subjected to the $\frac1\tau$ normalization.} & $1/(1+\exp(-\tau\cdot v))$ & $\tau^2(\exp(-\tau\cdot m_r) - \tau\cdot r)$ & $\tau^2(\exp(-\tau\cdot m_r) + \tau\cdot r)$\\
softplus & $\frac1\tau\ln(1 + \exp(\tau\cdot v))$ & $1/(1+\exp(\tau^2\cdot m_r^2))^2 - \tau\cdot r$ & $1/(1+\exp(\tau^2\cdot m_r^2))^2 + \tau\cdot r$\\
leaky softplus\footnote{Similar to the leaky ReLU, the leaky softplus uses a \emph{leakiness} parameter $k \in [0,1]$. $k = 0$ retrieves the softplus whereas $k = 1$ retrieves the identity activation function.} & $\frac1\tau\br{\ln(1 + \exp(\tau\cdot v)) - \ln(1 + \exp(-k\tau\cdot v))}$ & $\frac{1+k^2}{(1+\exp(\tau^2\cdot m_r^2))^2} - \tau\cdot r$ & $\frac{1+k^2}{(1+\exp(\tau^2\cdot m_r^2))^2} + \tau\cdot r$\\
\bottomrule
\end{tabular}
\end{adjustbox}
\end{minipage}
\end{table*}

\section{Problem Setting}
\label{sec:setting}

Recall from Section~\ref{sec:intro} that given covariates $\vx_1,\ldots,\vx_n \in \bR^n$ and labels generated as $y_i = \phi(\vx_i^\top\vwo) + \epsilon_i$ using a gold model $\vwo \in \bR^d$ and an activation function $\phi$ such as sigmoid, SiLU etc, our objective to optimize objectives of the form given in \eqref{eq:obj-main} namely
\begin{equation}
\cL(\vw) = \frac1n\sum_{i=1}^n(y_i - \phi(\vx_i^\top\vw))^2.
\label{eq:obj-main-restated}
\end{equation}
To develop \alg, we need the notion of \emph{graduated} activation functions and objectives as defined below.
\begin{definition}[Graduated Activation and Objective]
Given an activation function $\phi: \bR \rightarrow \bR$ and a \emph{temperature} parameter $\tau > 0$, we define the $\tau$-graduated activation function as
\[
\phi_\tau \deff x \mapsto \frac1\tau\cdot\phi(\tau\cdot x).
\]
For activations with a bounded range such as the sigmoid, the normalization $\frac1\tau$ is omitted. We use the shorthand $\phi \equiv \phi_1$. Given data $\bc{(\vx_i,y_i)}_{i=1}^n$, the graduated objective function is defined as
\[
\cL_\tau(\vw) \deff \frac1n\sum_{i=1}^n(y^\tau_i - \phi_\tau(\vx_i^\top\vw))^2,
\]
where the graduated label $y^\tau_i$ is defined as
\[
y^\tau_i \deff \phi_\tau(\phi^{-1}(y_i)).
\]
We will similarly use the shorthands $y^1_i \equiv y_i$ and $\cL \equiv \cL_1$.
\end{definition}
The sigmoid and (leaky) softplus activation functions are invertible so that the operation $\phi^{-1}$ is well-defined\footnote{The SiLU activation has a one-to-many inverse in the negative half of the real line. However, due to the smoothness of $\phi_\tau$, especially for small temperature values $\tau$, taking either value and applying $\phi_\tau$ to it was found to work well.}. Figure~\ref{fig:surfaces} (top row) shows the graduated versions of various activation functions such as sigmoid, (leaky) softplus and SiLU. Note that the loss surfaces of the original objectives i.e. $\cL \equiv \cL_1$ are always non-convex (see Figure~\ref{fig:surfaces} middle row) with near-stationary points far away from the optimum. There does exist a small basin of convexity but very close to the optimum. However, Figure~\ref{fig:surfaces} (bottom row) shows that once graduated sufficiently aggressively (i.e. using a small enough value of temperature $\tau$), this basin of local strong convexity extends to a much larger region of space. This motivates the following notion of extendable local strong convexity and strong smoothness that is key to the \alg approach. The following discussion considers the noiseless setting i.e. $y_i = \phi(\vx_i^\top\vwo)$ to present the core ideas without worrying about noise. Section~\ref{sec:analysis} will discuss the noisy setting.
\begin{definition}[Extendable Local Strong Convexity, Smoothness]
\label{defn:elsc-elss}
Consider a training objective $\cL$ as in \eqref{eq:obj-main-restated} constructed using data $\bc{(\vx_i,y_i)}_{i=1}^n$ with noiseless labels i.e. $y_i = \phi(\vx_i^\top\vwo)$ corresponding to a gold model $\vwo$. Then $\cL$ is said to posses extendable local strong convexity ELSC (resp. extendable local strong smoothness ELSS) if for any $r > 0$, there exists some temperature $\tau_r > 0$ and some constant $\lambda_r > 0$ (resp. $\Lambda_r > 0$) such that graduated objective $\cL_{\tau_r}$ is $\lambda_r$-strongly convex (resp. $\Lambda_r$-strongly smooth) within the ball $\cB(\vwo, r)$ of radius $r$ centered at the gold model $\vwo$. Specifically, we wish that for any $\vu, \vv \in \cB(\vwo,r)$ we must have
\begin{align*}
	\cL_{\tau_r}(\vv) &\geq \cL_{\tau_r}(\vu) + \ip{\nabla\cL_{\tau_r}(\vu)}{\vv-\vu} + \frac{\lambda_r}2\norm{\vu-\vv}_2^2\\
	\cL_{\tau_r}(\vv) &\leq \cL_{\tau_r}(\vu) + \ip{\nabla\cL_{\tau_r}(\vu)}{\vv-\vu} + \frac{\Lambda_r}2\norm{\vu-\vv}_2^2
\end{align*}
\end{definition}
Definition~\ref{defn:elsc-elss} allows the local basin of strong convexity around $\vwo$ to be expanded using a graduated objective. It now becomes important to assure ourselves that these nice properties can indeed be expected in realistic scenarios. Lemma~\ref{lem:elsc-elss} and Table~\ref{tab:elsc-elss} assure us that appropriate local strong convexity/smoothness constants do exist for several activation functions when data covariates are sampled i.i.d. from the standard normal distribution i.e. $\vx_i \sim \cN(\vzero,I_d)$ where $I_d$ is the $d \times d$ identity matrix.
\begin{lemma}[ELSC/ELSS Precursor]
\label{lem:elsc-elss}
Suppose $\norm\vwo_2 \leq R$ and data covariates are sampled i.i.d. from a $d$-dimensional normal distribution $\vx_i \sim \cN(\vzero, I_d)$. If the training set has $n \geq \Om{d\log d}$ data points, then with probability at least $1 - \exp\br{-\Om n}$, for all values of $r \in [0,2R]$ and all values of the temperature parameter $\tau \in [0,1]$, the objective $\cL_\tau$ satisfies the local strong convexity/smoothness properties in the ball $\cB(\vwo, r)$ with the constants $\lambda_r,\Lambda_r$ as given in Table~\ref{tab:elsc-elss}.
\end{lemma}
Section~\ref{sec:method} uses Lemma~\ref{lem:elsc-elss} and gives a way to choose $\tau_r$ for a given $r > 0$ so that the resulting constants $\lambda_r,\Lambda_r$ from Table~\ref{tab:elsc-elss} are always strictly positive thus fulfilling the ELSC/ELSS conditions. Similar results can be obtained for data covariates sampled from non-standard Gaussians or general subGaussian distributions by using standard techniques \cite{BhatiaJK2015,MukhotyGJK2019} but with less sharp constants.\\

\noindent\textbf{Full Proofs and Code.} Due to lack of space, all detailed calculations and proofs have been provided in the \suppurlcite. Code for \alg is available at the following \codecite.%

\begin{algorithm}[t]
	\begin{algorithmic}[1]
		{
		\REQUIRE Training data $\bc{(\vx_i,y_i}_{i=1}^n$, activation function $\phi$, initial temperature $\tau_0$, temperature increments $\beta_t > 1$, step lengths $\eta_t$
		\ENSURE An estimate $\hvw$ of the gold model $\vwo$
		\STATE Initialize $\vw_0$ and set $t \leftarrow 0$ \COMMENT{For example $\vw_0 \leftarrow \vzero$}
		\FOR{$t = 0, 1, 2, \ldots, T-1$}
			\STATE Obtain graduated labels $y^{\tau_t}_i = \phi_{\tau_t}(\phi^{-1}(y_i))$
			\STATE Construct the graduated objective function
			\[
			\cL_{\tau_t}(\vw) = \frac1n\sum_{i=1}^n(y^{\tau_t}_i - \phi_{\tau_t}(\vx_i^\top\vw))^2
			\]
			\STATE Let $\vw_{t+1} \leftarrow \vw_t - \eta_t\cdot\left.\nabla_{\vw}\cL_{\tau_t}(\vw)\right|_{\vw_t}$ \COMMENT{Alternatives are mini-batch SGD, Adam, variance reduction etc}
			\STATE $\tau_{t+1} \leftarrow \min\bc{\beta_t\cdot\tau_t, 1}$ \COMMENT{Cap temperature at unity}
			\STATE $t \leftarrow t + 1$
		\ENDFOR
		\STATE \textbf{return} {$\vw^T$}
		}
	\end{algorithmic}
	\caption{\alg Pseudocode}
	\label{algo:main}
\end{algorithm}

\section{\alg: Method Description}
\label{sec:method}
Using Lemma~\ref{lem:elsc-elss}, we first instantiate the ELSC/ELSS property from Definition~\ref{defn:elsc-elss} and then derive the \alg algorithm. We take sigmoid activation as an example to present the arguments, but similar arguments hold for other activation functions too.

\begin{lemma}[ELSC/ELSS Guarantee -- Sigmoid Activation]
\label{lem:sigmoid-elsc}
Suppose we are working with the sigmoid activation function. Then for any neighborhood radius $r \in [0,2R]$ (recall that $R \geq \norm{\vwo}_2$ is an upper bound on the norm of the gold model), using any value of temperature $\tau_r \leq \min\bc{\frac{\exp(-R)}{3r},1}$ always ensures the ELSC/ESSS properties with constants $\lambda_r \geq \tau^2\cdot\exp(-\tau R)/2 \geq \tau^2\cdot\exp(-R)/2$ and $\Lambda_r \leq \tau^2\cdot\exp(-\tau R) + \tau^3 R \leq \tau^2(1 + R)$.
\end{lemma}
\begin{proof}
We notice that Table~\ref{tab:elsc-elss} assures (upto universal constants) that for any radius $r$, the $\tau$-graduated objective function $\cL_\tau$ is locally strongly convex and strongly smooth in the ball $\cB(\vwo,r)$ with the parameters $\lambda_r = \tau^2(\exp(-\tau\cdot m_r) - \tau\cdot r)$ and $\Lambda_r = \tau^2(\exp(-\tau\cdot m_r) + \tau\cdot r)$ where $m_r = \norm\vwo_2 + r$. Now clearly $\Lambda_r > 0$ for all $\tau \in [0,1]$ but to ensure $\lambda_r > 0$, we need to set $\tau$ to be small enough. Without loss of generality, we assume $R \geq 1$ and promise to always set $\tau \in [0,1]$. As $\frac v3\exp\br{\frac v3} \leq \frac v2$ for all $v \in [0,1]$, setting $\tau \leq \min\bc{\frac{\exp(-R)}{3r},1}$ ensures $\lambda \geq \tau^2\cdot\exp(-\tau R)/2$. For this value of $\tau$, we also get $\Lambda \leq \tau^2\cdot\exp(-\tau R) + \tau^3 R$ finishing the proof.
\end{proof}
Although the above calculations do not take into account the universal constants omitted in Table~\ref{tab:elsc-elss}, doing so does not change the form of the result. Similar calculations can be done for the other activation functions in Table~\ref{tab:elsc-elss} to arrive at feasible settings for temperature $\tau_r$ given a neighborhood radius $r > 0$.

\noindent\textbf{Summary of \alg.} Standard results from optimization theory \cite{BoydVandenberghe2004,BottouCN2018} tell us that performing descent e.g. GD/SGD on objectives that are strongly convex and smooth offers a linear rate of convergence to the global optimum. The above result allows us to extend the basin of local strong convexity to arbitrarily large regions $r > 0$ simply by using a sufficiently small temperature e.g. $\tau_r \leq \bigO{\frac1r}$ for sigmoid activation. Thus, it is tempting to initialize arbitrarily, set $\tau_0$ to an extremely small value so that the initial model falls in the basin of local strong convexity, and proceed with a descent strategy of choice. However, this method may offer very slow convergence owing to the extremely small strong convexity parameter $\lambda_r = \bigO{\tau_r^2}$. The following discussion shows how we can bootstrap into using larger and larger values of $\tau$ to get more rapid convergence.

\noindent\textbf{Bootstrapping to larger temperatures.} Given data $\bc{(\vx_i,y_i}_{i=1}^n$, we initialize our model at any location $\vw_0$ such that $\norm{\vw_0}_2 \leq R$ (recall that $R \geq \norm\vwo_2$) or else initialize at the origin i.e. $\vw_0 = \vzero$. In either case we are assured that $\norm{\vw_0 - \vwo}_2 \leq 2R$. Since $\vw_0$ is likely to be too far away from $\vwo$ to reside in the local basin of strong convexity, we set the initial temperature to be low $\tau_0 \leq \exp(-R)/3R$ so that ELSC can be obtained via Lemma~\ref{lem:sigmoid-elsc} and $\vw_0$ is assured to fall in the basin of strong convexity surrounding $\vwo$. We now perform a few descent steps using (stochastic) gradient descent after which standard results \cite{BoydVandenberghe2004,BottouCN2018} ensure that we would have obtained a model, say $\vw_1$, such that $\norm{\vw_1 - \vwo}_2 \leq \frac1\beta\cdot\norm{\vw_0 - \vwo}_2$ for some $\beta > 1$. This implies that we have slipped closer to $\vwo$ and can afford a larger temperature say upto $\tau_1 \leq \beta\cdot\tau_0$. This allows a larger strong convexity parameter and hence faster convergence. Continuing this way allows us converge to $\vwo$ at a linear rate as established in Section~\ref{sec:analysis}. The resulting algorithm is described in Algorithm~\ref{algo:main}.

\noindent\textbf{Hyperparameters.} The initial temperature $\tau^0$, the rate $\beta$ at which we elevate the temperature, as well as the step length parameter $\eta$ for the descent technique used (GD/SGD) are the three hyperparamters for \alg and are tuned using standard cross validation and a grid search. Figure~\ref{fig:sensitivity} establishes that \alg is not sensitive misspecification of these hyperparameters.

\noindent\textbf{\alg Variants.} Depending on the algorithm used to perform descent, \alg can be adapted to various settings. For instance, full-batch gradient descent could be used to offer stable convergence at a linear rate \cite{BoydVandenberghe2004} but may be algorithmically more expensive. However, as \alg maintains the invariant that the current model always lies in the basin of local strong convexity surrounding the gold model, other options e.g. stochastic gradient descent using a single data point per iteration or else minibatch versions could be used. Standard results \cite{BottouCN2018} assure a linear rate of convergence in expectation as well as in probability for this option as well. This offers \alg an advantage over existing works such as \cite{HazanLS-S2015} which necessarily require large mini-batch sizes failing which they are unable to assure convergence. Variance-reduced versions using techniques such as SVRG \cite{JohnsonZhang2013} or SAGA \cite{DefazioBL-J2014}, or momentum and adaptive gradient methods \cite{KingmaBa2015,ReddiZSKK2018} could also be used and offer interesting directions for future work.

\section{Theoretical Guarantees}
\label{sec:analysis}

All detailed proofs are placed in the \suppcite. As before, we present arguments with the sigmoid activation as an example which similarly extend to the other activation functions as well. We will first establish the convergence behavior for the noiseless case i.e. when $y_i = \phi(\vx_i^\top\vwo)$ and then extend the analysis to the noisy case under two settings: pre- and post-activation noise.

\subsection{Noiseless Labels}
\label{sec:idealized-conv}

Theorem~\ref{thm:conv-main-sigmoid} establishes the convergence rate of \alg with explicit constants for sigmoid activation when used with GD steps.

\begin{theorem}
\label{thm:conv-main-sigmoid}
Suppose Algorithm~\ref{algo:main} is executed with GD for $T$ steps with step length at time $t \leq T$ set to any value $\eta_t \leq \frac1{\Lambda_{r_t}}$ and the temperature increment parameter set to any value $\beta_t \leq \exp\br{\frac{\lambda_{r_t}}{\Lambda_{r_t}}}$ where $\lambda_{r_t}, \Lambda_{r_t}$ are the ELSC/ELSS parameter corresponding to the $\tau_t$-graduated objective, then \alg offers a linear rate of convergence  with $\norm{\vw_T - \vwo}_2^2 \leq \prod_{t=0}^T\exp\br{-\frac{\lambda_{r_t}}{\Lambda_{r_t}}}\cdot\norm{\vw_0-\vwo}_2^2$. In particular, for the sigmoid activation, we may use any step length $\eta_t \leq \frac1{\tau_t^2(1+R)}$ and a constant temperature increment rate $\beta_t \equiv \beta \leq \exp\br{\frac{\exp(-R)}{1+R}}$ (note that this still allows values of $\beta > 1$). This offers $\norm{\vw_T - \vwo}_2^2 \leq \exp(-\zeta\cdot T)\cdot\norm{\vw_0-\vwo}_2^2$ where $\zeta = \ln\beta = \frac{\exp(-R)}{1+R} > 0$.
\end{theorem}
\begin{proof}
Suppose we are given an iterate $t$ with $\norm{\vwt -\vwo}_2 \leq r_t$ and $\tau_t$ was set according to $r_t$ as described in Lemma~\ref{lem:sigmoid-elsc} such that $\cL_{\tau_t}$ is $\lambda_{r_t}$-strongly convex and $\Lambda_{r_t}$-strongly smooth in the region $\cB(\vwo,r_t)$. Note that this is guaranteed in the base case i.e. $t = 0$ by initialization and hyperparameter tuning. Now the ELSS guarantee from Lemma~\ref{lem:sigmoid-elsc} tells us that
\[
\cL_{\tau_t}(\vwn) - \cL_{\tau_t}(\vwt) \leq \ip{\nabla_\vw\cL_{\tau_t}(\vwt)}{\vwn-\vwt} + \frac{\Lambda_{r_t}}2\norm{\vwn-\vwt}_2^2
\]
Using the fact that $\vwn = \vwt - \eta_t\cdot\nabla_\vw\cL_{\tau_t}(\vwt)$ and we set $\eta_t = \frac1{\Lambda_{r_t}}$ gives us upon some manipulations,
\[
	\cL_{\tau_t}(\vwn) - \cL_{\tau_t}(\vwt) \leq \ip{\nabla_\vw\cL_{\tau_t}(\vwt)}{\vwo-\vwt} +\frac{\Lambda_{r_t}}2\br{\norm{\vwt-\vwo}_2^2-\norm{\vwn-\vwo}_2^2}
\]
On the other hand, the ELSC guarantee from Lemma~\ref{lem:sigmoid-elsc} tells us
\[
\cL_{\tau_t}(\vwo) - \cL_{\tau_t}(\vwt) \geq \ip{\nabla_\vw\cL_{\tau_t}(\vwt)}{\vwo-\vwt} + \frac{\lambda_{r_t}}2\norm{\vwo-\vwt}_2^2
\]
Combining and using the fact that $0 = \cL_{\tau_t}(\vwo) \leq \cL_{\tau_t}(\vwn)$ (in the noiseless setting $\cL_{\tau}(\vwo) = 0$ for all $\tau \in (0,1]$) gives us
\[
\frac{\Lambda_{r_t}}2\cdot\norm{\vwn-\vwo}_2^2 \leq \frac{\Lambda_{r_t} - \lambda_{r_t}}2\cdot\norm{\vwt-\vwo}_2^2,
\]
which can be rewritten using $1 - x \leq \exp(-x)$ as
\[
\norm{\vwn-\vwo}_2^2 \leq \exp\br{-\frac{\lambda_{r_t}}{\Lambda_{r_t}}}\cdot\norm{\vwn-\vwo}_2^2,
\]
that gives us a linear rate of convergence and also allows us to set $\beta_t = \exp\br{\frac{\lambda_{r_t}}{\Lambda_{r_t}}}$ as the temperature increment rate. Note that this always allows values of $\beta_t > 1$ to be set since $\exp(v) \geq 1 + v$ and thus $\exp\br{\frac{\lambda_{r_t}}{\Lambda_{r_t}}} \geq 1 + {\frac{\lambda_{r_t}}{\Lambda_{r_t}}} > 1$. Although the above proof uses the precise step length setting $\eta_t = \frac1{\Lambda_{r_t}}$, the proof readily extends \cite{BoydVandenberghe2004} to cases where we instead set $\eta_t \leq \frac1{\Lambda_{r_t}}$ but the arguments get a bit more involved. For the special case of the sigmoid activation, using the ELSC/ELSS constants from Lemma~\ref{tab:elsc-elss} gives us the desired parameters $\eta_t \leq \frac1{\tau^2_t(1+R)}$ and $\beta \leq \exp\br{\frac{\exp(-R)}{1+R}}$. Since $R$ is typically a small constant, the above result offers a good rate of linear convergence in practice (see Section~\ref{sec:exps}).
\end{proof}

Using standard proof techniques for instance \cite[Theorem 4.6]{BottouCN2018}, we may also derive a similar linear convergence guarantee for variants of \alg that use (mini-batch) stochastic gradient steps instead. We provide the details in the \suppciteshort.

\subsection{Noisy Labels}
\label{sec:noisy-conv}

There exist two distinct regimes of label noise namely \emph{pre}-activation noise, where we observe $y_i = \phi(\vx_i^\top\vwo + \epsilon_i)$ and \emph{post}-activation noise where we observe $y_i = \phi(\vx_i^\top\vwo) + \epsilon_i$ where $\epsilon_i$ is the label noise. These are often referred to as corruption of the canonical parameter and corruption of the response, respectively \cite{YangTR2013}.

\noindent\textbf{Pre-activation Noise.} For pre-activation noise, our analysis can handle arbitrary sub-Gaussian noise distributions. Note that sub-Gaussian distributions are a fairly large family of distributions and include all Gaussian distributions as well as all distributions with bounded support. A gentle introduction to sub-Gaussian distributions is presented in the \suppciteshort.

\noindent\textbf{Post-activation Noise.} Addressing post-activation noise takes more care since corruption can lead to illegal labels. For example, even simple Gaussian noise applied post-activation may cause the response $y_i$ to be negative or greater than unity whereas the sigmoid activation has range restricted to the interval $[0,1]$. We avoid this by considering noise distributions with restricted support for post-activation settings. Alternatives could include exploring multiplicative noise models instead of additive models. We note that such restrictions are not required for pre-activation noise.

Establishing the ELSS property is relatively straightforward even in the presence of pre-/post-activation noise but establishing ELSC requires more care otherwise we may not obtain strictly positive values of $\lambda_r$. The \suppciteshort discuss how imposing a non-trivial cap on the temperature of the form $\tau_{\max} = \bigO{\frac1\varsigma}$ ensures that ELSC can still be established where $\varsigma$ depending on the noise distribution is either its standard deviation or width of its support or more-generally, its sub-Gaussian constant. A subsequent analysis then shows that despite capping the temperature early, a model $\hvw$ with bounded error $\norm{\hvw - \vwo}_2 \leq \bigO\varsigma$ can still be obtained. Theorem~\ref{thm:conv-noisy-sigmoid} presents an informal statement of the recovery offered by \alg in noisy setting for sigmoid activation. The detailed theorem statement and proof are presented in the \suppciteshort.

\begin{theorem}[Informal]
\label{thm:conv-noisy-sigmoid}
Suppose label noise is pre-activation and $\varsigma$-sub-Gaussian, or else post-activation with bounded support $[-\varsigma,\varsigma]$. Let Algorithm~\ref{algo:main} be executed as given in the statement of Theorem~\ref{thm:conv-main-sigmoid} except with a temperature cap i.e. the temperature is never allowed to exceed $\tau_{\max} = \bigO{\frac1\varsigma}$ where $\varsigma$ is the variance parameter associated with the distribution. Then in $T = \bigO{\ln\frac1\varsigma}$ steps, \alg offers a model $\vw_T$ with model recovery error $\norm{\vw_T - \vwo}_2 \leq \bigO\varsigma$.
\end{theorem}

\noindent\textbf{Consistent Recovery.} An interesting special case arises if the noise variance parameter $\varsigma$ is small enough so that $\tau_{\max} \geq 1$ i.e. the temperature cap is a vacuous one. In this case, \alg guarantees consistent recovery for post-activation noise i.e. $\norm{\hvw - \vwo}_2 \rightarrow 0$ as the number of data points increases as $n \rightarrow \infty$.

\begin{theorem}[Informal]
\label{thm:consistent-post-noisy-sigmoid}
Suppose label noise is post-activation with bounded support $[-\varsigma,\varsigma]$ with $\varsigma$ sufficiently small so that the temperature cap is vacuous i.e. $\tau_{\max} \geq 1$. Let Algorithm~\ref{algo:main} be executed as given in the statement of Theorem~\ref{thm:conv-main-sigmoid} i.e. with no additional temperature caps. Then for any $\epsilon > 0$, in $T = \bigO{\ln\frac1\epsilon}$ steps, \alg offers a model $\vw_T$ with model recovery error $\norm{\vw_T - \vwo}_2 \leq \bigO{\varsigma\sqrt\frac dn + \epsilon}$.
\end{theorem}

Note that as $n \rightarrow \infty$, \alg offers $\norm{\hvw- \vwo}_2^2 \rightarrow 0$ error i.e. \emph{consistent} recovery. Figure~\ref{fig:consistency} shows that with low-variance noise, \alg does offer consistent recovery under low-variance post-activation noise in experiments too. However, what is curious is that experiments show that this happens for pre-activation noise as well. Theorem~\ref{thm:consistent-post-noisy-sigmoid} is unable to assure consistent recovery for pre-activation noise which seems reasonable as even though the pre-activation noise $\epsilon_i$ may be unbiased, the non-linear activation function causes the label to appear biased since $\Ee{\epsilon}{\sigma(a + \epsilon)} \neq \sigma(a)$. Investigating why \alg still seems to offer consistent recovery in experiments in pre-activation settings is an interesting direction of future work.

\begin{figure*}[t]
	\centering
	\begin{tabular}{@{\hskip 0.02\textwidth}c@{\hskip 0.02\textwidth}c@{\hskip 0.02\textwidth}c@{\hskip 0.02\textwidth}}
		\includegraphics[width=0.3\textwidth]{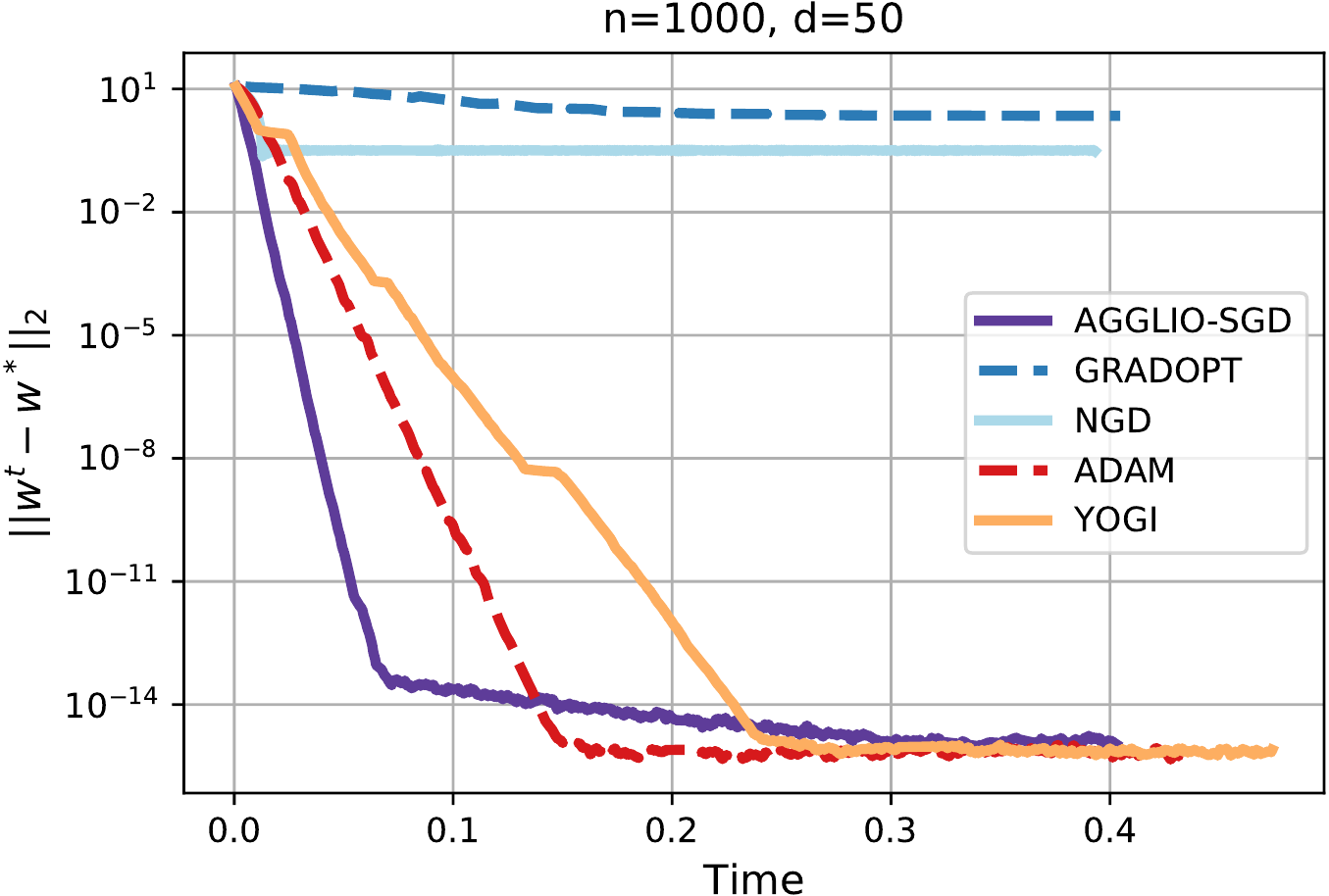} & \includegraphics[width=0.3\textwidth]{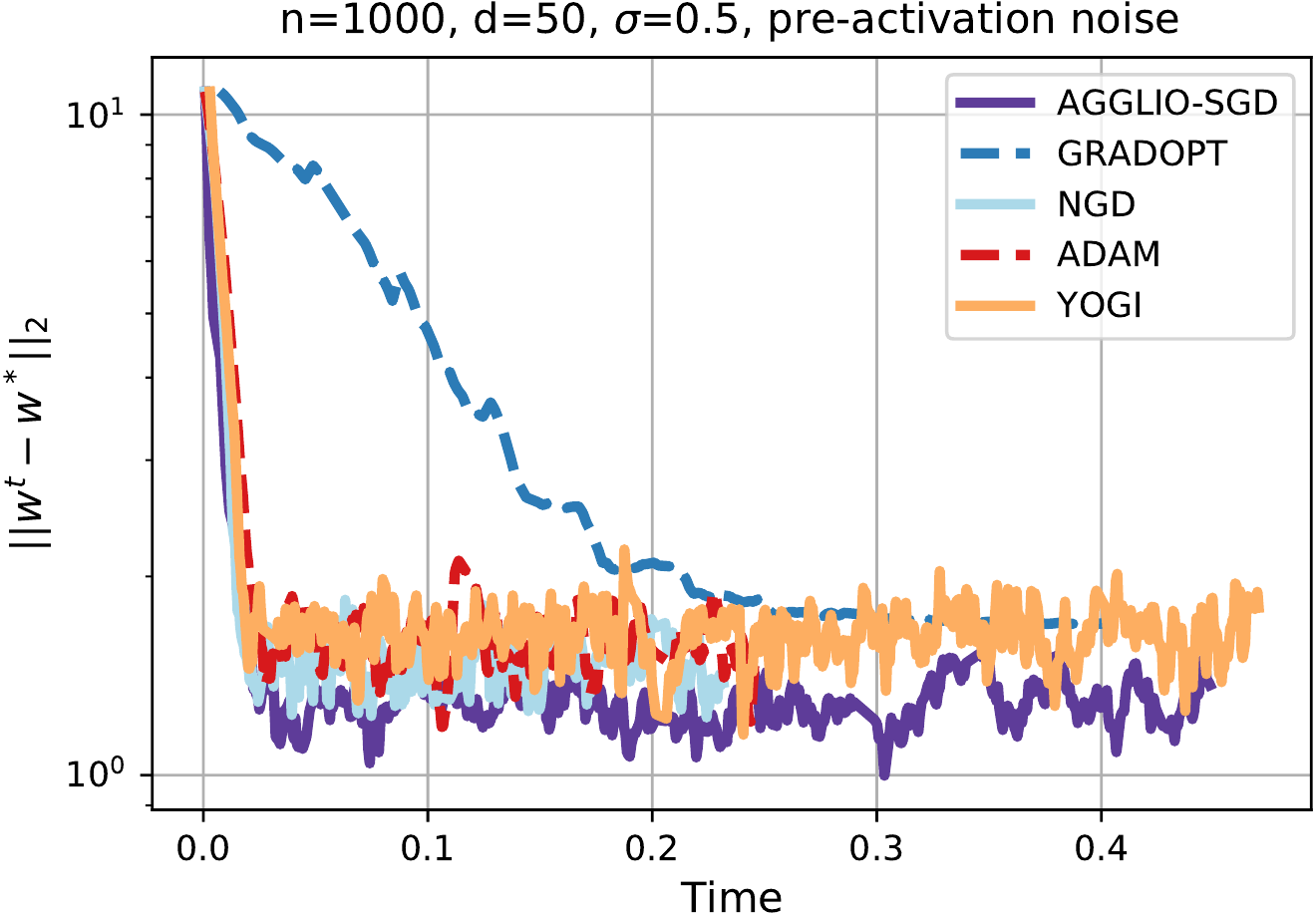} & 
		\includegraphics[width=0.3\textwidth]{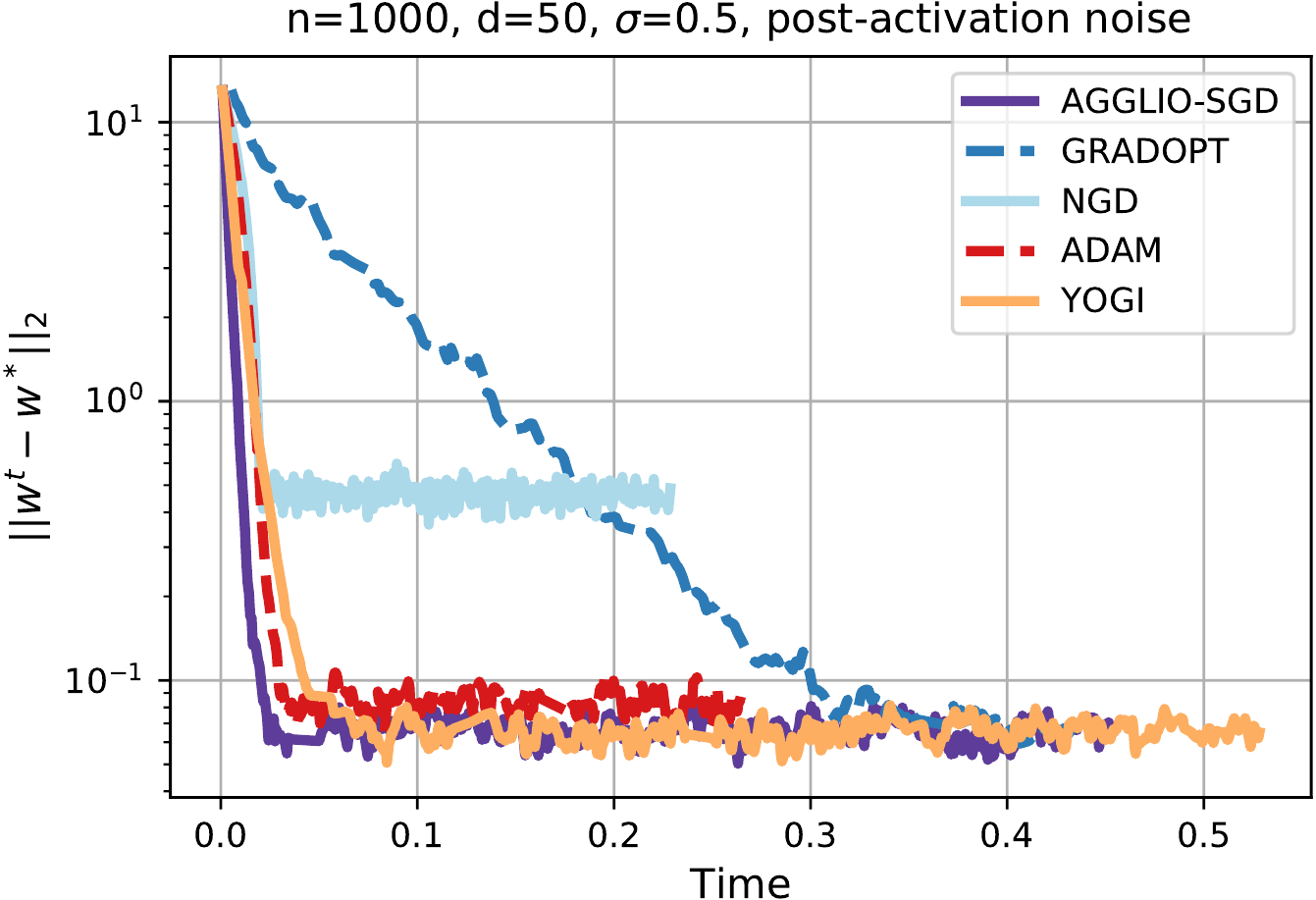}\\
		\small (a) noiseless & \small (b) pre-activation noise & \small (c) post-activation noise
	\end{tabular}
	\caption{Convergence behavior of \alg and various competitors on a learning task with sigmoid activation in (a) noiseless, (b) pre-activation (c) and post-activation noise settings. The number of data points $n$, dimensions $d$, and noise standard deviation $\varsigma$, as applicable, are shown at the top of each figure. Noise added in the figures (b,c) was sampled i.i.d. as $\epsilon \sim \cN(0. 0.5^2)$. All methods were afforded a mini-batch size of $50$. Note that the y-axis values in (a) are in machine precision range around $10^{-16}$. The y-axes of all three plots is in log-scale. Time is measured in seconds.}
	\label{fig:conv}
\end{figure*}

\begin{figure*}[t]
	\centering
	\begin{tabular}{@{\hskip 0.02\textwidth}c@{\hskip 0.02\textwidth}c@{\hskip 0.02\textwidth}c@{\hskip 0.02\textwidth}}
		\includegraphics[width=0.3\textwidth]{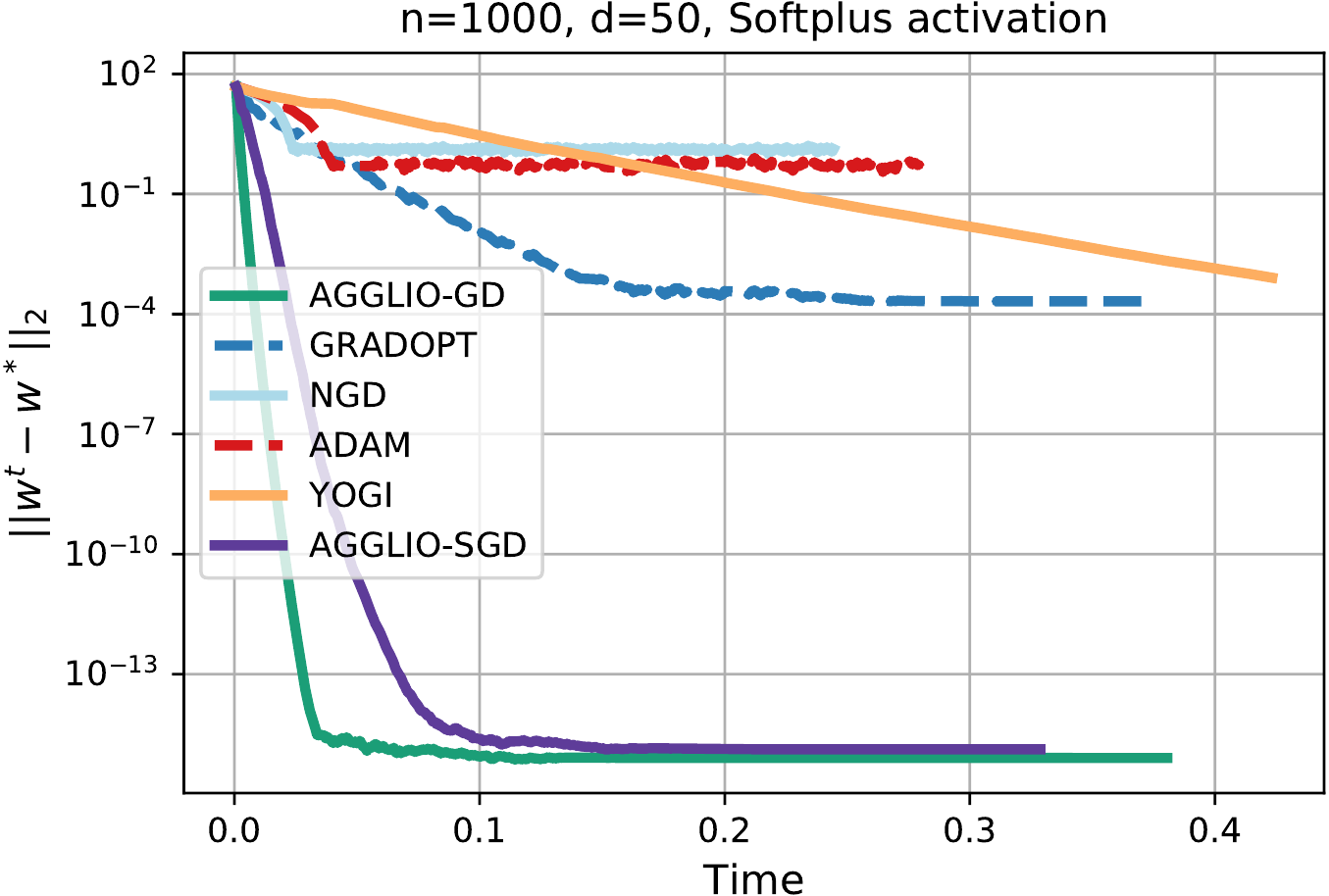} & \includegraphics[width=0.3\textwidth]{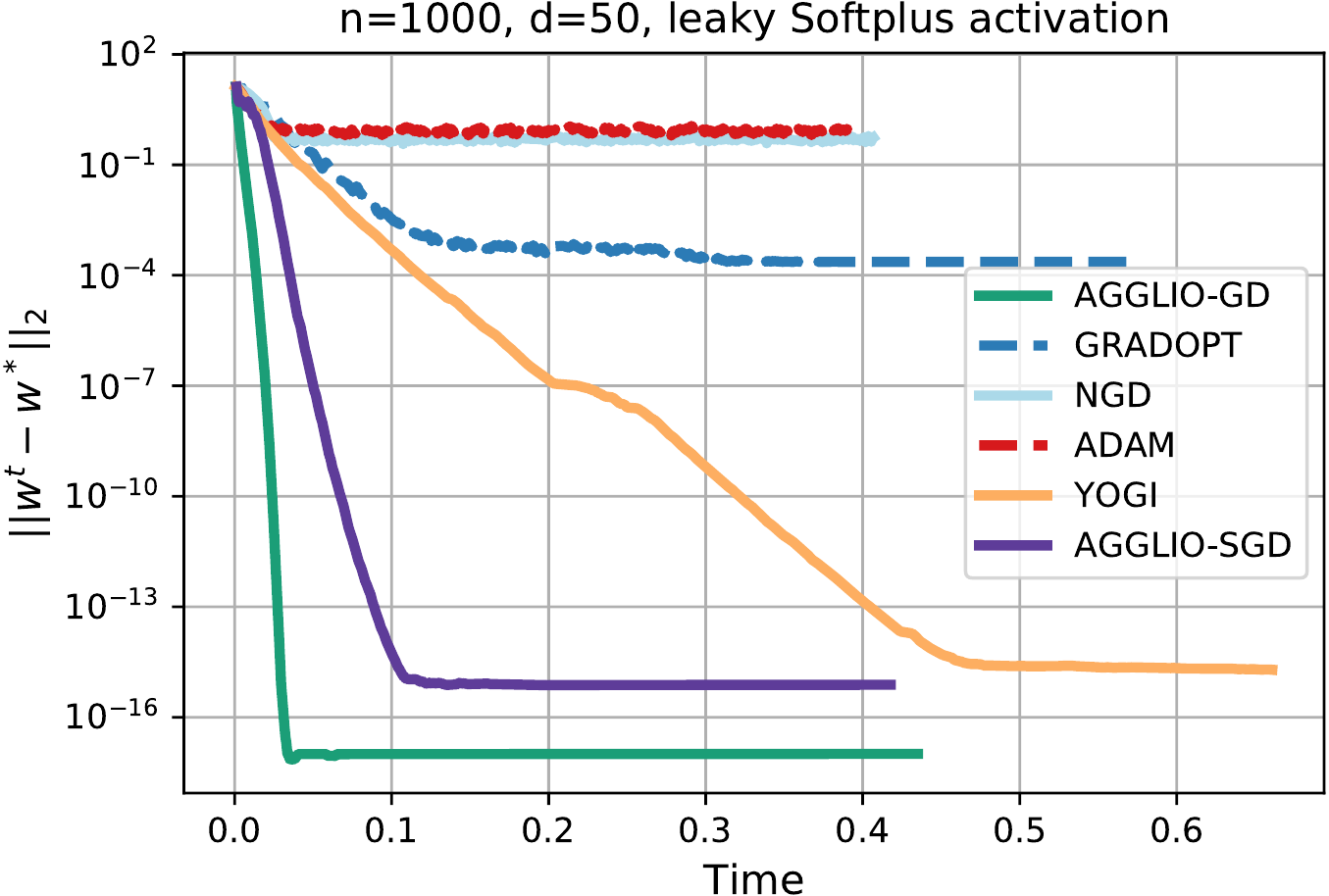} & 
		\includegraphics[width=0.3\textwidth]{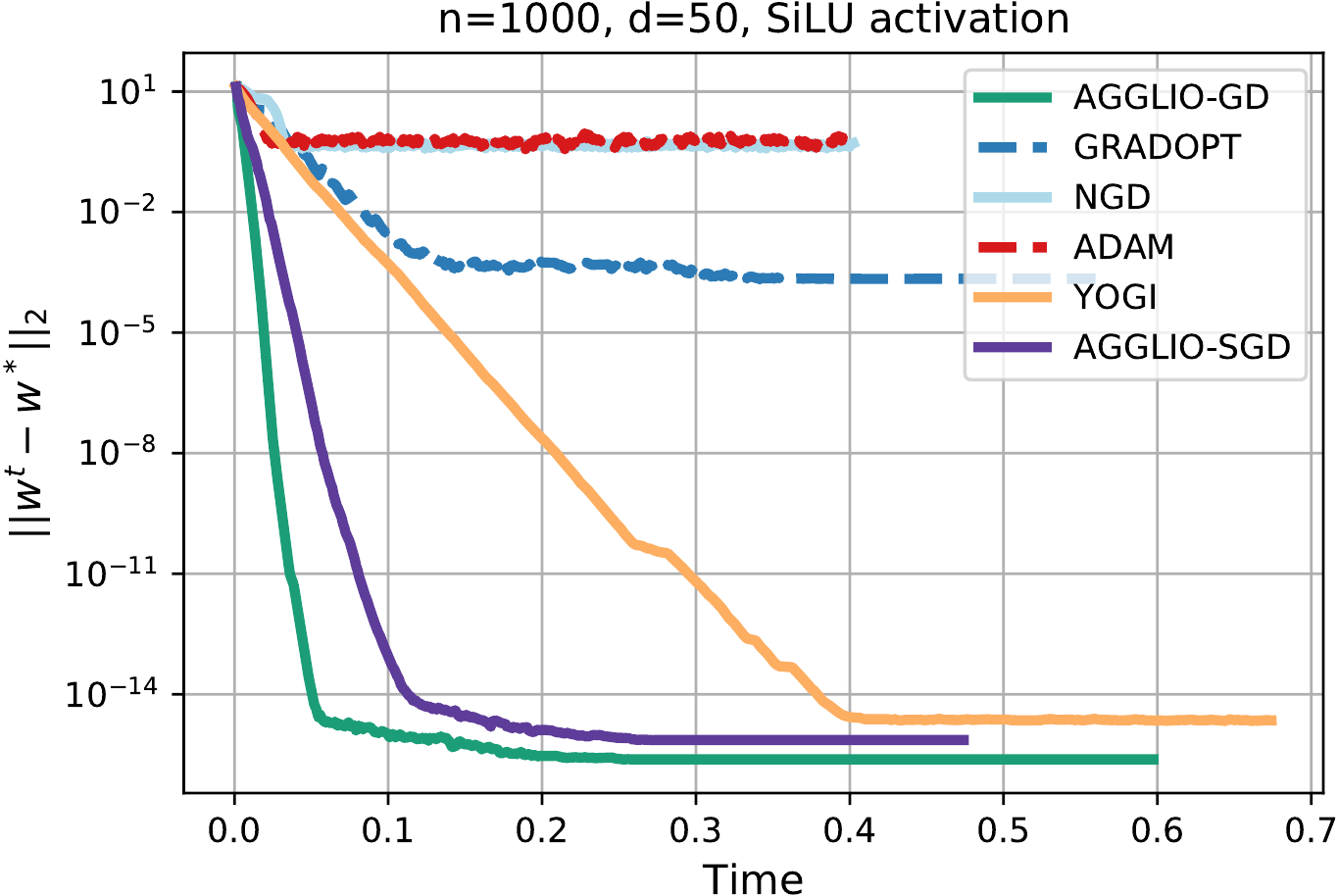}\\
		\small (a) Softplus & \small (b) Leaky Softplus & \small (c) SiLU
	\end{tabular}
	\caption{A study comparing \alg-GD and \alg-SGD with various competitors on different activation functions. \alg-GD outperforms competitor algorithms by several orders of magnitude, as it can compute the full gradient. All competitors and \alg-SGD were offered a mini-batch size of $50$. Time is measured in seconds.}
	\label{fig:act_conv}
\end{figure*}

\begin{figure*}[t]
	\centering
	\begin{tabular}{@{\hskip 0.02\textwidth}c@{\hskip 0.02\textwidth}c@{\hskip 0.02\textwidth}c@{\hskip 0.02\textwidth}}
		\includegraphics[width=0.3\textwidth]{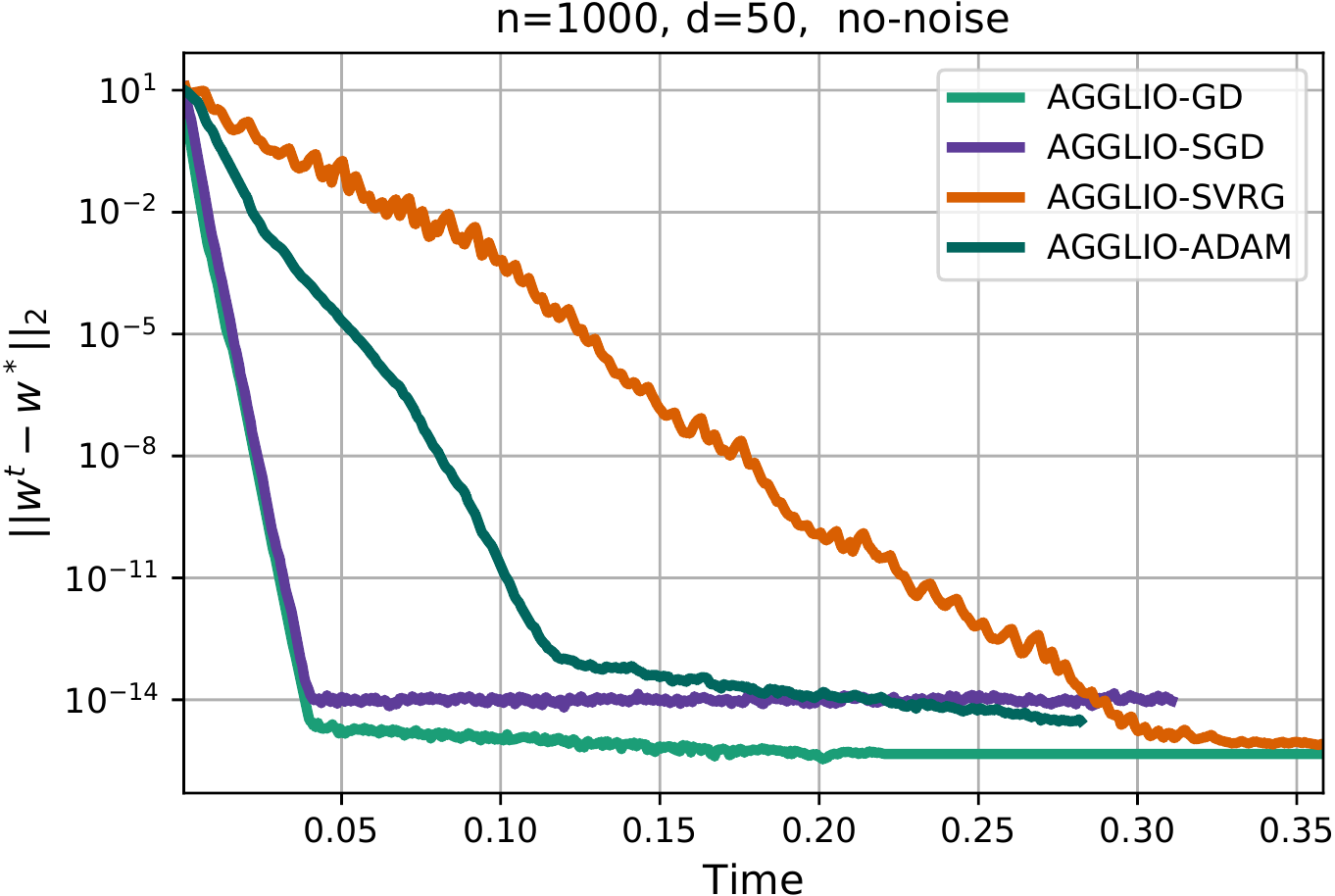} & \includegraphics[width=0.3\textwidth]{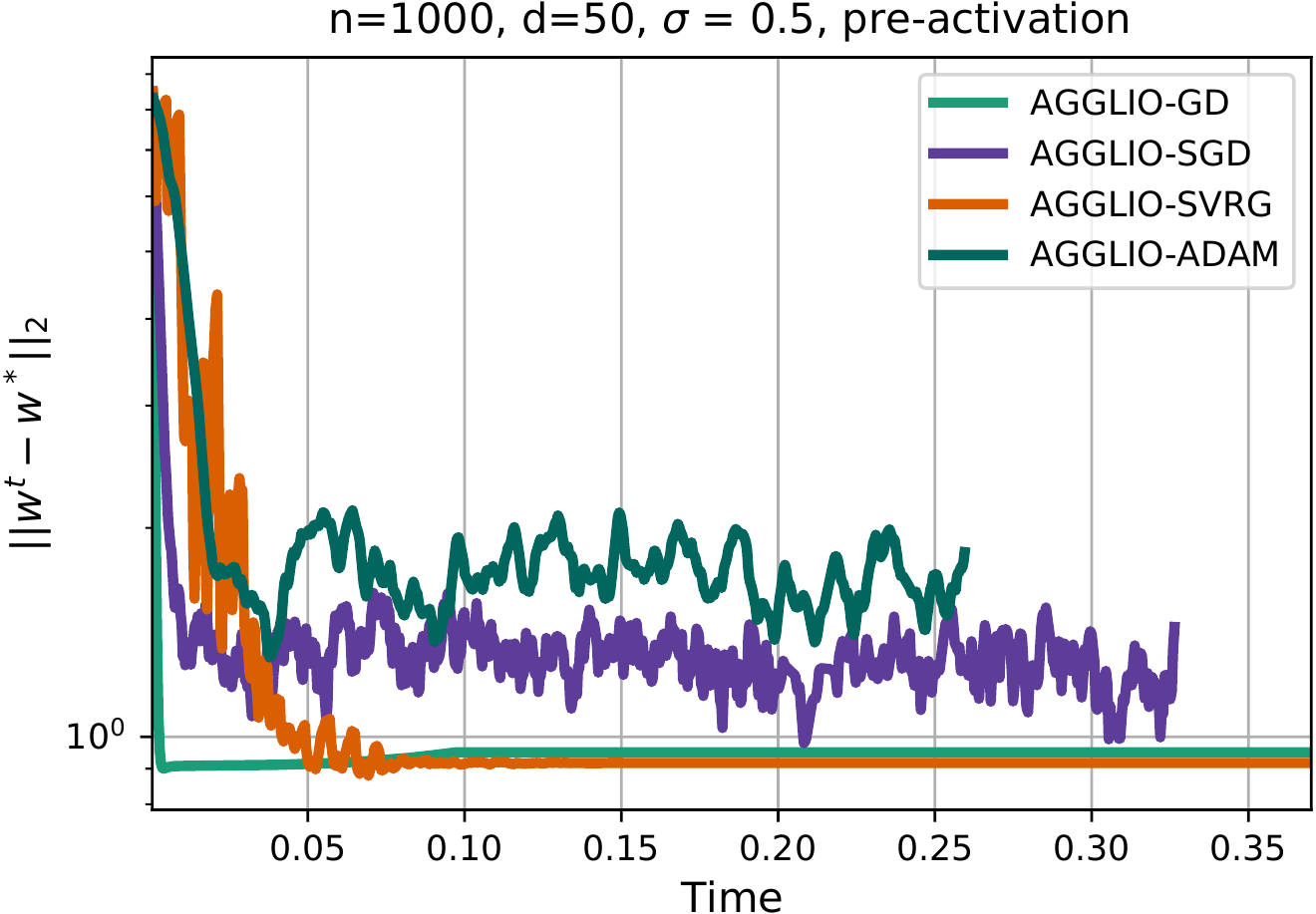} & 
		\includegraphics[width=0.3\textwidth]{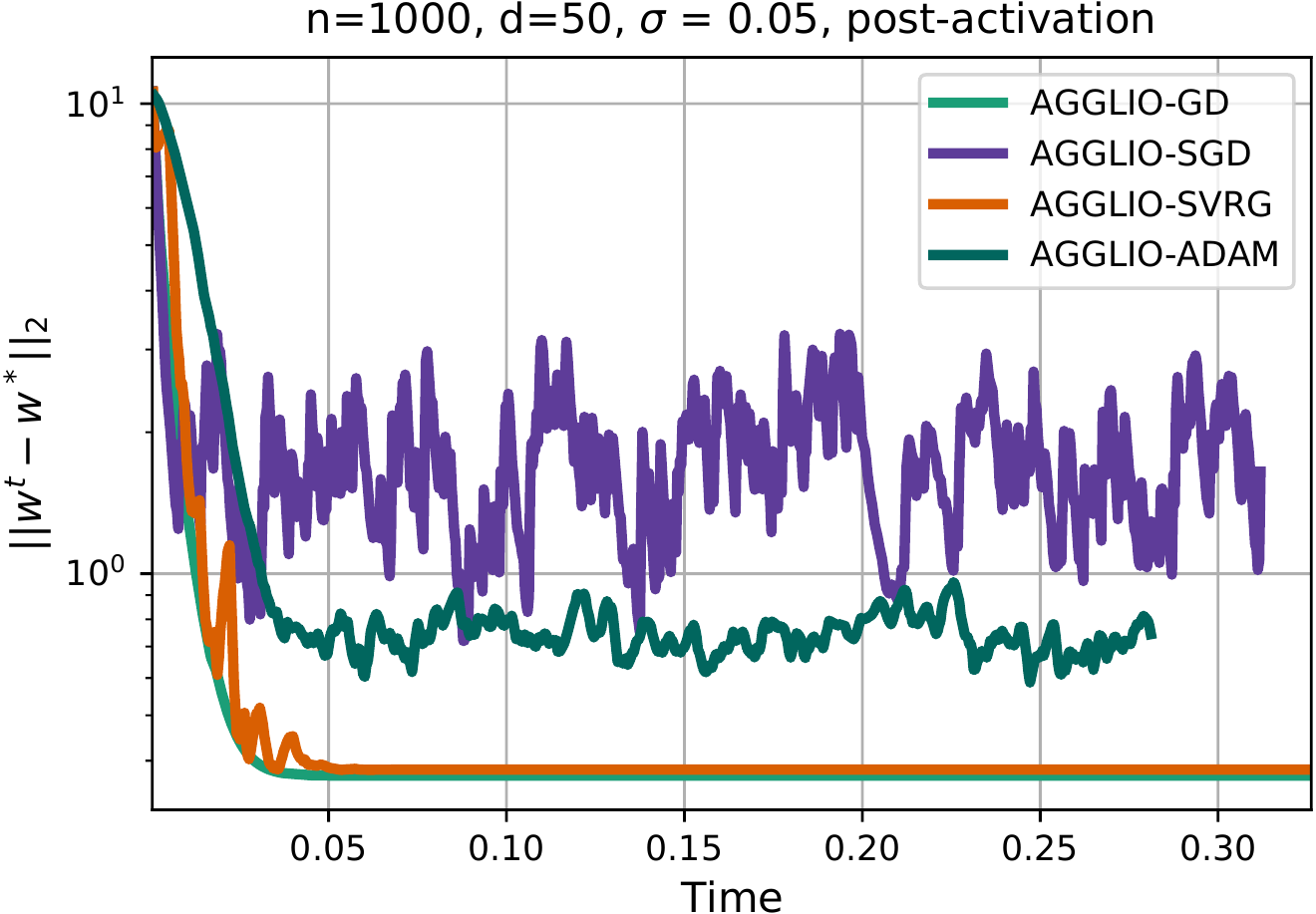}\\
		\small (a) noiseless convergence & \small (b) pre-activation noise & \small (c) post activation noise
	\end{tabular}
	\caption{An ablation study where four \alg variants (GD, SGD, SVRG, Adam) were executed in noiseless, pre-activation and post-activation noise settings with sigmoid activation. All \alg variants except the GD variant were afforded a mini-batch size of $50$. The noiseless setting in figure (a) allows recovery error in the order of machine precision i.e. around $10^{-16}$. The SVRG variant is slower but usually offers final performance at par with the GD variant. Time is measured in seconds.}
	\label{fig:agglio_conv}
\end{figure*}

\begin{figure*}[t]
	\centering
	\begin{tabular}{@{\hskip 0.02\textwidth}c@{\hskip 0.02\textwidth}c@{\hskip 0.02\textwidth}c@{\hskip 0.02\textwidth}}
 		\includegraphics[width=0.3\textwidth]{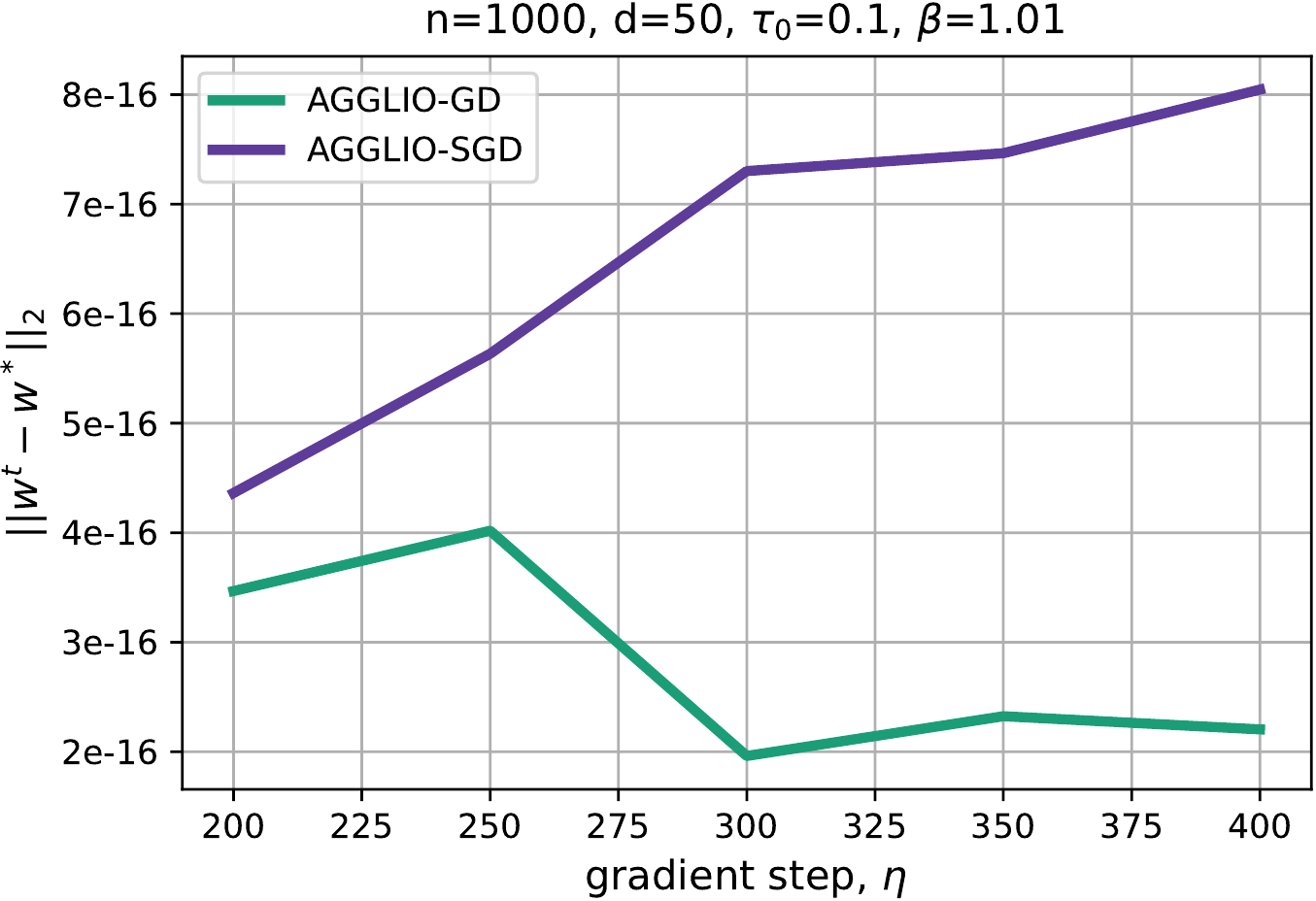} & \includegraphics[width=0.3\textwidth]{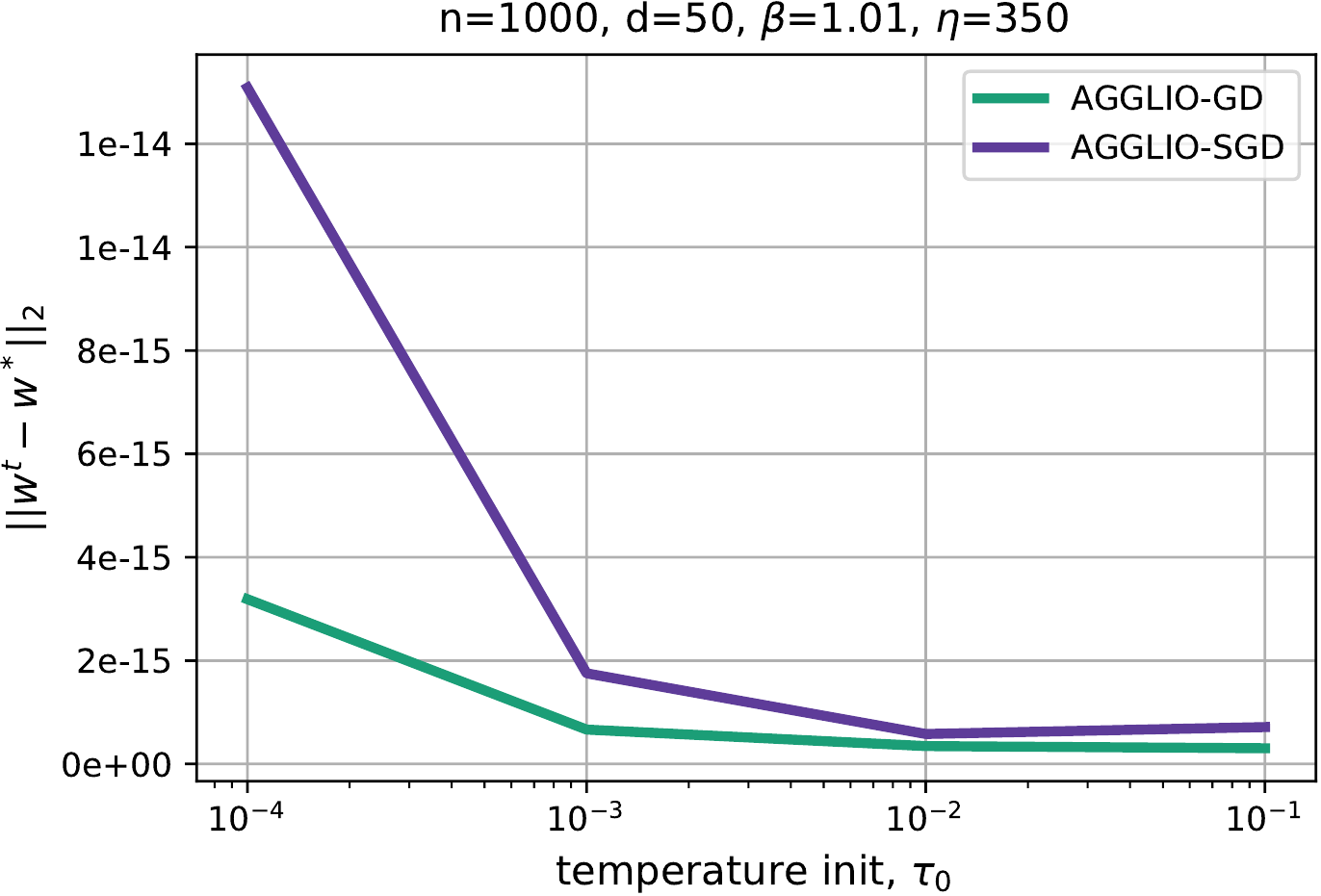} &
		\includegraphics[width=0.3\textwidth]{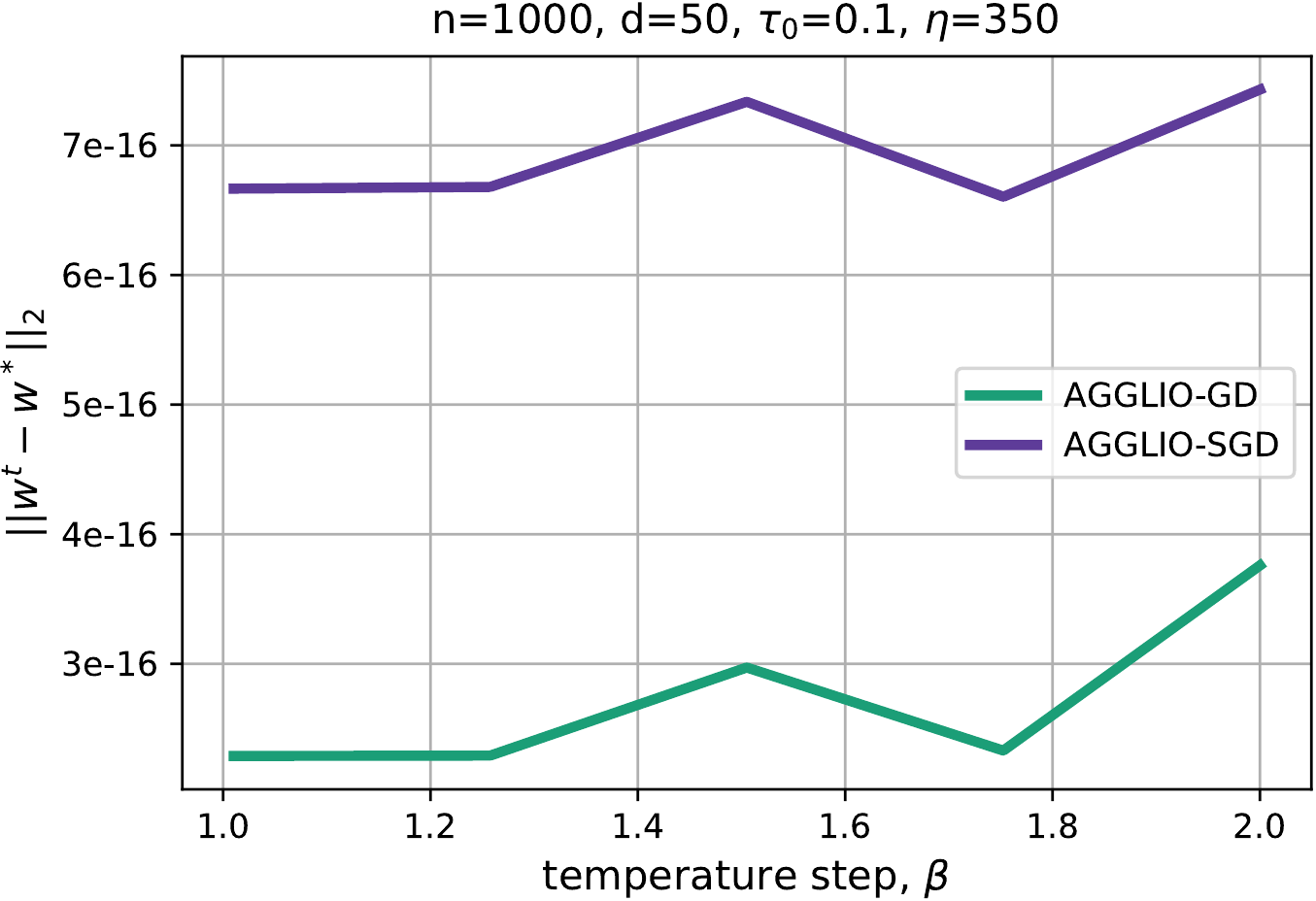} \\
		\small (a) gradient step, $\eta$ & \small (b) initial temperature, $\tau_0$  & \small (c) temperature increment, $\beta$
	\end{tabular}
	\caption{A study investigating the effect of mis-specifying \alg's hyper-parameters in the noiseless setting. Each hyper-parameter was varied in a wide range while keeping the rest fixed to their default values as highlighted in the title. The experiment demonstrates that \alg offers roughly the same model recovery error despite significant mis-specification in its hyper-parameters. Note that the y-axes values are all at machine precision levels around $10^{-16}$ or $10^{-14}$ i.e. the difference in performance of the GD and SGD variants of \alg, although accentuated in these plots, is actually minuscule.}
	\label{fig:sensitivity}
\end{figure*}

\begin{figure*}[t]
	\centering
	\begin{tabular}{@{\hskip 0.01\textwidth}c@{\hskip 0.01\textwidth}c@{\hskip 0.02\textwidth}c@{\hskip 0.01\textwidth}}
 		\includegraphics[width=0.3\textwidth]{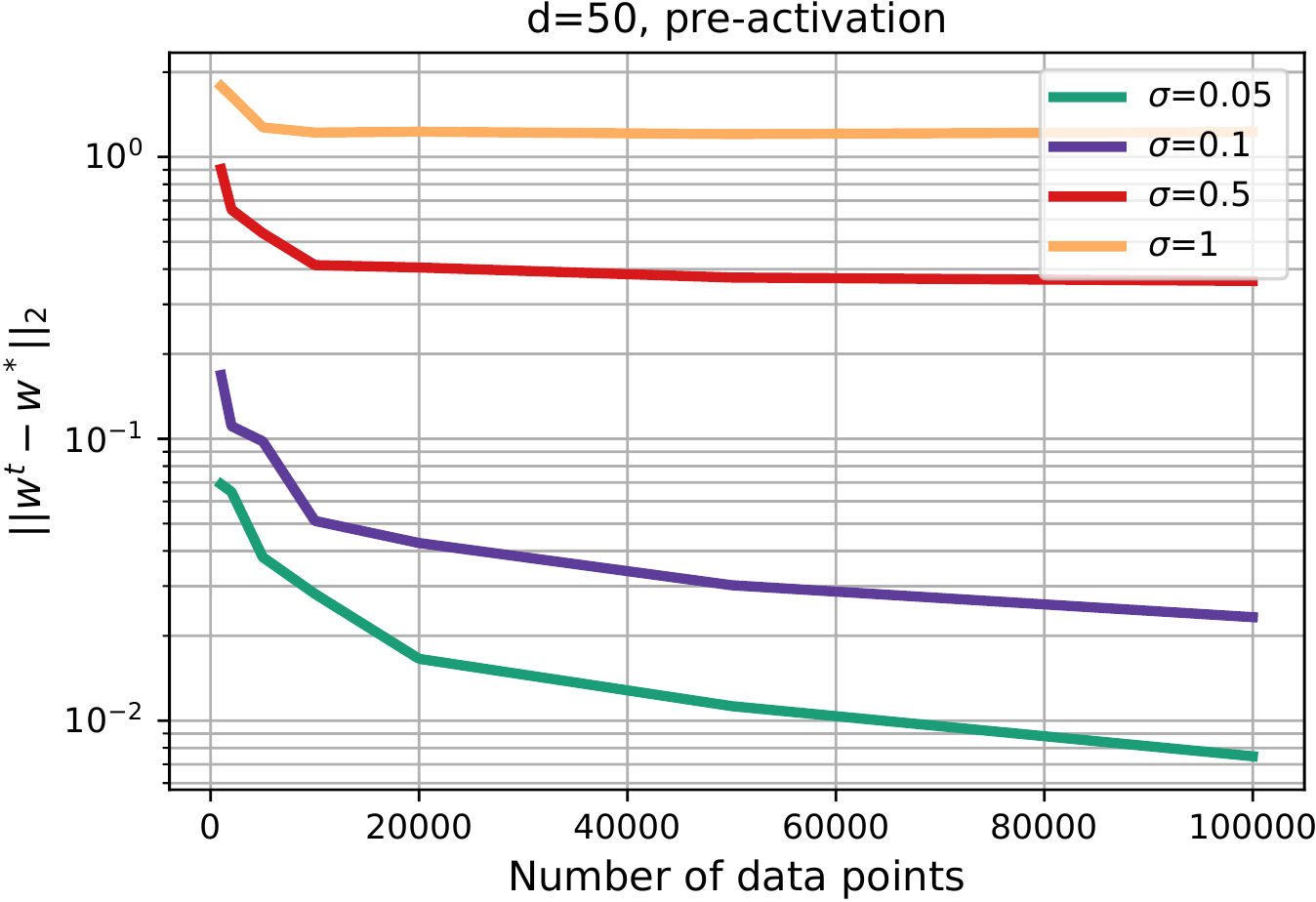} & \includegraphics[width=0.3\textwidth]{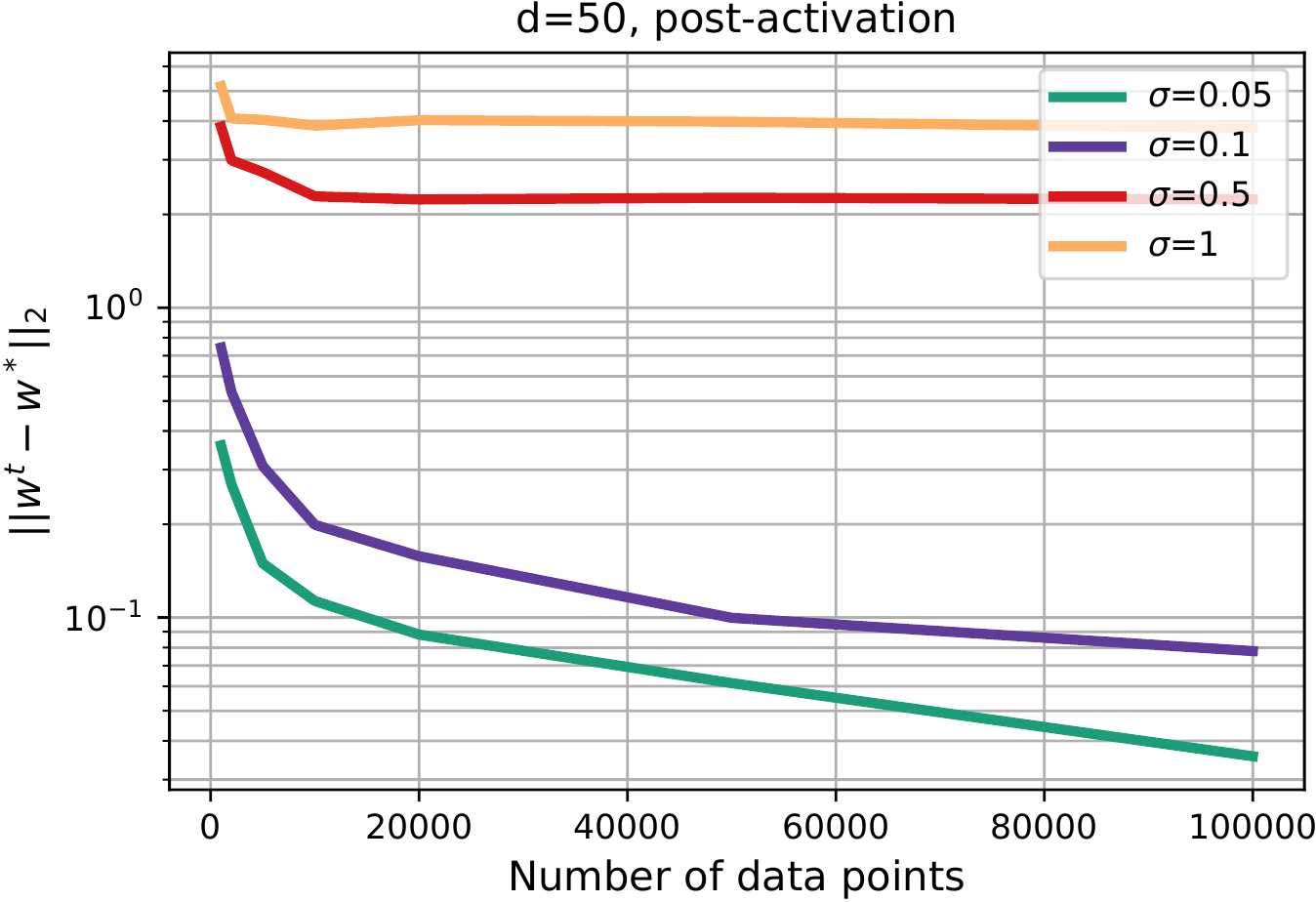} &
		\includegraphics[width=0.3\textwidth]{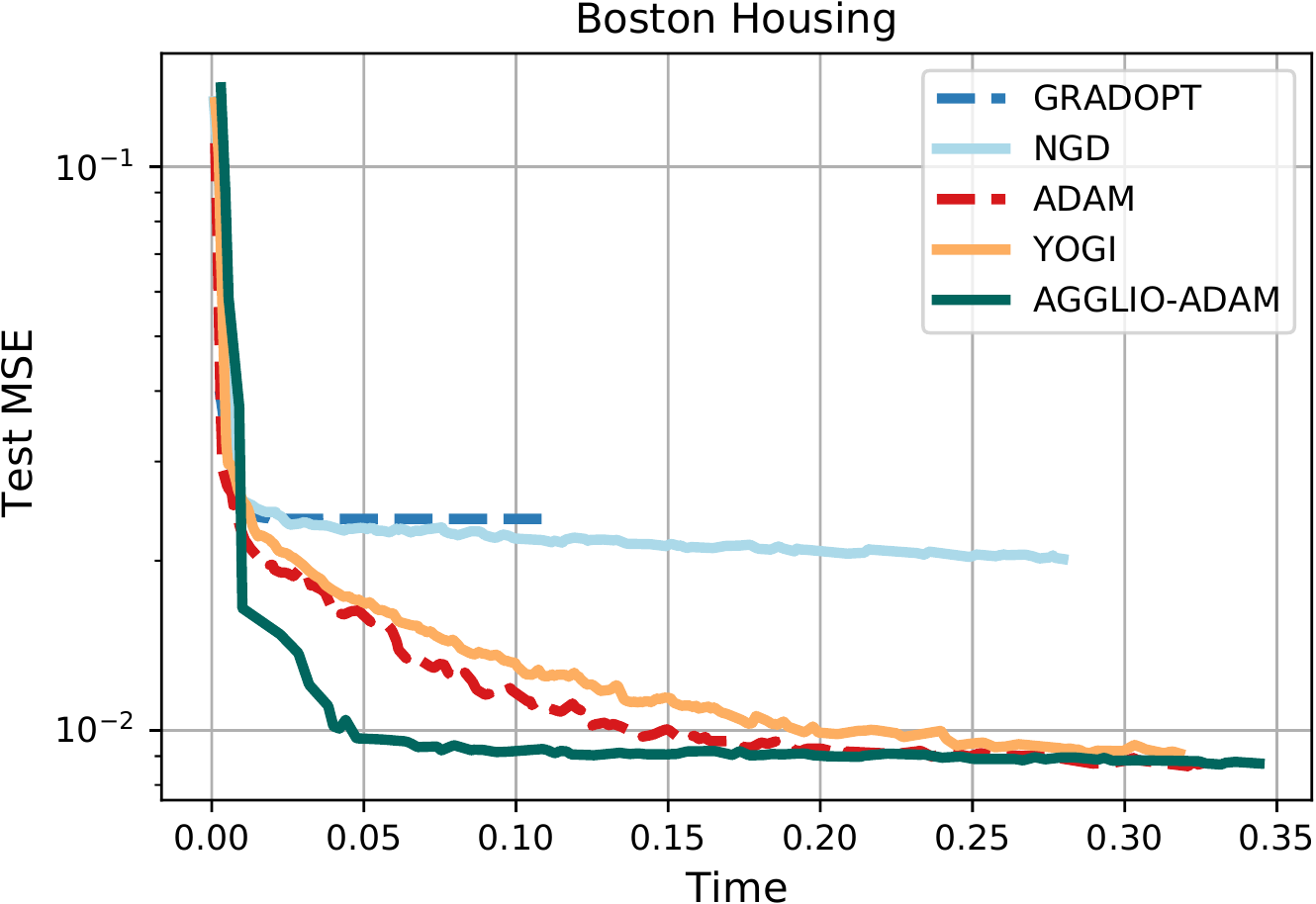} \\
		\small (a) consistency (pre-activation noise) & \small (b) consistency (post-activation noise) & \small (c) Boston Housing dataset
	\end{tabular}
	\caption{Figure (a) and (b) investigate whether \alg offers consistent model recovery in noisy settings as the number of data points is increased. \alg seems to offer consistent recovery for noise with small variance whereas for large variance noise, consistent recovery does not seem to be indicated. Pre-activation noise is seen to offer lower recovery error than post-activation noise. Figure (c) compares \alg in its Adam variant to competitor algorithms on the Boston Housing regression dataset. \alg, Adam and YOGI offer superior test RMSE than NGD, GradOpt. \alg converged the fastest.}
	\label{fig:consistency}
\end{figure*}

\begin{figure}[t]
	\centering
    \begin{tabular}{@{\hskip 0.01\linewidth}c@{\hskip 0.01\textwidth}c@{\hskip 0.01\textwidth}}
	\includegraphics[width=0.3\linewidth]{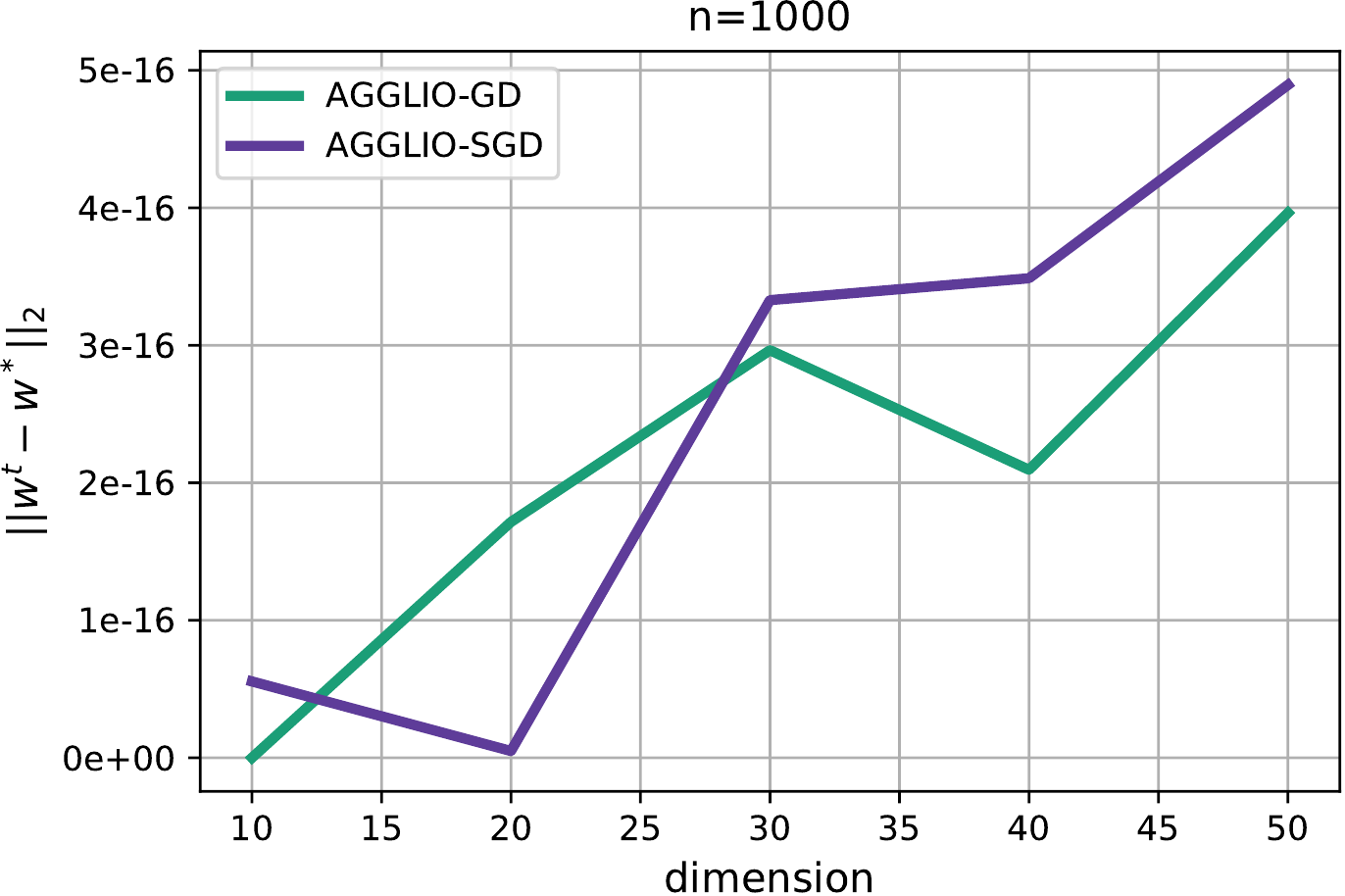}&
	\includegraphics[width=0.3\linewidth]{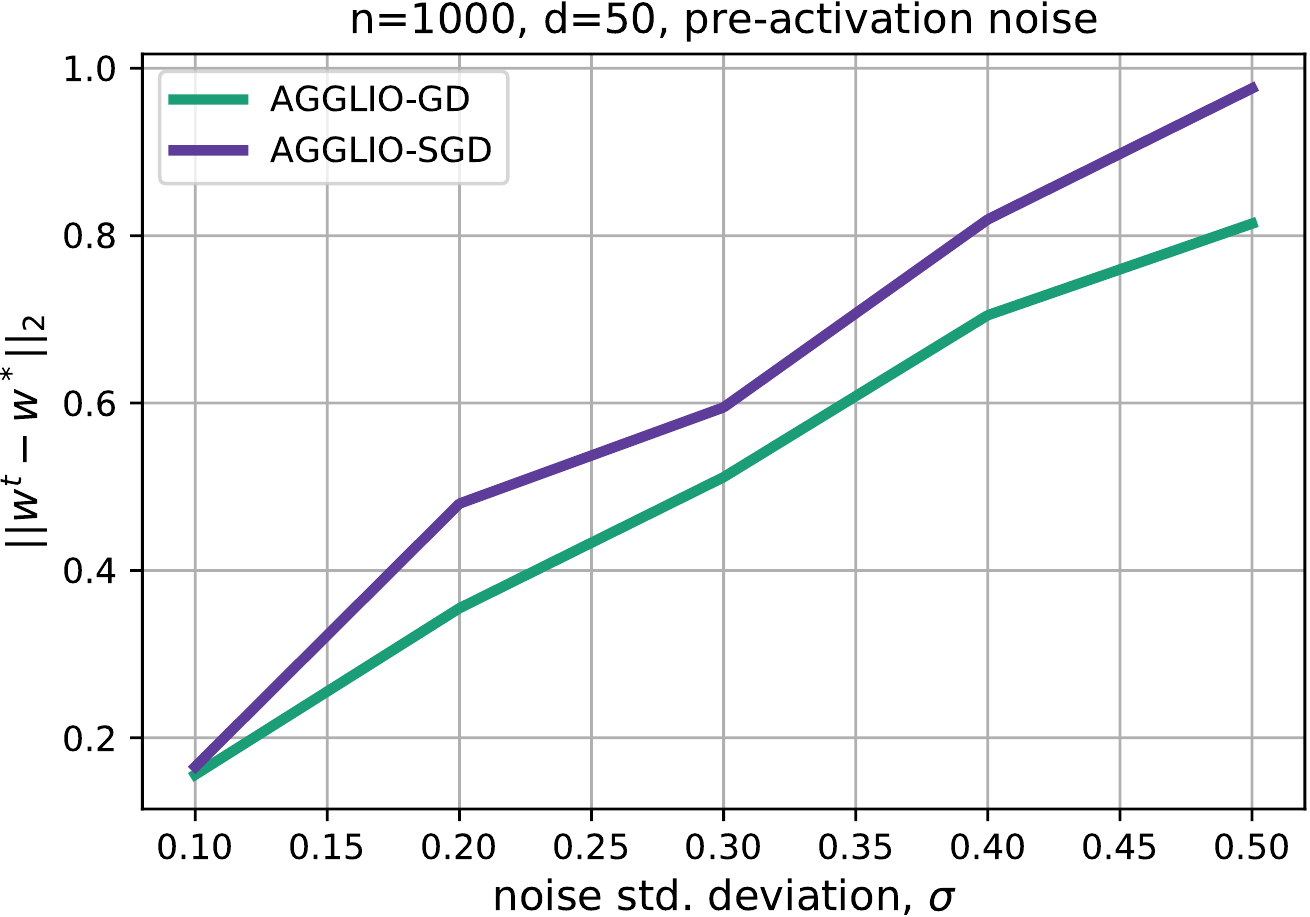}\\
	\small (a) dimension, $d$ & \small (b) noise std.dev. $\varsigma$
	\end{tabular}
	\caption{Figure (a) shows the effect of increasing covariate dimension $d$ on the performance of \alg variants in the noise-less setting. Both variants converge to near machine precision level accuracies (note that the y-axis is of the order of $10^{-16}$). Figure (b) shows the effect of noise magnitude, where the standard deviation of Gaussian noise was varied in the range $[0.1, 0.5]$ and applied pre-activation.}
	\label{fig:vary_dim_noise}
\end{figure}

\section{Experiments}
\label{sec:exps}
\noindent\textbf{Experimental Setting.} All experiments were performed on an Intel Core i7-6500U CPU with 16GB RAM. Statistics such as dataset size $n$, covariate dimensions $d$ and noise variance $\varsigma^2$ are mentioned above each figure. $20\%$ of train data was used as a held-out validation set, and grid search was used to tune hyperparameters except in hyper-parameter sensitivity experiments Fig~\ref{fig:sensitivity}(a-c). Synthetic datasets were generated using co-variates $\vx_i \sim \cN(\vzero, I_d)/\sqrt{d}$ and gold model $\vwo \sim N(0, I_d)$. All algorithms were initialized at a random model $\vw_0 \sim 10*\cN(\vzero, I_d)/\sqrt{d}$.

\noindent\textbf{Benchmarks.} \alg was bench-marked against several baselines and competitors, including Adam \cite{KingmaBa2015}, YOGI \cite{ReddiZSKK2018}, GradOpt \cite{HazanLS-S2016} and NGD \cite{HazanLS-S2015}. We note that this implicitly covers some other recent works such as \cite{SoltanolkotabiJL2019} which essentially propose SGD as the base method. Some techniques such as \cite{BalakrishnanWY2017} were not chosen as baselines since they require initializing the algorithm ``sufficiently'' close to the optimum. Adam was implemented with bias correction steps as recommended in \cite{KingmaBa2015}, with $m_0, v_0$ initialized to zero. YOGI has the same set of hyper-parameters as Adam and suggests initialization of $v_0$ to the square of the gradient, but initialization $v_0$ to zero gave better results. Equal batch sizes were offered to all minibatch-SGD methods, including \alg-SGD. 

\noindent\textbf{Hyperparameter Ranges.} We use notation used for the \verb!linspace! method in the Python library \verb!numpy! to describe hyperparameter ranges. Thus, \verb!linspace(a,b,n)! indicates that $n$ equi-spaced values in the range $[a,b]$ were used. For \alg, hyperparameters were tuned as follows: step length $\eta \in $\verb!linspace(1,500,10)!, initial temperature $\tau_0 \in \{10^{-1},10^{-2}, 10^{-3}, 10^{-4}\}$, and temperature increment $\beta \in $\verb!linspace(1.01,2,5)!. Both Adam and YOGI share the same four hyperparameters among which YOGI tunes only two in it's experiments and fixes the others. However, we found the fixed values to be sub-optimal and tuned all the four hyperparamaters for both the methods. We report the hyperparameter ranges in \verb!linspace! notation as well. The step $\alpha$ was tuned in \verb!linspace(0.01,0.2,5)!, momentum parameters $\beta_1, \beta_2$ in \verb!linspace(0.01,0.9,5)! and adaptivity level $\epsilon \in \{10^{-3},10^{-5},10^{-8}\}$. Similar hyperparameter tuning was performed for NGD, where step length was taken in the range \verb!linspace(0.01,10,20)! and GradOpt, for which the $\alpha$ hyperparameter was tuned in the range \verb!linspace(1,400,20)!.

\noindent\textbf{Convergence results.} Figures~\ref{fig:conv}(a,b,c) show the convergence of \alg and competitors in noiseless as well as noisy (pre- and post-activation label noise) setting with sigmoid activation. In the noisy setting, Gaussian noise $\epsilon \sim \cN(0, 0.5^2)$ was added for pre- and post-activation. \alg demonstrates superior convergence over the state-of-the-art methods in all settings, converging to machine precision-level errors e.g. $10^{-14}$ for noiseless labels at least twice as fast as any competitor. In the noisy setting, \alg is comparable or better convergence rates (note that the x-axes in Figures~\ref{fig:conv}(b,c) are in log-scale). Figures~\ref{fig:agglio_conv}(a,b,c) reproduces this experiment with \alg variants in settings similar to Fig.~\ref{fig:conv}. We observe \alg's GD, SGD and Adam versions give comparable performance for sigmoid activation with the SVRG variant being a bit more sluggish. Figures~\ref{fig:act_conv}(a,b,c) compare the convergence of various algorithms in the noiseless setting but with different activation functions including softplus, leaky softplus and SiLU. \alg offers performance superior to competitors both in terms of convergence speed as well as final model recovery error. YOGI \cite{ReddiZSKK2018} is the closest competitor and convergae to the gold model but slower than \alg.

\noindent\textbf{Sensitivity to Hyperparameters.} Figures~\ref{fig:sensitivity}(a,b,c) show the stability of \alg with respect to its hyper-parameters in a noiseless setting by varying each hyper-parameter in an admissible range while keeping the rest fixed to pre-tuned values as mentioned in each figure title. Note that y-axes values in all these plots is of the order of $10^{-16}$. The plots establish that both \alg-GD and \alg-SGD converge to essentially the gold model $\vwo$ despite wide mis-specification in the hyperparameters.

\noindent\textbf{Consistency of \alg.} Figure~\ref{fig:vary_dim_noise}(b) shows the effect of increasing the noise magnitude by varying the noise variance $\varsigma^2$ in the range $[0.1^2, 0.5^2]$ and used to corrupt the labels in the pre-activation setting. As expected, both variants offer linear increase in error levels with increasing $\varsigma$ that is consistent with the $\bigO\varsigma$ error rate established in Section~\ref{sec:noisy-conv}. Figures~\ref{fig:consistency}(a,b) show that if the noise variance is small (e.g. $\varsigma = 0.1, 0.05$, then \alg does offer vanishing error as the number of data points is increased. This does not seem to be the case for larger noise variance (e.g. $\varsigma = 0.5, 1$) where recovery error does not appreciably decrease even when the number of data points is increased by an order of magnitude. Thus, these results corroborate Theorem~\ref{thm:consistent-post-noisy-sigmoid} and further indicate that low-variance seems to be sufficient as well as necessary for consistent recovery. What is curious is that Figure~\ref{fig:consistency}(a) seems to indicate consistent recovery even in pre-activation setting that is not assured by Theorem~\ref{thm:consistent-post-noisy-sigmoid} and needs to be explored further.

\noindent\textbf{Effect of covariate dimension and Real Data.} Figure~\ref{fig:vary_dim_noise}(a) shows the effect of varying the covariate dimension $d$ on \alg's performance in the noise-less setting, where both \alg-GD and \alg-SGD converge to near-machine precision errors (y-axis values in the range $10^{-16}$). Figure~\ref{fig:consistency}(c) shows the performance of \alg and competitors on the Boston Housing regression dataset. The dataset was processed by normalizing its target labels and passing them through a sigmoidal activation before learning a GLM model using the sigmoidal activation by minimizing an objective of the form \eqref{eq:obj-main-restated}. \alg, Adam and YOGI offered significantly lower test RMSE scores compared to NGD and GradOpt although \alg offered noticeably faster convergence.

\section{Conclusion and Future Work}
\label{sec:conc}
This paper developed the \alg method and associated techniques such as the notions of ELSC/ELSS (Definition~\ref{defn:elsc-elss}) to analyze its convergence for a wide range of generalized linear model learning problems using popular activation functions such as sigmoid, softplus and SiLU. The most prominent directions of future work include \textbf{1)} exploring consistent model recovery under low variance pre-activation noise as experiments seem to indicate (Theorem~\ref{thm:conv-noisy-sigmoid} works only for low variance post-activation noise), \textbf{2)} applying momentum and other acceleration techniques to further boost \alg's convergence as has been shown to benefit other descent methods \cite{CutkoskyMehta2020}, \textbf{3)} exploring learning settings with \emph{censored} labels e.g. where $y = \phi(\vx^\top\vwo) + \epsilon$ is thresholded/binned to produce class/count labels, and \textbf{4)} extending \alg to deep networks with one or more hidden layers.

\section*{Acknowledgements}
D.D. is supported by the Visvesvaraya PhD Scheme for Electronics \& IT (FELLOW/2016-17/MLA/194). P.K. thanks Microsoft Research India and Tower Research for research grants.

\bibliographystyle{plain}
\bibliography{refs}

\begin{thebibliography}{10}

\bibitem{BalakrishnanWY2017}
Sivaraman Balakrishnan, Martin~J. Wainwright, and Bin Yu.
\newblock {Statistical Guarantees for the EM Algorithm: From Population to
  Sample-based Analysis}.
\newblock {\em Annals of Statistics}, 45(1):77--120, 2017.

\bibitem{BhatiaJK2015}
Kush Bhatia, Prateek Jain, and Purushottam Kar.
\newblock {Robust Regression via Hard Thresholding}.
\newblock In {\em Proceedings of the 29th Annual Conference on Neural
  Information Processing Systems (NIPS)}, 2015.

\bibitem{BlumensathD2009}
Thomas Blumensath and Mike~E. Davies.
\newblock {Iterative Hard Thresholding for Compressed Sensing}.
\newblock {\em Applied and Computational Harmonic Analysis}, 27(3):265--274,
  2009.

\bibitem{BottouCN2018}
L\'{e}on Bottou, Frank~E. Curtis, , and Jorge Nocedal.
\newblock {Optimization Methods for Large-Scale Machine Learning}.
\newblock {\em SIAM Review}, 60(2):223--311, 2018.

\bibitem{BoydVandenberghe2004}
Stephen Boyd and Lieven Vandenberghe.
\newblock {\em {Convex Optimization}}.
\newblock Cambridge University Press, 2004.

\bibitem{CutkoskyMehta2020}
Ashok Cutkosky and Harsh Mehta.
\newblock {Momentum Improves Normalized SGD}.
\newblock In {\em 37th International Conference on Machine Learning (ICML)},
  2020.

\bibitem{DanilovaDGGGKS2020}
Marina Danilova, Pavel Dvurechensky, Alexander Gasnikov, Eduard Gorbunov,
  Sergey Guminov, Dmitry Kamzolov, and Innokentiy Shibaev.
\newblock {Recent Theoretical Advances in Non-Convex Optimization}, 2020.
\newblock arXiv:2012.06188 [math.OC].

\bibitem{DefazioBL-J2014}
Aaron Defazio, Francis~R. Bach, and Simon Lacoste-Julien.
\newblock {SAGA: A Fast Incremental Gradient Method With Support for
  Non-Strongly Convex Composite Objectives}.
\newblock In {\em 28th Annual Conference on Neural Information Processing
  Systems (NIPS)}, 2014.

\bibitem{GlorotBB2011}
Xavier Glorot, Antoine Bordes, and Yoshua Bengio.
\newblock {Deep Sparse Rectifier Neural Networks}.
\newblock In {\em 14th International Conference on Artificial Intelligence and
  Statistics (AISTATS)}, 2011.

\bibitem{HastieTW2015}
Trevor Hastie, Robert Tibshirani, and Martin Wainwright.
\newblock Statistical learning with sparsity: The lasso and generalizations.
\newblock In {\em Monographs on Statistics and Applied Probability}, volume
  143. Chapman and Hall/CRC, 2015.

\bibitem{HazanLS-S2015}
Elad Hazan, Kfir~Y. Levy, and Shai Shalev-Shwartz.
\newblock {Beyond Convexity: Stochastic Quasi-Convex Optimization}.
\newblock In {\em Proceedings of the 29th Annual Conference on Neural
  Information Processing Systems (NIPS)}, 2015.

\bibitem{HazanLS-S2016}
Elad Hazan, Kfir~Y. Levy, and Shai Shalev-Shwartz.
\newblock {On Graduated Optimization for Stochastic Non-Convex Problems}.
\newblock In {\em 33rd International Conference on Machine Learning (ICML)},
  2016.

\bibitem{HendrycksGimpel2016}
Dan Hendrycks and Kevin Gimpel.
\newblock {Gaussian Error Linear Units (GELUs)}, 2020.
\newblock arXiv:1606.08415v4 [cs.LG].

\bibitem{JainKar2017}
Prateek Jain and Purushottam Kar.
\newblock Non-convex optimization for machine learning.
\newblock {\em Foundations and Trends{\textregistered} in Machine Learning},
  10(3--4):142--336, 2017.

\bibitem{JainTK2014}
Prateek Jain, Ambuj Tewari, and Purushottam Kar.
\newblock {On Iterative Hard Thresholding Methods for High-dimensional
  M-Estimation}.
\newblock In {\em 28th Annual Conference on Neural Information Processing
  Systems (NIPS)}, 2014.

\bibitem{JohnsonZhang2013}
Rie Johnson and Tong Zhang.
\newblock {Accelerating Stochastic Gradient Descent using Predictive Variance
  Reduction}.
\newblock In {\em 27th Conference on Neural Information Processing Systems
  (NIPS)}, 2013.

\bibitem{KingmaBa2015}
Diederik~P. Kingma and Jimmy Ba.
\newblock {Adam: A Method for Stochastic Optimization}.
\newblock In {\em 3rd International Conference on Learning Representations
  (ICLR)}, 2015.

\bibitem{MaasHN2013}
Andrew~L. Maas, Awni~Y. Hannun, and Andrew~Y. Ng.
\newblock {Rectifier Nonlinearities Improve Neural Network Acoustic Models}.
\newblock In {\em 30th International Conference on Machine Learning (ICML)},
  2013.

\bibitem{McCullaghNelder1989}
P.~McCullagh and John~A. Nelder.
\newblock {\em {Generalized Linear Models}}.
\newblock Chapman and Hall, 1989.

\bibitem{MukhotyGJK2019}
Bhaskar Mukhoty, Govind Gopakumar, Prateek Jain, and Purushottam Kar.
\newblock {Globally-convergent Iteratively Reweighted Least Squares for Robust
  Regression Problems}.
\newblock In {\em Proceedings of the 22nd International Conference on
  Artificial Intelligence and Statistics (AISTATS)}, 2019.

\bibitem{NairHinton2010}
Vinod Nair and Geoffrey~E. Hinton.
\newblock {Rectified Linear Units Improve Restricted Boltzmann Machines}.
\newblock In {\em 27th International Conference on Machine Learning (ICML)},
  2010.

\bibitem{NelderWedderburn1972}
J.~A. Nelder and R.~W.~M. Wedderburn.
\newblock {Generalized Linear Models}.
\newblock {\em Journal of the Royal Statistical Society. Series A},
  135(3):370--384, 1972.

\bibitem{ReddiHSPS2016}
Sashank~J. Reddi, Ahmed Hefny, Suvrit Sra, Barnab\'{a}s P\'{o}cz\'{o}s, and
  Alex Smola.
\newblock {Stochastic Variance Reduction for Nonconvex Optimization}.
\newblock In {\em 33rd International Conference on Machine Learning (ICML)},
  2016.

\bibitem{ReddiZSKK2018}
Sashank~J. Reddi, Manzil Zaheer, Devendra Sachan, Satyen Kale, and Sanjiv
  Kumar.
\newblock {Adaptive Methods for Nonconvex Optimization}.
\newblock In {\em 32nd Conference on Neural Information Processing Systems
  (NeurIPS)}, 2018.

\bibitem{SoltanolkotabiJL2019}
Mahdi Soltanolkotabi, Adel Javanmard, and Jason~D Lee.
\newblock {Theoretical Insights Into the Optimization Landscape of
  Over-Parameterized Shallow Neural Networks}.
\newblock {\em IEEE Transactions on Information Theory}, 65(2):742--769, 2019.

\bibitem{Vershynin2018}
Roman Vershynin.
\newblock {\em {High-Dimensional Probability: An Introduction with Applications
  in Data Science}}.
\newblock Cambridge Series in Statistical and Probabilistic Mathematics.
  Cambridge University Press, 2018.

\bibitem{WilsonMW2019}
Ashia~C. Wilson, Lester Mackey, and Andre Wibisono.
\newblock {Accelerating Rescaled Gradient Descent: Fast Optimization of Smooth
  Functions}.
\newblock In {\em 33rd Conference on Neural Information Processing Systems
  (NeurIPS)}, 2019.

\bibitem{YangTR2013}
Eunho Yang, Ambuj Tewari, and Pradeep Ravikumar.
\newblock {On Robust Estimation of High Dimensional Generalized Linear Models}.
\newblock In {\em 23rd International Joint Conference on Artificial
  Intelligence (IJCAI)}, 2013.

\end{thebibliography}

\appendix
\allowdisplaybreaks
\section{Calculation of ELSC/ELSS Constants in Table~\ref{tab:elsc-elss} and Proof of Lemma~\ref{lem:elsc-elss}}
\label{app:app-elsc-elss}

The proof technique to establish Lemma~\ref{lem:elsc-elss} and the constants in Table~\ref{tab:elsc-elss} uses tail deviation bounds for the spectra of positive definite matrices \cite{Vershynin2018} to show that with high probability, the Hessians of the (graduated) objective functions are positive definite at all points within a ball around $\vwo$. In summary, we execute the following proof steps given data $\bc{(\vx_i,y_i)}_{i=1}^n$ where $\vx_i \sim \cN(\vzero, I_d)$ and $y_i = \phi(\vx_i^\top\vwo)$:
\begin{enumerate}[leftmargin=4\parindent]
	\item[\textbf{Step 1}]: Show that for any $r \in [0, R], \tau \in [0,1]$ and any $\vw \in \cB(\vwo, r)$, the Hessian $\nabla^2_\vw\cL_\tau(\vw)$ can be written as $\frac1n\sum s_i\cdot\vx_i\vx_i^\top$ where the weights $s_i$ depend only on $\vx_i, \vw, \vwo, \phi, \tau$
	\item[\textbf{Step 2}]: Show that for some $\Lambda \geq \lambda > 0$, for any fixed unit vector $\vv \in S^{d-1}$, we have
	\[
	\Ee{\vx \sim \cN(\vzero, I)}{s\ip{\vx}{\vv}^2} \in [\lambda, \Lambda]
	\]
	\item[\textbf{Step 3}]: Show that the random variable $s\ip{\vx}{\vv}^2$ is sub-exponential and use Bernstein-style tail inequalities for subexponential random variables such as those from \cite{Vershynin2018} to show that with confidence at least $1 - \exp(-\Om n)$, for any fixed unit vector $\vv$
	\[
	\frac1n\sum s_i\ip{\vx_i}{\vv}^2 \in [0.999\lambda, 1.001\Lambda]
	\] 
	\item[\textbf{Step 4}]: Take a union bound over an $\epsilon$-cover over all unit vectors $\vv \in S^{d-1}$ with $\epsilon = \frac14$ and use standard arguments to show that whenever $n \geq \Om{d\log d}$, we have with confidence at least $1 - \exp(-\Om n)$
	\[
	0.998\lambda \leq \lambda_{\min}(\nabla^2_\vw\cL_\tau(\vw)) \leq \lambda_{\max}(\nabla^2_\vw\cL_\tau(\vw)) \leq 1.002\Lambda,
	\]
	where $\lambda_{\min}$ and $\lambda_{\max}$ denote respectively, the smallest and largest eigenvalues of a square PSD matrix.
	\item[\textbf{Step 5}]: Take another union bound over another $\epsilon$-cover, this time over all models $\vw \in \cB(\vwo,r)$ with $\epsilon \approx \frac1{(d + \log n)^{\bigO1}}$ to show that, with confidence at least $1 - \exp(-\Om n)$, for all $\vw \in \cB(\vwo,r)$, we have
	\[
	0.997\lambda \leq \lambda_{\min}(\nabla^2_\vw\cL_\tau(\vw)) \leq \lambda_{\max}(\nabla^2_\vw\cL_\tau(\vw)) \leq 1.003\Lambda
	\]
	\item[\textbf{Step 6}]: Take union bounds over separate $\epsilon$-nets over all $r \in [0, R]$ and all $\tau \in [0,1]$ to show that with confidence at least $1 - \exp(-\Om n)$, for any $r \in [0, R], \tau \in [0,1]$ and any $\vw \in \cB(\vwo, r)$, we have
	\[
	0.996\lambda \leq \lambda_{\min}(\nabla^2_\vw\cL_\tau(\vw)) \leq \lambda_{\max}(\nabla^2_\vw\cL_\tau(\vw)) \leq 1.004\Lambda
	\]
\end{enumerate}
The above chain of arguments can be seen to establish the ELSC/ELSS properties by noting that a doubly differentiable function $f: \bR^d \rightarrow \bR$ is $\lambda$-strongly convex (resp. $\Lambda$-strongly smooth) if and only if
\[
\lambda \leq \lambda_{\min}(\nabla^2f(\vx)) \leq \lambda_{\max}(\nabla^2f(\vx)) \leq \Lambda
\]
We also note that Steps 3, 4, 5, 6 are quite routine and readily established. We will only sketch them for the sigmoid activation case to present a use case. Steps 1 and 2, on the other hand, are the critical calculations that would need to be repeated for every activation function.

\subsection{Steps 1 and 2: Sigmoid Activation}
\label{app:step12-sigmoid}
Let us use the shorthands $\sigma^\tau_i := \sigma_\tau(\vx_i^\top\vw)$ to denote the prediction by model $\vw$ with the $\tau$-graduated sigmoid $\sigma$ activation and $y^\tau_i := \sigma_\tau(\vx_i^\top\vwo) = \sigma_\tau(\sigma^{-1}(y_i))$ to denote the same for the gold model $\vwo$ (recall that we are working in the noiseless setting here). Then, elementary calculations show that
\begin{align*}
	\nabla_\vw\cL_\tau(\vw) &= \frac{2\tau}n\sum_{i=1}^n(\sigma^\tau_i - y^\tau_i)\br{\sigma^\tau_i - (\sigma^\tau_i)^2}\cdot\vx_i\\
	\nabla^2_\vw\cL_\tau(\vw) &= \frac{2\tau^2}n\sum_{i=1}^n\br{\br{\sigma^\tau_i - (\sigma^\tau_i)^2}^2 + (\sigma^\tau_i - y^\tau_i)(1 - 2\sigma^\tau_i)\br{\sigma^\tau_i - (\sigma^\tau_i)^2}}\cdot\vx_i\vx_i^\top = \frac1n\sum s_i\cdot\vx_i\vx_i^\top,
\end{align*}
where $s_i = 2\tau^2\br{\underbrace{\br{\sigma^\tau_i - (\sigma^\tau_i)^2}^2}_{(A)} + \underbrace{(\sigma^\tau_i - y^\tau_i)(1 - 2\sigma^\tau_i)\br{\sigma^\tau_i - (\sigma^\tau_i)^2}}_{(B)}}$. We note that the term $(A)$ is always strictly positive but the second term $(B)$ could be positive or negative. However, we note that we nevertheless have the bound
\begin{align*}
	\abs{(B)} &\leq \frac1{10}\abs{\sigma^\tau_i - y^\tau_i}\\
					  &\leq \frac\tau{40}\abs{\vx_i^\top(\vw - \vwo)},
\end{align*}
where the first step follows since the function $x \mapsto \abs{(1-2x)(x-x^2)}$ takes values only upto $0.1$ in the interval $x \in [0,1]$ which the range of the sigmoid function, and the second step follows since the $\tau$-graduated sigmoid function is $\frac\tau4$-Lipschitz. Using this, we bound the contributions due to both terms below. We first bound the contribution due to $(B)$ followed by that due to $(A)$.

For $\vx \sim \cN(\vzero, I_d)$ with $\vx = (x_1, x_2, \ldots, x_d)$ and fixed vectors $\vw, \vwo, \vv$ we wish to calculate the quantity $\E{\abs{\vx_i^\top(\vw - \vwo)}\ip{\vx_i}{\vv}^2}$. However, by rotational symmetry of the normal distribution, we can w.l.o.g. take the vector $\vw - \vwo$ to be non-zero only on its first coordinate i.e. $\vw - \vwo = (\norm{\vw-\vwo}_2\cdot\pm1,0,0\ldots,0)$ and $\vv$ to be non-zero in only its first two coordinates i.e. $\vv = (v_1, v_2, 0, 0, \ldots, 0)$ with $v_1^2 + v_2^2 = 1$. Using the fact that coordinates of a normal random vector $\vx$ are i.i.d. normal scalar random variables themselves, we have
\begin{align*}
	\E{\abs{\vx^\top(\vw - \vwo)}\ip\vx\vv^2} &= \Ee{x_1, x_2 \sim \cN(0,1)}{\norm{\vw-\vwo}_2\cdot\abs{x_1}\cdot(x_1v_1 + x_2v_2)^2}\\
	&\leq r \cdot \Ee{x_1 \sim \cN(0,1)}{\abs{x_1}^3v_1^2 + \abs{x_1}v_2^2}\\
	&= r\sqrt{\frac2\pi}(2v_1^2 + v_2^2) \leq 2r\sqrt{\frac2\pi},
\end{align*}
where in the second step we use the fact that $\E{\cond{x_1x_2}x_1} = 0$ and $\E{x_2^2} = 1$ as $x_1,x_2$ are i.i.d. normal, apply the law of total expectation as well as use the fact that by assumption $\vw \in \cB(\vwo,r)$, and in the last step we used the fact that the definite Gaussian integral obtained in the second step happens to have a closed-form solution, and that $v_1^2 + v_2^2 = 1$ and as a result, $v_1^2 \leq 1$. This bounds the contribution due to the term $(B)$ upon incorporating the multiplicative factor of $\frac\tau{40}$.

We now move on to bound the contribution due to the term $(A)$. Note that this term does not involve $\vwo$ at all. Yet again, using rotational symmetry, we take $\vw = (\norm\vw_2\cdot\pm1,0,0,\ldots,0)$ and $\vv = (v_1, v_2, 0, 0, \ldots, 0)$ with $v_1^2 + v_2^2 = 1$. This gives us
\begin{align*}
	\E{\br{\sigma^\tau - (\sigma^\tau)^2}^2\ip{\vx_i}{\vv}^2} &= \Ee{x_1, x_2 \sim \cN(0,1)}{\frac{\exp(-2\tau\norm{\vw}_2\cdot x_1)}{\br{1+\exp(-\tau\norm{\vw}_2\cdot x_1)}^4}\br{x_1^2v_1^2+x_2^2v_2^2+2x_1x_2v_1v_2}}\\
	&= \Ee{x_1 \sim \cN(0,1)}{\frac{\exp(-2\tau\norm{\vw}_2\cdot x_1)}{\br{1+\exp(-\tau\norm{\vw}_2\cdot x_1)}^4}\br{x_1^2v_1^2+v_2^2}}\\
	&= \sqrt{\frac2\pi}\int_0^\infty \frac{\exp(-2\tau\norm{\vw}_2\cdot t)}{\br{1+\exp(-\tau\norm{\vw}_2\cdot t)}^4}\br{v_1^2t^2+v_2^2}\exp\br{-\frac{t^2}2}\ dt\\
	&= \sqrt{\frac2\pi}\exp(2c^2)\int_0^\infty \frac{\exp\br{-\frac12(t+2c)^2}}{\br{1+\exp(-c\cdot t)}^4}\br{v_1^2t^2+v_2^2}\ dt\\
	&\geq \frac1{16}\sqrt{\frac2\pi}\exp(2c^2)\int_0^\infty \exp\br{-\frac12(t+2c)^2}\br{v_1^2t^2+v_2^2}\ dt,
\end{align*}
where the second step follows from using $\E{\cond{x_1x_2}x_1} = 0$ and $\E{x_2^2} = 1$, the third step follows since the integrand is an even function, the fourth step follows from simple manipulations and using the shorthand $c =  \tau\norm\vw_2 \leq \tau\cdot m_r$ where we define $m_r = \max_{\vw \in \cB(\vwo,r)}\ \norm\vw_2 \leq \norm\vwo_2 + r$, the fifth step follows from using $c > 0$. The resulting definite integral can be computed as shown below, but in terms of the (complementary) error functions. These do not have closed form expressions but can be readily lower bounded as shown below.
\begin{align*}
	\exp(2c^2)\int_0^\infty \exp\br{-\frac12(t+2c)^2}\br{v_1^2t^2+v_2^2}\ dt &= v_1^2\br{\sqrt{\frac\pi2}(4c^2+1)\exp(2c^2)\erfc(\sqrt2c) - 2c}\\
	&\quad+ v_2^2\br{\sqrt{\frac\pi2}\exp(2c^2)\erfc(\sqrt2c)}\\
	&\geq \sqrt{\frac\pi2}(4c^2+1)\exp(2c^2)\erfc(\sqrt2c) - 2c\\
	&\geq \exp(-3c),
\end{align*}
where the second step hold since $v_1^2 + v_2^2 = 1$ and the multiplier of $v_1^2$ on the right hand side of the first line is dominated by the multiplier of $v_2^2$, and the last step follows by inspection. This allows us to establish the following lower bound on the expectation of interest
\[
\Ee{\vx \sim \cN(\vzero, I)}{s\ip{\vx}{\vv}^2} \geq \lambda \equiv \sqrt{\frac2\pi}\tau^2\br{\frac{\exp(-3\tau\cdot m_r)}8 - \frac{\tau\cdot r}{10}}
\]
Similarly, by using bound $1 + \exp(-c\cdot t) \geq 1$ to upper bound the expectation $\E{\br{\sigma^\tau - (\sigma^\tau)^2}^2\ip{\vx_i}{\vv}^2}$, and choosing the term corresponding to $v_2^2$ as the dominant one this time, we can get an upper bound of the following form
\[
\Ee{\vx \sim \cN(\vzero, I)}{s\ip{\vx}{\vv}^2} \leq \Lambda \equiv \sqrt{\frac2\pi}\tau^2\br{\exp(-\tau\cdot m_r/10) + \frac{\tau\cdot r}{10}}
\]
This concludes the expectation calculations for the sigmoid activation function.

\subsection{Steps 1 and 2: Softplus Activation}
\label{app:step12-softplus}
Let us denote the softplus activation with the letter $\rho$ (since it is a soft counterpart to the \emph{ReLU} activation). We will use $y^\tau_i := \rho_\tau(\vx_i^\top\vwo) = \rho_\tau(\rho^{-1}(y_i))$ to denote the predictions by the gold model $\vwo$ (recall that we are working in the noiseless setting here) and $\rho^\tau_i := \rho_\tau(\vx_i^\top\vw)$ to to denote the predictions by model $\vw$ with the $\tau$-graduated softplus $\rho$ activation. It turns out that the derivatives of the (graduated) softplus activation are related to those of the (graduated) sigmoid activation. Thus, we will continue to use $\sigma^\tau_i := \sigma_\tau(\vx_i^\top\vw)$ to to denote the predictions by model $\vw$ with the $\tau$-graduated sigmoid $\sigma$ activation. Then straightforward calculations show that
\begin{align*}
	\nabla_\vw\cL_\tau(\vw) &= \frac2n\sum_{i=1}^n(\rho^\tau_i - y^\tau_i)\sigma^\tau_i\cdot\vx_i\\
	\nabla^2_\vw\cL_\tau(\vw) &= \frac2n\sum_{i=1}^n\br{(\sigma^\tau_i)^2 + \tau(\rho^\tau_i - y^\tau_i)\br{\sigma^\tau_i - (\sigma^\tau_i)^2}}\cdot\vx_i\vx_i^\top = \frac1n\sum s_i\cdot\vx_i\vx_i^\top,
\end{align*}
where $s_i = 2\br{\underbrace{(\sigma^\tau_i)^2}_{(A)} + \underbrace{\tau(\rho^\tau_i - y^\tau_i)\br{\sigma^\tau_i - (\sigma^\tau_i)^2}}_{(B)}}$. Yet again, note that the term $(A)$ is always strictly positive but the second term $(B)$ could be positive or negative. However, we can similarly bound
\begin{align*}
	\abs{(B)} &\leq \frac\tau4\abs{\rho^\tau_i - y^\tau_i}\\
	&\leq \frac\tau4\abs{\vx_i^\top(\vw - \vwo)}
\end{align*}
where the first step follow since we always have $\br{\sigma^\tau_i - (\sigma^\tau_i)^2} \leq \frac14$ and the second step follows since the $\tau$-graduated softplus function is $1$-Lipschitz. However, upto the multiplicative constants, we have already bounded the contribution of this term above as
\[
\E{\abs{\vx_i^\top(\vw - \vwo)}\ip{\vx_i}{\vv}^2} \leq r\sqrt{\frac2\pi}(2v_1^2 + v_2^2) \leq 2r\sqrt{\frac2\pi}
\]
which completes the upper bound on the contribution due to $(B)$.

To bound the contribution due to $(A)$ we get, by exploiting rotational symmetry as before,
\begin{align*}
\Ee{\vx \sim \cN(\vzero, I_d)}{\sigma_\tau(\vw^\top\vx)^2\ip{\vx}{\vv}^2}
&= \Ee{x_1,x_2\sim \cN(0,1)}{\sigma_\tau(\norm{\vw}_2x_1)^2(x_1^2v_1^2+x_2^2v_2^2+2x_1x_2v_1v_2)}\\
&= \Ee{x_1 \sim \cN(0,1)}{\sigma_\tau(\norm{\vw}_2x_1)^2(x_1^2v_1^2+v_2^2)}
\end{align*}
Now we apply the Jensen's inequality and use the notation $c = \tau\norm\vw_2 \leq \tau\cdot m_r$ where we define $m_r = \max_{\vw \in \cB(\vwo,r)}\ \norm\vw_2 \leq \norm\vwo_2 + r$, to get
\begin{align*}
\Ee{x_1 \sim \cN(0,1)}{\sigma_\tau(\norm{\vw}_2x_1)^2(x_1^2v_1^2)}
&\geq \exp\br{\E{2\ln\sigma_\tau(\norm{\vw}_2x_1) + \ln(x_1^2v_1^2)}}\\
&= \exp\br{\Ee{x}{-2\ln (1 + \exp(-cx))+ \ln(x^2v_1^2)}}\\
&\geq \exp\br{-2\ln (1 + \E{\exp(-cx)})+ \E{\ln(x^2v_1^2)}}\\
&= \exp\br{\ln\frac{1}{(1+ \exp(c^2/2))^2}+ \ln v_1^2 - 4}\\
&= \frac{v_1^2}{e^4}\frac{1}{(1+\exp(c^2/2))^2},
\end{align*}
where in the fifth step we used $\E{\ln(x^2)} \geq -4$. The second part corresponding to $v_2^2$ is identical but for the $e^4$ term. Combining as before and choosing the smaller term (corresponding to $v_1^2$) and using $v_1^2+v_2^2=1$ gives us
\[
\Ee{\vx \sim \cN(\vzero, I_d)}{\sigma_\tau(\vx^\top\vw)^2\ip{\vx}{\vv}^2}\geq \frac1{e^4}\frac1{(1+\exp(c^2/2))^2}
\]
Putting these results together give us the following lower bound on the expectation of interest
\[
\Ee{\vx \sim \cN(\vzero, I)}{s\ip{\vx}{\vv}^2} \geq \lambda \equiv 2\br{\frac1{e^4}\frac1{(1+\exp(c^2/2))^2} - \frac{\tau\cdot r}{\sqrt{2\pi}}}
\]
We can similarly obtain the upper bound as well. This concludes the expectation analysis for the softplus activation function.

\subsection{Steps 1 and 2: Leaky Softplus Activation}
\label{app:step12-leaky-softplus}
Let us denote the leaky softplus activation with the letter $\lrelu$ (since it is a soft counterpart to the \emph{leaky} ReLU activation). We will use $y^\tau_i := \lrelu_\tau(\vx_i^\top\vwo) = \lrelu_\tau(\lrelu^{-1}(y_i))$ to denote the predictions by the gold model $\vwo$ (recall that we are working in the noiseless setting here) and $\lrelu^\tau_i := \lrelu_\tau(\vx_i^\top\vw)$ to to denote the predictions by model $\vw$ with the $\tau$-graduated leaky softplus $\lrelu$ activation. We note that the leaky softplus has a parameter $k \in [0,1]$ to decide its \emph{leakiness} (much the same way the leaky ReLU does too).

It turns out that the derivatives of the (graduated) leaky softplus activation are related to those of the (graduated) sigmoid activation. Thus, we will continue to use $\sigma^\tau_i := \sigma_\tau(\vx_i^\top\vw)$ to to denote the predictions by model $\vw$ with the $\tau$-graduated sigmoid $\sigma$ activation. Then straightforward calculations show that
\begin{align*}
	\nabla_\vw\cL_\tau(\vw) &= \frac2n\sum_{i=1}^n(\lrelu^\tau_i - y^\tau_i)(\br{\sigma^\tau_i + k(1-\sigma^{k\tau}_i)}^2)\cdot\vx_i\\
	\nabla^2_\vw\cL_\tau(\vw) &= \frac1n\sum s_i\cdot\vx_i\vx_i^\top,
\end{align*}
where $s_i = 2\br{\underbrace{\br{\sigma^\tau_i + k(1-\sigma^{k\tau}_i)}^2}_{(A)} + \underbrace{\tau(\lrelu^\tau_i - y^\tau_i)\br{\sigma^\tau_i - (\sigma^\tau_i)^2 - k^2\br{\sigma^{k\tau}_i - (\sigma^{k\tau}_i)^2}}}_{(B)}}$. As before, the term $(A)$ is always strictly positive but the second term $(B)$ could be positive or negative. However, we can similarly bound
\begin{align*}
	\abs{(B)} &\leq \frac\tau4\abs{\lrelu^\tau_i - y^\tau_i}\\
	&\leq \frac\tau4\abs{\vx_i^\top(\vw - \vwo)}
\end{align*}
where the first step follow since we always have $\br{\sigma^\tau_i - (\sigma^\tau_i)^2} \leq \frac14$ as well as $\br{\sigma^{k\tau}_i - (\sigma^{k\tau}_i)^2} \leq \frac14$ and $k \leq 1$, the second step follows since the $\tau$-graduated leaky softplus function is $1$-Lipschitz. However, upto the multiplicative constants, we have already bounded the contribution of this term above as $\E{\abs{\vx_i^\top(\vw - \vwo)}\ip{\vx_i}{\vv}^2} \leq r\sqrt{\frac2\pi}(2v_1^2 + v_2^2) \leq 2r\sqrt{\frac2\pi}$ which completes the upper bound on the contribution due to $(B)$.

To bound the contribution due to $(A)$, we use the handy lower bound
\[
\br{\sigma^\tau_i + k(1-\sigma^{k\tau}_i)}^2 \geq \br{\sigma^\tau_i}^2 + k^2\br{(1-\sigma^{k\tau}_i)}^2,
\]
which holds since both terms are non-negative. We notice that we have already bounded the contribution due to the first term, namely $\br{\sigma^\tau_i}^2$, in our analysis of the softplus activation function in Appendix~\ref{app:step12-softplus}. Specifically, we have
\[
\Ee{\vx \sim \cN(\vzero, I_d)}{\sigma_\tau(\vx^\top\vw)^2\ip{\vx}{\vv}^2}\geq \frac1{e^4}\frac1{(1+\exp(c^2/2))^2}
\]
To bound the second term we get, by exploiting rotational symmetry as before,
\begin{align*}
\Ee{\vx \sim \cN(\vzero, I_d)}{(1-\sigma_{k\tau}(\vw^\top\vx))^2\ip{\vx}{\vv}^2}
&= \Ee{x_1,x_2\sim \cN(0,1)}{(1-\sigma_{k\tau}(\norm{\vw}_2x_1))^2(x_1^2v_1^2+x_2^2v_2^2+2x_1x_2v_1v_2)}\\
&= \Ee{x_1 \sim \cN(0,1)}{(1-\sigma_{k\tau}(\norm{\vw}_2x_1))^2(x_1^2v_1^2+v_2^2)}
\end{align*}
Now we apply the Jensen's inequality and use the notation $c = k\tau\norm\vw_2 \leq k\tau\cdot m_r$ where we define $m_r = \max_{\vw \in \cB(\vwo,r)}\ \norm\vw_2 \leq \norm\vwo_2 + r$, to get
\begin{align*}
\Ee{x_1 \sim \cN(0,1)}{(1-\sigma_{k\tau}(\norm{\vw}_2x_1))^2(x_1^2v_1^2+v_2^2)}
&\geq \exp\br{\E{2\ln(1-\sigma_{k\tau}(\norm{\vw}_2x_1)) + \ln(x_1^2v_1^2)}}\\
&= \exp\br{\E{2\ln\frac{\exp(-c\cdot x_1)}{1+\exp(-c\cdot x_1)} + \ln(x_1^2v_1^2)}}\\
&= \exp\br{\E{-2c\cdot x_1 - 2\ln(1+\exp(-c\cdot x_1)) + \ln(x_1^2v_1^2)}}\\
&\geq \exp\br{\log\frac1{(1 + \exp(-c^2/2))^2} - 4 + \ln(v_1^2)}\\
&= \frac{v_1^2}{e^4}\frac{1}{(1+\exp(c^2/2))^2},
\end{align*}
where in the fifth step we used $\E{\ln(x^2)} \geq -4$. The second part corresponding to $v_2^2$ is identical but for the $e^4$ term. Combining as before and choosing the smaller term (corresponding to $v_1^2$) and using $v_1^2+v_2^2=1$ gives us
\[
\Ee{\vx \sim \cN(\vzero, I_d)}{(1-\sigma_{k\tau}(\vx^\top\vw))^2\ip{\vx}{\vv}^2}\geq \frac1{e^4}\frac1{(1+\exp(c^2/2))^2}
\]
Putting these results together give us the following lower bound on the expectation of interest
\[
\Ee{\vx \sim \cN(\vzero, I)}{s\ip{\vx}{\vv}^2} \geq \lambda \equiv 2\br{\frac1{e^4}\frac{1+k^2}{(1+\exp(c^2/2))^2} - \frac{\tau\cdot r}{\sqrt{2\pi}}}
\]
We can similarly obtain the upper bound as well. This concludes the expectation analysis for the leaky softplus activation function.

\subsection{Step 3}
\label{app:step3}
In this step we convert the expected bounds into high-confidence bounds. First we establish that for any fixed unit vector $\vv \in S^{d-1}$ the random variable $s\cdot\ip{\vx}{\vv}^2$ is subexponential to allow us to apply Bernstein-style tail bounds. To see this, notice that for the Gaussian case (the (leaky) softplus case is handled separately below), we have $s \leq 4\tau^2 \leq 4$ since the sigmoid activation always takes values in the range $[0,1]$ and $\tau \in [0,1]$. Moreover, since $\vx \sim \cN(\vzero, I_d)$, standard results \cite{Vershynin2018} tell us that $\abs{\ip{\vx}{\vv}}$ is $1$-subGaussian. This implies that the random variable $\sqrt{\abs s}\cdot\abs{\ip{\vx}{\vv}}$ is at most $2$-subGaussian. Standard results now tell us that the random variable $s\cdot\ip{\vx}{\vv}^2$ must be at most $4$-subexponential. We now apply the Bernstein bound below with $\lambda$ (resp. $\Lambda$) being the lower (resp. upper) bound on the expectation $\E{s\cdot\ip{\vx}{\vv}^2}$ calculated in Appendices~\ref{app:step12-sigmoid}, \ref{app:step12-softplus}, and \ref{app:step12-leaky-softplus} above. 
\[
\P{\frac1n\sum_{i=1}^ns_i\cdot\ip{\vx_i}{\vv}^2 \leq \lambda - \epsilon} \leq \exp\br{-\frac{cn\epsilon^2}{16}},
\]
where $c > 0$ is some small universal constant. Taking $\epsilon = 0.001\lambda$ gives us
\[
\P{\frac1n\sum_{i=1}^ns_i\cdot\ip{\vx_i}{\vv}^2 \leq 0.999\lambda} \leq \exp\br{-\frac{cn(0.001\lambda)^2}{16}} = \exp(-\Om n).
\]
Doing a similar calculation to upper bound the summation $\frac1n\sum_{i=1}^ns_i\cdot\ip{\vx_i}{\vv}^2$ by $1.001\Lambda$ finishes this step. We now address the case of (leaky) softplus activation where the activation function itself does not take bounded values. Even in those cases, we notice that the analyses in Appendices~\ref{app:step12-softplus} and \ref{app:step12-leaky-softplus} tells us that
\[
\abs{s_i} \leq \bigO{1 + \tau\cdot r\abs{\vx_i^\top\vdelta}},
\]
where $\vdelta \in S^{d-1}$ is a unit vector in the direction of $\vw - \vwo$. Since $\vw, \vwo$ are fixed for now (we will take a union bound over them later), we can take $\vdelta$ to be a fixed (but otherwise arbitrary) unit vector. Now, it is easy to see that
\[
\Ee{\vx \sim \cN(\vzero, I_d)}{\br{\vx^\top\vdelta}^2} = 1
\]
As argued before, $\vx^\top\vdelta$ is a $1$-subGaussian variable since $\vx \sim \cN(\vzero, I_d)$ and hence $\br{\vx^\top\vdelta}^2$ is $1$-sub-exponential. Bernstein bounds then tell us that for any $t \geq \Om1$ we have
\[
\P{\br{\vx^\top\vdelta}^2 > t} \leq \exp\br{-\Om t}
\]
Taking $t = \sqrt[3]n$ and a union bound over all $i \in [n]$ gives us
\[
\P{\exists i \in [n]: \br{\vx_i^\top\vdelta}^2 > \Om{\sqrt[3]n}} \leq n\exp\br{-\Om{\sqrt[3]n}}
\]
This means that with probability at least $1 - \exp(-\softOm{\sqrt n})$, we have $\abs{s_i} \leq \bigO{\sqrt[6]n}$ simultaneously for all $i \in [n]$. This means that we can now take $s\cdot\ip{\vx}{\vv}^2$ to be at most $\bigO{\sqrt[3]n}$-subexponential. Applying the Bernstein inequality now tells us that
\[
\P{\frac1n\sum_{i=1}^ns_i\cdot\ip{\vx_i}{\vv}^2 \leq 0.999\lambda} \leq \exp\br{-\Om{\frac n{(\sqrt[3]n)^2}}} = \exp\br{-\Om{\sqrt[3]n}}.
\]
Thus, we retrieve essentially the same result but with diminished confidence even for the (leaky) softplus activations.

\subsection{Step 4}
\label{app:step4}
The previous results all held for a fixed, but otherwise arbitrary unit vector $\vv \in S^{d-1}$. We now extend this to establish bounds on the eigenvalues of the matrix $A = \frac1n\sum_{i=1}^ns_i\cdot\vx_i\vx_i^\top$. For any square symmetric matrix $B \in \bR^{d \times d}$, \cite[Lemma 4.4.1, Exercise 4.4.3(2)]{Vershynin2018} tells us that
\[
\norm B_2 \leq 2 \sup_{\vv \in \cN_{1/4}}\ \norm{\vv^\top B\vv}_2,
\]
where $\cN_{1/4}$ is a $\frac14$-net over the unit sphere $S^{d-1}$ that contains at most $9^d$ elements \cite[Corollary 4.2.13]{Vershynin2018}. Now, the results from Appendix~\ref{app:step3} tell us that for any fixed $\vv \in S^{d-1}$, we have
\[
\P{\frac1n\sum_{i=1}^ns_i\cdot\ip{\vx_i}{\vv}^2 \geq 1.001\Lambda} \leq \exp(-\Om n).
\]
Taking a union bound over all $\vv \in \cN_{\frac14}$ tells us that
\[
\P{\exists \vv \in \cN_{1/4}: \frac1n\sum_{i=1}^ns_i\cdot\ip{\vx_i}{\vv}^2 \geq 1.001\Lambda} \leq 9^d\cdot\exp(-\Om n) \leq \exp(-\Om n),
\]
for $n \geq \Om d$. This result, if we set $B = A - \Lambda\cdot I_d$, can be interpreted as saying that
\[
\sup_{\vv \in \cN_{1/4}}\ \norm{\vv^\top B\vv}_2 \leq 0.001\Lambda
\]
This gives us with confidence at least $1 - \exp(-\Om n)$,
\[
\norm{A - \Lambda\cdot I_d}_2 \leq 0.002\Lambda,
\]
or in other words, applying the triangle inequality $\norm{A - \Lambda\cdot I_d}_2 \geq \norm A_2 - \Lambda$, with the same confidence $1 - \exp(-\Om n)$,
\[
\lambda_{\max}\br{\frac1n\sum_{i=1}^ns_i\cdot\vx_i\vx_i^\top} \leq 1.002\Lambda,
\]
which concludes the high confidence upper bound on the eigenvalues of the Hessian $A = \nabla^2_\vw\cL_\tau(\vw)$. For the (leaky) softplus activations, a similar result holds, just with diminished confidence $1 - \exp\br{-\Om{\sqrt[3]n}}$.

Now, we move on to establish a high confidence lower bound on the eigenvalues of the Hessian $A = \nabla^2_\vw\cL_\tau(\vw)$. Let us establish an $\epsilon$-net $\cN_\epsilon$ over the unit sphere $S^{d-1}$. Such a net has atmost $\br{1+\frac2\epsilon}^d$ elements \cite[Corollary 4.2.13]{Vershynin2018}. For any $\vv \in S^{d-1}$, we use $\vc$ to denote its closest net element i.e. $\norm{\vv - \vc}_2 \leq \epsilon$. We now note the following chain of inequalities
\[
\lambda_{\min}(A) = \inf_{\vv \in S^{d-1}}\ \vv^\top A\vv = \vc^\top A\vc + 2(\vv - \vc)^\top A\vv + (\vv - \vc)^\top A(\vv - \vc) \geq \vc^\top A\vc - 2\epsilon\norm A_2 \geq \inf_{\vc \in \cN_\epsilon}\ \vc^\top A\vc - 2\epsilon\norm A_2,
\]
where we used the Cauchy-Schwartz inequality to get $\abs{(\vv - \vc)^\top A\vv} \leq \epsilon\cdot\norm A_2$ since $\norm\vv_2 = 1$, the fact that $\norm{\vv - \vc}_2 \leq \epsilon$ by construction, and $(\vv - \vc)^\top A(\vv - \vc) \geq 0$ since $A$ is always positive semi-definite. We use the upper bound on $\norm A_2 \leq 1.002\cdot\Lambda$ obtained above and set $\epsilon = \frac{0.001\lambda}{2.004\Lambda}$ where $\lambda, \Lambda$ are the upper and lower bounds on the expectations obtained in Appendices~\ref{app:step12-sigmoid}, \ref{app:step12-softplus}, and \ref{app:step12-leaky-softplus} above to get
\[
\lambda_{\min}(A) \geq \inf_{\vc \in \cN_\epsilon}\ \vc^\top A\vc - 0.001\lambda
\]
Now, the results from Appendix~\ref{app:step3} tell us that for any fixed $\vv \in S^{d-1}$, we have
\[
\P{\frac1n\sum_{i=1}^ns_i\cdot\ip{\vx_i}{\vv}^2 \leq .999\lambda} \leq \exp(-\Om n).
\]
Taking a union bound over all $\vv \in \cN_\epsilon$ tells us that
\[
\P{\exists \vv \in \cN_\epsilon: \frac1n\sum_{i=1}^ns_i\cdot\ip{\vx_i}{\vv}^2 \leq 0.999\lambda} \leq \br{\frac{2005\cdot\Lambda}\lambda}^d\cdot\exp(-\Om n) \leq \exp(-\Om n),
\]
for $n \geq \Om d$. This tells us that with confidence at least $1 - \exp(-\Om n)$, we have
\[
\lambda_{\min}\br{\frac1n\sum_{i=1}^ns_i\cdot\vx_i\vx_i^\top} \geq 0.998\lambda,
\]
which finishes the analysis. As before, for the (leaky) softplus activations, a similar result holds, just with diminished confidence $1 - \exp\br{-\Om{\sqrt[3]n}}$.

\subsection{Step 5}
\label{app:step5}
The results on the upper and lower bounds of the eigenvalues of the Hessian obtained above in Appendix~\ref{app:step4} hold for a fixed model $\vw$. We now show that essentially the same result holds uniformly over all models $\vw \in \cB(\vwo, r)$ for any $r \in [0,R]$. Let $\vw, \tvw \in \cB(\vwo, r)$ be any two models such that $\norm{\vw-\tvw}_2 \leq \varepsilon$ and let $s_i, \ts_i$ be the weights corresponding to these models so that
\begin{align*}
	A &\deff \nabla_\vw\cL_\tau(\vw) = \frac1n\sum_{i=1}^ns_i\cdot\vx_i\vx_i^\top\\
	\tilde A &\deff \nabla_\vw\cL_\tau(\tvw) = \frac1n\sum_{i=1}^n\ts_i\cdot\vx_i\vx_i^\top
\end{align*}
Now, we notice that using similar analyses as done in Appendices~\ref{app:step12-sigmoid}, \ref{app:step12-softplus}, and \ref{app:step12-leaky-softplus}, by utilizing that our activation functions are Lipschitz, we can show that
\[
\abs{s_i - \ts_i} \leq \bigO{\abs{(\vw-\tvw)^\top\vx_i}} \leq \epsilon R_X
\]
where we let $R_X \deff \max_{i\in[n]}\norm{\vx_i}_2$. This gives us, for any unit vector $\vv \in S^{d-1}$
\begin{align*}
	\abs{\vv^\top\nabla_\vw\cL_\tau(\vw)\vv - \vv^\top\nabla_\vw\cL_\tau(\tvw)\vv} &= \abs{\frac1n\sum_{i=1}^n(s_i-\ts_i)\cdot\ip{\vx_i}\vv^2}\\
	&\leq \frac1n\sum_{i=1}^n\abs{s_i-\ts_i}\cdot\ip{\vx_i}\vv^2\\
	&\leq \varepsilon R_X\cdot\frac1n\sum_{i=1}^n \ip{\vx_i}\vv^2\\
	&\leq \varepsilon R_X\cdot\frac1n\lambda_{\max}(X^\top X)
\end{align*}
where $X \in \bR^{n \times d}$ is the matrix of covariates stacked together. Standard results for instance \cite[Lemma 14]{BhatiaJK2015} tell us that with confidence at least $1 - \exp(-\Om n)$, we have $\frac1n\lambda_{\max}(X^\top X) \leq 5$. An application of the triangle inequality as done in Appendix~\ref{app:step4} to bound the eigenvalues of matrices now tells us that whenever $\norm{\vw-\tvw}_2 \leq \epsilon$, we always have
\begin{align*}
	\abs{\lambda_{\min}(A) - \lambda_{\min}(\tilde A)} &\leq 5\varepsilon R_X\\
	\abs{\lambda_{\max}(A) - \lambda_{\max}(\tilde A)} &\leq 5\varepsilon R_X
\end{align*}
Now, suppose $\lambda_{\min}(A) \geq 0.998\lambda$ and $\Lambda_{\min}(A) \leq 1.002\Lambda$, then setting $\varepsilon = \frac{0.001\lambda}{5R_x}$ tells us that $\lambda_{\min}(\tilde A) \geq 0.997\lambda$ and $\Lambda_{\min}(\tilde A) \leq 1.003\Lambda$ since $\Lambda \geq \lambda$.

Now, the results of Appendix~\ref{app:step4} tell us that for any fixed $\vw \in \cB(\vwo, r)$, we have
\[
\P{\lambda_{\min}\br{\nabla_\vw\cL_\tau(\vw)} \leq 0.998\lambda} \leq \exp(-\Om n)
\]
We now establish an $\varepsilon$-net $\cC_\varepsilon$ over the ball $\cB(\vwo, r)$ for $\varepsilon = \frac{0.001\lambda}{5R_x}$ and ensure that the eigenvalue bound holds for all $\br{1 + \frac r\varepsilon}^d$ elements of the net to get
\[
\P{\exists \vw \in \cC_\varepsilon: \lambda_{\min}\br{\nabla_\vw\cL_\tau(\vw)} \leq 0.998\lambda} \leq \br{1 + \frac r\varepsilon}^d\exp(-\Om n)
\]
To analyze the confidence term, we use standard results, for example \cite[Lemma 10]{MukhotyGJK2019}, that tell us that with confidence at least $1 - \delta$, we have $R_X \leq \bigO{\sqrt{d + \ln\frac n\delta}}$ i.e. whenever $n \geq \Om d$, with confidence at least $1 - \exp{-\Om n}$, we have $R_X \leq n$. This gives us $\varepsilon = \frac{0.001\lambda}{5n}$. This in turn gives us the following confidence bound
\[
\br{1 + \frac r\varepsilon}^d\exp(-\Om n) \leq \exp\br{-\Om{n - d\ln\frac{rn}\lambda}} \leq \exp\br{-\Om{n - d\ln n}} \leq \exp(-\Om n),
\]
whenever $n \geq \Om{d \ln d}$ (which ensures, say $\frac n{\ln n} \geq \Om d$). Combined with the analysis above along with the setting of $\varepsilon$ and the fact that elements of $\cC$ $\varepsilon$-cover the entire ball $\cB(\vwo, r)$ gives us
\[
\P{\exists \vw \in \cB(\vwo, r): \lambda_{\min}\br{\nabla_\vw\cL_\tau(\vw)} \leq 0.997\lambda} \leq \exp(-\Om n)
\]
A similar analysis holds for the upper eigenvalue bound as well
\[
\P{\exists \vw \in \cB(\vwo, r): \lambda_{\max}\br{\nabla_\vw\cL_\tau(\vw)} \geq 1.003\Lambda} \leq \exp(-\Om n)
\]
This finishes the argument.

\subsection{Step 6}
\label{app:step6}
The above results hold for fixed values of $r \in [0,R]$ and $\tau \in [0,1]$ but can be readily made uniform in them as well using a very similar technique as used in Appendix~\ref{app:step5}. We exploit the fact that the weights $s_i$ are Lipschitz in their dependence on the $\tau$ parameter, then set up an $\epsilon$-net over the interval $[0,1]$ and take a union bound. Taking $\epsilon = \frac1{n^{\bigO1}}$ ensures that the additional error incurred is negligible (for example, of the order of $0.001\lambda$ and $0.001\Lambda$ for the lower and upper eigenvalue bounds, respectively), as well as that the net size is manageable i.e. $n^{\bigO1}$ so that upon taking a union bound, we are left with confidence
\[
n^{\bigO1}\exp(-\Om n) \leq \exp(-\Om{n - \ln n}) \exp(-\Om n),
\]
for large enough $n \geq \Om{d\ln d}$. A similar argument holds for uniformity over the choice of $r$ as well. All together, these give us the final result as follows: 
\begin{align*}
\P{\exists r \in [0, R], \tau \in [0, 1], \vw \in \cB(\vwo, r): \lambda_{\min}\br{\nabla_\vw\cL_\tau(\vw)} \leq 0.996\lambda} &\leq \exp(-\Om n)\\
\P{\exists r \in [0, R], \tau \in [0, 1], \vw \in \cB(\vwo, r): \lambda_{\max}\br{\nabla_\vw\cL_\tau(\vw)} \geq 1.004\Lambda} &\leq \exp(-\Om n),
\end{align*}
that establishes Lemma~\ref{lem:elsc-elss}.

\begin{algorithm}[t]
	\begin{algorithmic}[1]
		{
		\REQUIRE Training data $\bc{(\vx_i,y_i}_{i=1}^n$, activation function $\phi$, initial temperature $\tau_0$, temperature increments $\beta_t > 1$, step lengths $\eta_t$
		\ENSURE An estimate $\hvw$ of the gold model $\vwo$
		\STATE Initialize $\vw_0$ and set $t \leftarrow 0$ \COMMENT{For example $\vw_0 \leftarrow \vzero$}
		\FOR{$t = 0, 1, 2, \ldots, T-1$}
			\STATE Obtain graduated labels $y^{\tau_t}_i = \phi_{\tau_t}(\phi^{-1}(y_i))$
			\STATE Construct the graduated objective function $\displaystyle \cL_{\tau_t}(\vw) = \frac1n\sum_{i=1}^n(y^{\tau_t}_i - \phi_{\tau_t}(\vx_i^\top\vw))^2$
			\STATE Rename $\vu^t_0 \leftarrow \vwt$
			\FOR{$i = 0, 1, 2, \ldots, I_t$}
				\STATE Obtain a stochastic gradient $\vg^t_i$ of the objective $\cL_{\tau_t}$ at the model $\vui$ i.e. $\E{\vg^t_i} = \left.\nabla_{\vu}\cL_{\tau_t}(\vu)\right|_{\vui}$
				\STATE Let $\vun \leftarrow \vui - \eta_t\cdot\vg^t_i$
				\STATE $i \leftarrow i + 1$
			\ENDFOR
			\STATE Rename $\vwn \leftarrow \vu^t_{I_t}$
			\STATE $\tau_{t+1} \leftarrow \min\bc{\beta_t\cdot\tau_t,1}$
			\STATE $t \leftarrow t + 1$
		\ENDFOR
		\STATE \textbf{return} {$\vw^T$}
		}
	\end{algorithmic}
	\caption{\alg-SGD Pseudocode}
	\label{algo:sgd}
\end{algorithm}

\section{Convergence Analysis for \alg with Stochastic Updates}
\label{app:app-stochastic}

Theorem~\ref{thm:conv-main-sigmoid} in the main paper offered a convergence analysis of the \alg variant that uses gradient descent in Algorithm~\ref{algo:main}. Theorem~\ref{thm:conv-sgd-sigmoid} below extends this to the SGD variant of \alg, presented in detail in Algorithm~\ref{algo:sgd}. This variant runs for $T$ epochs, with the $t\nth$ epoch lasting $I_t$ iterations with each iteration executing a single stochastic gradient descent step.

To analyze this variant, we need to make the following additional assumption:
\begin{assumption}
\label{asm:bound-variance}
For all $\tau \in [0,1]$ and $r \in [0,R]$,
\begin{enumerate}
	\item There exists a bound $B_r$ on the objective value i.e. $\cL_\tau(\vw) \leq B_r$ for all $\vw \in \cB(\vwo,r)$. However, to reduce notational clutter, we will use a \emph{uniform bound} i.e. fix $B \deff \sup_{r \in [0,R]}\ B_r$
	\item There exists some constant $s^2_r \geq 0$ such that for any $\vw \in \cB(\vwo,r)$, if $\vg_\tau(\vw)$ denotes a stochastic gradient for $\cL_\tau(\vw)$, then we have
\[
\E{\norm{\vg_\tau(\vw)}_2^2} \leq s^2_r + \norm{\nabla\cL_\tau(\vw)}_2^2
\]
\end{enumerate}
\end{assumption}
The quantity $s^2$ captures the noise variance introduced into the gradients due to stochasticity and becomes more promiment as $\vw$ approaches $\vwo$. For instance, at $\vwo$ we have $\nabla\cL_\tau(\vwo) = \vzero$ but $\E{\norm{\vg_\tau(\vw)}_2^2}$ is still strictly nonzero justifying the need to introduce this quantity. If we are performing gradient descent, then $s^2$ vanishes as there is no more noise in the gradient. This term is usually seen as preventing SGD from converging to models arbitrarily close to $\vwo$. However, steps can be taken to ameliorate this e.g. by using smaller/diminishing step-lengths, mini-batches with larger batch sizes, or by directly performing variance reducing steps such as SVRG \cite{JohnsonZhang2013}, SAGA \cite{DefazioBL-J2014}. For instance, in Figure~\ref{fig:conv}(a), and also Figures~\ref{fig:act_conv}(a,b,c), \alg-SGD does converge to models extremely close to the gold model.
\begin{theorem}[\alg-SGD Convergence]
\label{thm:conv-sgd-sigmoid}
Suppose Assumption~\ref{asm:bound-variance} is satisfied and let $\delta > 0$ be any confidence parameter. Let Algorithm~\ref{algo:main} is executed with stochastic gradient descent steps with temperature increment parameter set to some constant value $\beta > 1$, step length for epoch $t$ set to $\eta_t = \min\bc{\frac{r_t^2\lambda^2_{r_t}}{2\beta^2\Lambda_{r_t}s_{r_t}^2},\frac1{\Lambda_{r_t}}}$, and epoch $t$ run for $I_t \geq 2\ln\frac T\delta\br{1 + \frac1{\eta_t\lambda_{r_t}}\ln\frac{2B\lambda_{r_t}}{\eta_t\Lambda_{r_t}s_{r_t}^2}}^2$ iterations, where $\lambda_{r_t}, \Lambda_{r_t}$ are the ELSC/ELSS parameter corresponding to the $\tau_t$-graduated objective. Then with probability at least $1- \delta$ over the randomness in the stochastic gradients, we have $\norm{\vw_T - \vwo}_2^2 \leq \br{\frac1{\beta^2}}^T\cdot\norm{\vw_0-\vwo}_2^2$. In particular, for the case of the sigmoid activation, we may use $\eta_t = \bigO{\frac{R\exp(-R)}{\beta^{2t}}}$ and run epoch $t$ for $I_t \equiv I \leq \bigO{\frac{\beta^{4t}\exp(3R)}{R^4}\ln\frac T\delta}$ iterations.
\end{theorem}
\begin{proof}
We will adapt a standard proof of SGD convergence for strongly convex and smooth objectives by \cite{BottouCN2018}. Using similar arguments to those used in the proof of Theorem~\ref{thm:conv-main-sigmoid}, Theorem 4.6 from \cite{BottouCN2018} shows that for a choice of step length $\eta_t \leq \frac1{\Lambda_{r_t}}$, we have
\[
\E{\cond{\cL_{\tau_t}(\vun) - \cL_{\tau_t}(\vwo) - \epsilon_t}\vui} \leq (1-\eta_t\lambda_{r_t})\cdot\br{\cL_{\tau_t}(\vui) - \cL_{\tau_t}(\vwo) - \epsilon_t},
\]
where $\epsilon_t = \frac{\eta_t\Lambda_{r_t}s^2_{r_t}}{2\lambda_{r_t}}$. The above requires a minor modification from the proof strategy of \cite{BottouCN2018} -- instead of total expectations, we take a conditional expectation. This modification will help us use martingale inequalities later to offer a high probability bound. Simplifying using $\cL_{\tau_t}(\vwo) = 0$ in the noiseless setting, taking logarithms on both sides and using the Jensen's inequality (to get $\E{\ln(X)} \leq \ln(\E X)$ for any non-negative random variable $X$) gives us
\begin{equation}
\E{\cond{\ln\br{\cL_{\tau_t}(\vun) - \epsilon_t}}\vwt} \leq \ln\br{\cL_{\tau_t}(\vui) - \epsilon_t} - c_t,
\label{eq:super-martingale}
\end{equation}
where $c_t = \ln\br{\frac1{(1-\eta_t\lambda_{r_t})}} > 0$. Note that the above bound runs into problems in case the arguments to the logarithm function are negative or even approach zero. To ensure this never happens, we construct a \emph{stopping time} defined as $\bc{i: \cL_{\tau_t}(\vui) \leq 2\epsilon_t} \wedge I_t$ where we follow standard notation of using the $\wedge$ symbol to denote the minimum operator. This terminates the gradient descent procedure prematurely (i.e. before $I_t$ iterations are over for this epoch) if the loss is small enough since the objectives of training this epoch have already been met.

Let us now up a stochastic process $Z^t_1, Z^t_2, \ldots$ defined as
\[
Z^t_i \deff \ln\br{\cL_{\tau_t}(\vui) - \epsilon_t} + i \cdot c_t
\]
We now establish that this process actually constitutes a bounded-difference super-martingale. To show that it forms a super-martingale, we note that \eqref{eq:super-martingale} tells us that
\begin{align*}
	\E{\cond{Z^t_i}\vup} &= \E{\cond{\ln\br{\cL_{\tau_t}(\vui) - \epsilon_t} + i \cdot c_t}\vup}\\
	&\leq \ln\br{\cL_{\tau_t}(\vup) - \epsilon_t} - c_t + i \cdot c_t\\
	&= \ln\br{\cL_{\tau_t}(\vup) - \epsilon_t} + (i-1)\cdot c_t\\
	&= Z^t_{i-1}
\end{align*}
To establish the bounded differences property, we notice that by construction of the stopping time, we get
\[
\abs{Z^t_i - Z^t_{i-1}} \leq \abs{\ln\br{\cL_{\tau_t}(\vui) - \epsilon_t} - \ln\br{\cL_{\tau_t}(\vup) - \epsilon_t}} + c_t \leq \ln(B_{r_t}) - \ln(\epsilon_t) + \ln\br{\frac1{(1-\eta_t\lambda_{r_t})}} = C_t,
\]
where we define the shorthand $C_t \deff \ln\br{\frac{B_{r_t}}{\epsilon_t(1-\eta_t\lambda_{r_t})}}$. An application of the Azuma-Hoeffding's inequality now tells us that
\[
\P{Z_{I_t} - Z_0 > e} \leq \exp\br{-\frac{e^2}{2\cdot I_t\cdot C_t^2}} \leq \delta,
\]
by setting $e = C_t\sqrt{2I_t\ln\frac1\delta}$. Using the definition of $Z_{I_t}, Z_0$ tells us that with probability at least $1 - \delta$, we have
\[
\ln\br{\cL_{\tau_t}(\vu^t_{I_t}) - \epsilon_t} + {I_t} \cdot c_t - \ln\br{\cL_{\tau_t}(\vu^t_0) - \epsilon_t} \leq C_t\sqrt{2{I_t}\ln\frac1\delta}
\]
We wish to ensure $\cL_{\tau_t}(\vu^t_{I_t}) \leq 2\epsilon_t$ which is true whenever
\[
{I_t} \cdot c_t - C_t\sqrt{2{I_t}\ln\frac1\delta} \geq \ln\br{\frac{\epsilon_t}{\cL_{\tau_t}(\vu^t_0) - \epsilon_t}}
\]
By construction of the stopping time, we are assured that $\cL_{\tau_t}(\vu^t_0) \geq 2\epsilon_t$ which means that $\epsilon_t \leq \cL_{\tau_t}(\vu^t_0) - \epsilon_t$ i.e. we need only ensure
\[
{I_t} \cdot c_t - C_t\sqrt{2I\ln\frac1\delta} \geq 0
\]
This is fulfilled for all ${I_t} \geq \frac{2C_t^2\ln\frac1\delta}{c_t^2}$. Thus, we are assured that for long enough epochs, with high probability, we are assured at the end of the epoch, a model $\vwn = \vu^t_{I_t}$ such that $\cL_{\tau_t}(\vwn) \leq 2\epsilon_t$. Using the strong convexity property assured by ELSC along with $\cL_{\tau_t}(\vwo) = 0$ and $\nabla\cL_{\tau_t}(\vwo) = \vzero$, upon using the definition of the shorthand $\epsilon_t$, tells us that
\[
\norm{\vwn - \vwo}_2^2 \leq \frac{2\cL_{\tau_t}(\vwn)}{\lambda_{r_t}} \leq \frac{4\epsilon_t}{\lambda_{r_t}} = \frac{2\eta_t\Lambda_{r_t}s^2_{r_t}}{\lambda^2_{r_t}}
\]
We had $\norm{\vwt - \vwo}_2 \leq r_t$ and thus, we would really like $\norm{\vwn - \vwo}_2 \leq \frac {r_t}\beta$ so that we may increment our temperature by a factor of $\beta > 1$ at the end of the epoch to effect a linear rate of convergence across epochs as we did in the gradient descent case. Doing so requires us to set
\[
\eta_t \leq \min\bc{\frac{r_t^2\lambda^2_{r_t}}{2\beta^2\Lambda_{r_t}s^2_{r_t}},\frac1{\Lambda_{r_t}}}
\]
Note that since $\tau_t$ increases and $r_t$ decreases exponentially across epochs, the above schedule demands exponentially decreasing step lengths. As noticed earlier, this is unavoidable given the additional noise in the stochastic gradient process. Also, since $I_t$ depends inversely on $c_t^2$ which can be shown to depend on $\eta_t^2$ logarithmically on the inverse of $\epsilon_t$ which itself depends on $\eta_t$ by using $1 - x \leq \exp(-x)$, this implies that epoch lengths $I_t$ will go up exponentially across epochs. Thus we may expect upto $\exp(\bigO{T})$ iterations across the $T$ epochs. Since the overall decrease in error is also exponential, this represents a rate of error decrease that is inversely polynomial in the number of iterations which is known to be optimal\footnote{For instance, see Alekh Agarwal, Peter L. Bartlett, Pradeep Ravikumar, and Martin J. Wainwright. 2012. Information-Theoretic Lower Bounds on the Oracle Complexity of Stochastic Convex Optimization. \emph{IEEE Transactions On Information Theory} 58, 5 (2012), 3235 -- 3249.}. Given the above, we yet again find the algorithm converging to the gold model at a linear rate across epochs. All that remains is to set $\delta = \frac\delta T$ so that we may take a union bound to ensure that all epochs succeed in their Azuma-Hoeffding inequality, clean up bounds by making simplifications, which finishes the generic proof.

For the special case of the sigmoid activation, we have $\lambda_{r_t} \geq \tau_t^2\exp\br{-\frac R2}, \Lambda_{r_t} \leq \tau_t^2(1+R)$ and $s_{r_t}^2 \leq \bigO{\tau_t^2}$. Putting these into the generic expressions and simplifying, keeping in mind that $\tau_t \in [0,1]$ gives us the desired results.
\end{proof}

\section{Convergence Analysis for \alg under Noisy Labels}
\label{app:app-noisy}

We divide this section into several parts to better structure the discussion. Appendix~\ref{app:noise-models} first clarifies the noise models used in the analyses and presents a brief introduction to sub-Gaussian distributions. Appendix~\ref{app:elss-noisy} then discusses how ELSS constants can be established with minor changes in the presence of noisy labels. Appendix~\ref{app:elsc-noisy} shows how to establish ELSC constants under label noise using temperature capping. Appendix~\ref{app:conv-rate-noisy} then uses all these results to present a convergence bound for \alg under noisy labels, effectively establishing Theorem~\ref{thm:conv-noisy-sigmoid}. Finally, Appendix~\ref{app:consistent-noisy} shows how consistent recovery can be assured with low-variance post-activation noise. In all cases, we discuss the analysis and calculations keeping our running example of sigmoid activation but the same extends to other activation functions.

\subsection{Noise Models and sub-Gaussian Distributions}
\label{app:noise-models}
As discussed in Section~\ref{sec:analysis}, we use slightly different noise models in the pre- and post-activation settings.

\noindent\textbf{Pre-activation Noise.} In this case we have $y_i = \sigma(\vx_i^\top\vwo + \epsilon_i)$ where $\epsilon_i \sim \cD(0,\varsigma)$ is noise coming from some mean-centered \emph{sub-Gaussian} distribution. Mean centering ensures that $\E{\epsilon_i} = 0$. Note that even though the noise is unbiased, the noisy labels are still biased as $\E{y_i \cond \vx_i, \vwo} \neq \sigma(\vx_i^\top\vwo)$ as the non-linear activation function $\sigma$ acts on the noise as well. Sub-Gaussian distributions are popular in literature \cite{Vershynin2018} since they encompass a large variety of distributions including all Gaussian distributions, all distributions with bounded support, and admit Hoeffding and Bernstein-style tail bounds. There exist various equivalent definitions of sub-Gaussian distributions (see \cite[Proposition 2.5.2]{Vershynin2018}) and a popular one is as follows: a real-valued random variable $X$ is considered $\varsigma$-subGaussian if the moment generating function of $X$ satisfies the following inequality for all $\lambda \in \bR$
\[
\E{\exp(\lambda X} \leq \exp(\varsigma^2\lambda^2)
\]
Sub-Gaussian distributions also admit other interesting properties that we will invoke as and when required.

\noindent\textbf{Post-activation Noise.} In this case we have $y_i = \sigma(\vx_i^\top\vwo) + \epsilon_i$ where $\epsilon_i \sim \cB(0,\varsigma)$ is some unbiased noise i.e. $\E{\epsilon_i} = 0$. In this case, note that unbiased noise ensures that the noisy labels are unbiased as well since $\E{y_i \cond \vx_i, \vwo} = \sigma(\vx_i^\top\vwo)$ as in this case, the activation function does not act upon the noise. For any $\tau \in [0,1]$ we have $y^\tau_i = \sigma_\tau(\sigma^{-1}(\sigma(\vx_i^\top\vwo) + \epsilon_i))$. Post activation noise presents structural challenges since it may cause labels to violate their legal limits e.g. cause $y_i < 0$ or $y_i > 1$ which does not allow us to apply the inverse operation say $\sigma^{-1}$ in the case of the sigmoid activation function.

There may be several ways of overcoming this problem e.g. considering multiplicative noise instead of additive noise. However, we analyze a simple workaround here to present the essential ideas of the proof technique. The workaround involves normalizing data covariate norms and model norms and restricting the noise support. Specifically, for post-activation noise settings we will work with covariates of the form $\frac1{\sqrt{d}}\cdot\vx_i$, restrict $R = 1$ i.e. $\norm\vwo_2 \leq 1$ so that we can initialize the algorithm (and forever reside) within the unit ball. We also restrict the noise distribution to be unbiased and have bounded support $\cB(0,\varsigma) \subseteq [-\varsigma, \varsigma]$. In other words, $\abs{\epsilon_i} \leq \varsigma$ almost surely. The exact value of $\varsigma > 0$ will be set later.

\subsection{Calculation of ELSS Constants under Noisy Labels}
\label{app:elss-noisy}
We present calculations separately for pre-activation and post-activation noise settings using the sigmoid activation. Note that we need only redo steps 1 and 2 from Appendix~\ref{app:step12-sigmoid} here since the other steps follow identically since our noise is always sub-Gaussian (bounded distributions are always sub-Gaussian). For the analysis, the following result would prove handy:

\begin{remark}
\label{rem:log-inequality}
For any $a, b > 0$, in order to ensure $bt \leq \exp(-at)$, it is sufficient to ensure $t \leq \min\bc{\frac{\exp(-a)}b,1}$.
\end{remark}
\begin{proof}
Since $t \leq 1$ and $a > 0$, we are assured $\exp(-a) \leq \exp(-at)$. However, since $t \leq \frac{\exp(-a)}b$, we get $bt \leq \exp(-a) \leq \exp(-at)$ establishing the result.
\end{proof}

\subsubsection{Pre-activation Noise}
\label{app:elss-noisy-pre}
In this case we have $y_i = \sigma(\vx_i^\top\vwo + \epsilon_i)$ where $\epsilon_i \sim \cD(0,\varsigma)$ is sub-Gaussian noise. Thus, for any $\tau \in [0,1]$ we have $y^\tau_i = \sigma_\tau(\vx_i^\top\vwo + \epsilon_i)$. As before, we have
\[
\nabla^2_\vw\cL_\tau(\vw) = \frac1n\sum s_i\cdot\vx_i\vx_i^\top,
\]
where $s_i = 2\tau^2\br{\underbrace{\br{\sigma^\tau_i - (\sigma^\tau_i)^2}^2}_{(A)} + \underbrace{(\sigma^\tau_i - y^\tau_i)(1 - 2\sigma^\tau_i)\br{\sigma^\tau_i - (\sigma^\tau_i)^2}}_{(B)}}$. The analysis of the contribution of the term $(A)$ goes ahead identically since it does not depend on $y^\tau_i$ at all. However, the analysis for the second term changes since we now have
\[
\abs{(B)} \leq \frac\tau{40}\abs{\vx_i^\top(\vw - \vwo) - \epsilon_i}
\]
Analyzing the contribution of this term on expectation gives us
\begin{align*}
	\E{\abs{\vx^\top(\vw - \vwo) - \epsilon}\ip{\vx}{\vv}^2} &\leq \Ee{x_1, x_2 \sim \cN(0,1)}{\norm{\vw-\vwo}_2\cdot\abs{x_1}\cdot(x_1v_1 + x_2v_2)^2} + \Ee{\substack{\epsilon \sim \cD(0,\varsigma)\\\vx \sim \cN(\vzero, I_d)}}{\abs{\epsilon}\ip{\vx}{\vv}^2}\\
	&\leq r\cdot \Ee{x_1 \sim \cN(0,1)}{\abs{x_1}^3v_1^2 + \abs{x_1}v_2^2} + \Ee{\epsilon \sim \cD(0,\varsigma)}{\abs{\epsilon}} = 2r\sqrt{\frac2\pi} + 6\varsigma,
\end{align*}
where in the second step we used the independence of $\vx$ and $\epsilon$ and the fact that $\E{\ip\vx\vv^2} = \vv^\top\E{\vx\vx^\top}\vv = \vv^\top\vv = \norm\vv_2^2 = 1$ and the third step follows since $\E{\abs X} = \norm X_1 \leq 6\varsigma$ if $X$ is a $\varsigma$-sub-Gaussian random variable (implicit in \cite[Proposition 2.5.2]{Vershynin2018}). Putting all terms together gives us the new ELSS constant for the sigmoid activation.
\[
\Ee{\vx \sim \cN(\vzero, I)}{s\ip{\vx}{\vv}^2} \leq \Lambda_r \equiv \sqrt{\frac2\pi}\tau^2\br{\exp(-\tau\cdot m_r/10) + \frac{\tau\cdot r}{10} + \frac{3\tau\cdot\varsigma}{10}\sqrt{\frac\pi2}}
\]
Applying Remark~\ref{rem:log-inequality} and using the (somewhat loose) upper bound $m_r \leq R$ tells us that if we ensure
\[
\tau \leq 5\exp(-R/10)\cdot\min\bc{\frac1r, \frac1{3\varsigma}\sqrt{\frac2\pi}},
\]
then we are assured that $\Lambda_r \leq 2\tau^2\sqrt{\frac2\pi}\exp(-\tau\cdot R/10)$. The above is equivalent to a temperature cap of
\[
\tau_{\max} := \frac{5\exp(-R/10)}3\sqrt{\frac2\pi}\cdot\frac1\varsigma = \bigO{\frac1\varsigma}
\]
Notice that the only addition is the term corresponding to $\varsigma$ which increases the bound but does not change the essential nature of the bound. Steps 3-6 can now proceed as usual to get a high confidence bound. We next analyze the case of post-activation noise.

\subsubsection{Post-activation Noise}
\label{app:elss-noisy-post}
In this case we have $y_i = \sigma(\vx_i^\top\vwo) + \epsilon_i$ where $\epsilon_i \sim \cB(0,\varsigma) \subseteq [-\varsigma,\varsigma]$ is zero-mean, bounded noise. We also consider normalized covariates $\frac1{\sqrt{d}}\cdot\vx_i$ and restrict $R = 1$ i.e. $\norm\vwo_2 \leq 1$. \cite[Lemma 10]{MukhotyGJK2019}, that tell us that with confidence at least $1 - \delta$, we have $\max_{i\in[n]}\ \norm{\vx_i}_2 \leq \sqrt{d + \ln\frac n\delta}$. Setting $\delta = n\exp(-d)$ gives us, for the normalized covariates $R_X \leq \sqrt2$ with confidence at least $1 - \exp(-\Om d)$ (since we usually have $n = d^{\bigO1}$ so that $d - \ln n = \Om d$). Since $\norm\vw_2 \leq 1$ for all models considered by the algorithm by normalization, the \emph{canonical} parameter values satisfy $\abs{\vx_i^\top\vw} \leq \sqrt2$ which prevents the activated values from approaching either $0$ or $1$ since $\sigma(\vx_i^\top\vw) \in [0.195,0.805]$ if $\abs{\vx_i^\top\vw} \leq \sqrt2$. If we now restrict $\varsigma \leq 0.19$ we ensure that even after noise is added post activation, the labels remain within the valid range $[0.05,0.95]$. We now redo the expectation calculation for the ELSS constant. As before, we have
\[
\nabla_\vw\cL_\tau(\vw) = \frac1n\sum s_i\cdot\vx_i\vx_i^\top,
\]
where $s_i = 2\tau^2\br{\underbrace{\br{\sigma^\tau_i - (\sigma^\tau_i)^2}^2}_{(A)} + \underbrace{(\sigma^\tau_i - y^\tau_i)(1 - 2\sigma^\tau_i)\br{\sigma^\tau_i - (\sigma^\tau_i)^2}}_{(B)}}$. We need to analyze both parts here since there has been rescaling of covariates $\vx_i \mapsto \frac1{\sqrt d}\cdot\vx_i$ and our activation functions are non-linear. We have for part $(A)$,
\begin{align*}
	\E{\br{\sigma^\tau - (\sigma^\tau)^2}^2\frac{\ip{\vx_i}{\vv}^2}d} &= \frac1d\Ee{x_1, x_2 \sim \cN(0,1)}{\frac{\exp(-2\tau\norm{\vw}_2\cdot x_1)}{\br{1+\exp(-\tau\norm{\vw}_2\cdot x_1)}^4}\br{x_1^2v_1^2+x_2^2v_2^2+2x_1x_2v_1v_2}}\\
	&= \frac1d\Ee{x_1 \sim \cN(0,1)}{\frac{\exp(-2\tau/\sqrt d\norm{\vw}_2\cdot x_1)}{\br{1+\exp(-\tau/\sqrt d\norm{\vw}_2\cdot x_1)}^4}\br{x_1^2v_1^2+v_2^2}}\\
	&= \frac1d\sqrt{\frac2\pi}\int_0^\infty \frac{\exp(-2\tau/\sqrt d\norm{\vw}_2\cdot t)}{\br{1+\exp(-\tau/\sqrt d\norm{\vw}_2\cdot t)}^4}\br{v_1^2t^2+v_2^2}\exp\br{-\frac{t^2}2}\ dt\\
	&= \frac1d\sqrt{\frac2\pi}\exp(2c^2)\int_0^\infty \frac{\exp\br{-\frac12(t+2c)^2}}{\br{1+\exp(-c\cdot t)}^4}\br{v_1^2t^2+v_2^2}\ dt\\
	&\leq \frac1d\sqrt{\frac2\pi}\exp(2c^2)\int_0^\infty \exp\br{-\frac12(t+2c)^2}\br{v_1^2t^2+v_2^2}\ dt\\
	&\leq \frac1d\sqrt{\frac2\pi}\exp(-c/10),
\end{align*}
where $c = \tau/\sqrt d\cdot\norm\vw_2 \leq \tau/\sqrt d\cdot m_r$ and we recall that $m_r = \max_{\vw \in \cB(\vwo,r)}\ \norm\vw_2 \leq \norm\vwo_2 + r$. This gives us the contribution of the first term $(A)$. For the analysis of the second term, we now have
\begin{align*}
	\abs{(B)} &\leq \frac\tau{40}\abs{\frac{\vx_i^\top\vw}{\sqrt d} - \sigma^{-1}\br{\sigma\br{\frac{\vx_i^\top\vwo}{\sqrt d}} + \epsilon_i}}\\
	&\leq \frac\tau{40}\br{\abs{\frac{\vx_i^\top(\vw - \vwo)}{\sqrt d}} + 22\abs{\epsilon_i}},
\end{align*}
where the second step follows since the function $\sigma^{-1}$ is $22$-Lipschitz in the range $[0.05,0.95]$. From hereon, the analysis presented above for the pre-activation noise case in Appendix~\ref{app:elss-noisy-pre} holds with minor changes:
\begin{align*}
	\E{\br{\frac{\abs{\vx_i^\top(\vw - \vwo)}}{\sqrt d} + 22\abs{\epsilon_i}}\frac{\ip{\vx}{\vv}^2}d} &\leq \frac1{d\sqrt d}\cdot\Ee{x_1, x_2 \sim \cN(0,1)}{\norm{\vw-\vwo}_2\cdot\abs{x_1}\cdot(x_1v_1 + x_2v_2)^2}\\
	&\quad + \frac{22}d\Ee{\substack{\epsilon \sim \cD(0,\varsigma)\\\vx \sim \cN(\vzero, I_d)}}{\abs{\epsilon}\ip{\vx}{\vv}^2}\\
	&\leq \frac r{d\sqrt d}\cdot \Ee{x_1 \sim \cN(0,1)}{\abs{x_1}^3v_1^2 + \abs{x_1}v_2^2} + \Ee{\epsilon \sim \cD(0,\varsigma)}{\abs{\epsilon}}\\
	&= \frac{2r}{d\sqrt d}\sqrt{\frac2\pi} + \frac{66}d\varsigma,
\end{align*}
where in the second step we used the independence of $\vx$ and $\epsilon$ and the fact that $\E{\ip\vx\vv^2} = 1$ and the third step follows by applying the boundedness of $\cD$. Putting all terms together gives us the new ELSS constant for the sigmoid activation.
\[
\Ee{\vx \sim \cN(\vzero, I)}{\frac sd\ip{\vx}{\vv}^2} \leq \Lambda_r \equiv \sqrt{\frac2\pi}\frac{\tau^2}d\br{\exp\br{-\frac{\tau\cdot m_r}{10\sqrt d}} + \frac{\tau\cdot r}{10\sqrt d} + 3.3\tau\cdot\varsigma\sqrt{\frac\pi2}}
\]
Applying Remark~\ref{rem:log-inequality} and using the (somewhat loose) upper bound $m_r \leq R$ tells us that if we ensure
\[
\tau \leq \exp\br{-\frac R{10\sqrt d}}\cdot\min\bc{\frac{10\sqrt d}r, \frac1{3.3\cdot\varsigma}\sqrt{\frac2\pi}},
\]
then we are assured that $\Lambda_r \leq \frac{2\tau^2}{d}\sqrt{\frac2\pi}\exp\br{-\frac{\tau\cdot R}{10\sqrt d}}$. The above is assured by a temperature cap of
\[
\tau_{\max} := \frac13\exp\br{-\frac R{10\sqrt d}}\sqrt{\frac2\pi}\cdot\frac1{\varsigma} = \bigO{\frac1\varsigma}
\]
Notice that the bounds are scaled by an overall factor of the order of $\frac1d$ apart from minor changes to the internal constants. Steps 3-6 can now proceed as usual to get a high confidence bound.

\subsection{Calculation of ELSC Constants under Noisy Labels}
\label{app:elsc-noisy}

It turns out that the techniques used in Appendix~\ref{app:step12-sigmoid} used to establish ELSC constants without label noise and the techniques used in Appendix~\ref{app:elss-noisy} used to establish ELSS constants with label noise readily combine to offer us generic expressions for lower bounds on the ELSC parameters $\lambda_r$. However, to ensure that $\lambda_r > 0$, we need to impose additional temperature caps. Recall that we have $\nabla^2_\vw\cL_\tau(\vw) = \frac1n\sum s_i\cdot\vx_i\vx_i^\top$, where
\[
s_i = 2\tau^2\br{\underbrace{\br{\sigma^\tau_i - (\sigma^\tau_i)^2}^2}_{(A)} + \underbrace{(\sigma^\tau_i - y^\tau_i)(1 - 2\sigma^\tau_i)\br{\sigma^\tau_i - (\sigma^\tau_i)^2}}_{(B)}},
\]
where we continue to use the shorthands $\sigma^\tau_i := \sigma_\tau(\vx_i^\top\vw)$ and $y^\tau_i := \sigma_\tau(\sigma^{-1}(y_i))$. The term $(A)$ is unaffected by noise and its analysis proceeds similarly as before. However, the term $(B)$ is affected by label noise but was already bounded in the calculations in Appendix~\ref{app:elss-noisy}.

\subsubsection{Pre-activation Noise}
The contribution to the ELSC parameter $\lambda_r$ due to the term $(A)$ continues to remain lower bounded by $\sqrt{\frac2\pi}\tau^2\frac{\exp(-3\tau\cdot m_r)}8$ as observed in Appendix~\ref{app:step12-sigmoid} since $(A)$ does not contain the label $y_i$ at all and is thus unaffected by noise. Recall the shorthand $m_r = \max_{\vw \in \cB(\vwo,r)}\ \norm\vw_2 \leq \norm\vwo_2 + r$. The contribution of the term $(B)$ to diminish $\lambda_r$ on the other hand, apart from constant factors, can be directly taken from Appendix~\ref{app:elss-noisy} to be at most $2r\sqrt{\frac2\pi} + 6\varsigma$. This tells us gives us the following lower bound
\[
\Ee{\substack{\vx \sim \cN(\vzero, I)\\\epsilon\sim\cD(0,\varsigma)}}{s\ip{\vx}{\vv}^2} \geq \lambda_r \equiv \sqrt{\frac2\pi}\tau^2\br{\frac{\exp(-3\tau\cdot m_r)}8 - \frac{\tau\cdot r}{10} - \frac{3\tau\cdot\varsigma}{10}\sqrt{\frac\pi2}}
\]
Applying Remark~\ref{rem:log-inequality} and using the (somewhat loose) upper bound $m_r \leq R$ tells us that if we ensure
\[
\tau \leq \frac{5\exp(-3R)}{12}\cdot\min\bc{\frac1r, \frac1{3\varsigma}\sqrt{\frac2\pi}},
\]
then we are assured that $\lambda_r \geq \frac{\tau^2}{24}\sqrt{\frac2\pi}\exp(-3\tau\cdot R)$. The above is equivalent to a temperature cap of
\[
\tau_{\max} := \frac{5\exp(-3R)}{36}\sqrt{\frac2\pi}\cdot\frac1\varsigma = \bigO{\frac1\varsigma}
\]
Steps 3-6 can now proceed as usual to get a high confidence bound. We next analyze the case of post-activation noise.

\subsubsection{Post-activation Noise}
Due to the rescaling of covariates, the contribution to the ELSC parameter $\lambda_r$ due to the term $(A)$ is modified to now be lower bounded by $\sqrt{\frac2\pi}\frac{\tau^2}d\frac{\exp(-3\tau\cdot m_r/\sqrt d)}8$ as seen from calculations in Appendix~\ref{app:step12-sigmoid} and applying the rescaling modifications from Appendix~\ref{app:elss-noisy}. However, $(A)$ does not contain the label $y_i$ at all and is thus unaffected by noise. Recall that $m_r = \max_{\vw \in \cB(\vwo,r)}\ \norm\vw_2 \leq \norm\vwo_2 + r$. The contributions of the term $(B)$ are once again readily taken from Appendix~\ref{app:elss-noisy} to be at most $\frac{2r}{d\sqrt d}\sqrt{\frac2\pi} + \frac{66}d\varsigma$ apart from multiplicative constants. Put together, this gives us the following lower bound
\[
\Ee{\substack{\vx \sim \cN(\vzero, I)\\\epsilon\sim\cB(0,\varsigma)}}{s\ip{\vx}{\vv}^2} \geq \lambda_r \equiv \sqrt{\frac2\pi}\frac{\tau^2}d\br{\frac{\exp(-3\tau\cdot m_r/\sqrt d)}8 - \frac{\tau\cdot r}{10\sqrt d} - 3.3\tau\cdot\varsigma\sqrt{\frac\pi2}}
\]
Applying Remark~\ref{rem:log-inequality} and using the (somewhat loose) upper bound $m_r \leq R$ tells us that if we ensure
\[
\tau \leq \frac{\exp(-3R/\sqrt d)}8\cdot\min\bc{\frac{10\sqrt d}r, \frac1{3.3\cdot\varsigma}\sqrt{\frac2\pi}},
\]
then we are assured that $\lambda_r \geq \frac{\tau^2}{24d}\sqrt{\frac2\pi}\exp\br{-\frac{3\tau\cdot R}{\sqrt d}}$. The above is assured by a temperature cap of
\[
\tau_{\max} := \frac{\exp(-3R/\sqrt d)}{26}\sqrt{\frac2\pi}\cdot\frac1{\varsigma} = \bigO{\frac1\varsigma}
\]
Steps 3-6 can now proceed as usual to get a high confidence bound.

\subsection{Convergence Rate for \alg under Noisy Labels: a proof of Theorem~\ref{thm:conv-noisy-sigmoid}}
\label{app:conv-rate-noisy}

We will elaborate this argument taking the example of sigmoid activation as before. \alg as described in Algorithms~\ref{algo:main} and \ref{algo:sgd} increments the temperature till it hits the temperature cap. For noiseless settings, the temperature cap is the trivial unity but it is a non-trivial cap $\tau_{\max}$ as calculated in Appendices~\ref{app:elss-noisy} and \ref{app:elsc-noisy}. This temperature cap will be hit in $\bigO1$ steps of temperature increments with the constants depending on the temperature increment rate $\beta$ and the initial temperature, both of which depend on $R$.

After the temperature cap has been reached, standard results on the convergence of GD for strongly convex and smooth objectives tell us that \alg will start optimizing the objective $\cL_\taumax$ (since the temperature cannot be incremented anymore) and will approach its optimum at a linear rate of convergence. Let $\tvw = \arg\min\cL_\taumax$ denote the optimum of this objective. To prove Theorem~\ref{thm:conv-noisy-sigmoid}, all we need to do is upper bound $\norm{\vwo - \tvw}_2$ which we do below. In the following $\lambda, \Lambda$ denote the ELSC/ELSS parameters assured at the point the temperature cap is reached.

First, we apply strong convexity of $\cL_{\taumax}$ twice to obtain
\begin{align*}
	\cL_\taumax(\tvw) &\geq \cL_\taumax(\vwo) + \ip{\nabla\cL_\taumax(\vwo)}{\vwo - \tvw} + \frac{\lambda}2\norm{\vwo - \tvw}_2^2\\
	\cL_\taumax(\vwo) &\geq \cL_\taumax(\tvw) + \ip{\nabla\cL_\taumax(\tvw)}{\tvw - \vwo} + \frac{\lambda}2\norm{\vwo - \tvw}_2^2
\end{align*}
Adding the two equations and noticing that $\nabla\cL_\taumax(\tvw)$ since $\tvw$ is the optimum of the differentiable objective $\cL_\taumax$and applying the Cauchy-Schwartz inequality, we get
\[
\norm{\vwo - \tvw}_2 \leq \frac1\lambda\cdot\norm{\nabla\cL_\taumax(\vwo)}_2
\]
In the remainder of the proof, we bound $\norm{\nabla\cL_\taumax(\vwo)}_2$. We use the shorthands $\sigma^\tau_i \deff \sigma_\tau(\vx_i^\top\vwo), y^\tau_i = \sigma_\tau(\sigma^{-1}(y_i))$. We first note the expression of this gradient for the sigmoid activation case. To avoid notation clutter, we use $\tau$ instead of $\taumax$ everywhere.
\[
\nabla\cL_\tau(\vwo) = \frac{2\tau}n\sum_{i=1}^n(\sigma^\tau_i - y^\tau_i)\br{\sigma^\tau_i - (\sigma^\tau_i)^2}\cdot\vx_i = \frac{2\tau}nX\vzeta,
\]
where $\vzeta \in \bR^n$ is the vector collecting the values $\bs{(\sigma^\tau_i - y^\tau_i)\br{\sigma^\tau_i - (\sigma^\tau_i)^2}}_{i=1}^n$. Now, standard results for instance \cite[Lemma 14]{BhatiaJK2015} tell us that with confidence at least $1 - \exp(-\Om n)$, we have $\norm X_2 \leq \sqrt{5n}$. This means that
\[
\norm{\nabla\cL_\tau(\vwo)}_2 = \sqrt{\frac{20\tau^2}n}\norm\vzeta_2 \leq \sqrt{\frac{20\tau^2}n\sum_{i=1}^n \br{\sigma^\tau_i - y^\tau_i}^2},
\]
since $\br{\sigma^\tau_i - (\sigma^\tau_i)^2} \in [0,1]$ for all $i \in [n]$. Thus, the following quantity needs to be bounded which we do separately below for the pre- and post-activation settings.
\[
(Q) \deff \sum_{i=1}^n \br{\tsigma^\tau_i - y^\tau_i}^2
\]
\subsubsection{Pre-Activation Noise}
In this case we have $y^\tau_i = \sigma_\tau(\vx_i^\top\vwo + \epsilon_i)$ where $\epsilon_i \sim \cD(0,\varsigma)$ so that if we let $\vepsilon = [\epsilon_1,\epsilon_2,\ldots,\epsilon_n] \in \bR^n$ be the vector of noise values used to corrupt the labels pre-activation, then we have
\begin{align*}
	(Q) &= \sum_{i=1}^n \br{\sigma_\tau(\vx_i^\top\vwo) - \sigma_\tau(\vx_i^\top\vwo + \epsilon_i)}^2\\
	&\leq \frac{\tau^2}{16}\sum_{i=1}^n \br{\vx_i^\top\vwo - (\vx_i^\top\vwo + \epsilon_i))}^2\\
	&\leq \frac{\tau^2}{16}\norm\vepsilon_2^2 \leq \frac{5\tau^2}{16}n\varsigma^2 \leq n\tau^2\varsigma^2
\end{align*}
where the second step follows since the $\sigma_\tau(\cdot)$ function is $\frac\tau4$-Lipschitz and the fourth step follows from standard results for sub-Gaussian random variables \cite[Proposition 2.5.2]{Vershynin2018} that show that $\E{X^2} = \norm X_2^2 = \bigO{\varsigma^2}$ for a $\varsigma$-sub Gaussian variable $X$ and applying Bernstein inequality to the subexponential variable $X^2$ to see that with confidence at least $1 - \exp(-\Om n)$, we have $\norm\vepsilon_2^2 \leq 5n\varsigma^2$.

\subsubsection{Post-Activation Noise}
In this case we have $y^\tau_i = \sigma_\tau\br{\sigma^{-1}\br{\sigma\br{\frac{\vx_i^\top\vwo}{\sqrt d}} + \epsilon_i}}$ where $\epsilon_i \sim \cB(0, \varsigma) \subseteq [-\varsigma,\varsigma]$ is the finite-support noise distribution introduced in Appendix~\ref{app:noise-models}. Recall that covariates are normalized in this setting. Continuing to use the notation $\vepsilon \in \bR^n$ to denote the vector of noise values used to corrupt the labels post-activation, then we have
\begin{align*}
	(Q) &= \sum_{i=1}^n \br{\sigma_\tau\br{\frac{\vx_i^\top\vwo}{\sqrt d}} - \sigma_\tau\br{\sigma^{-1}\br{\sigma\br{\frac{\vx_i^\top\vwo}{\sqrt d}} + \epsilon_i}}}^2\\
	&\leq \frac{\tau^2}{16}\sum_{i=1}^n \br{\frac{\vx_i^\top\vwo}{\sqrt d} - \sigma^{-1}\br{\sigma\br{\frac{\vx_i^\top\vwo}{\sqrt d}} + \epsilon_i}}^2\\
	&\leq \frac{22\tau^2}{16} \norm\vepsilon_2^2 \leq \frac{22\tau^2}{16}n\varsigma^2 \leq 2n\tau^2\varsigma^2
\end{align*}
where the second step follows since the $\sigma_\tau(\cdot)$ function is $\frac\tau4$-Lipschitz, the third step uses the fact that the inverse function $\sigma^{-1}$ is $22$-Lipschitz in the interval $[0.05,0.95]$ and we have ensured that all labels reside in that interval, and the last step follows since $\abs{\epsilon_i} \leq \varsigma$ since the noise distribution has bounded support.

Thus, in both cases, apart from exact constants, we obtain $(Q) \leq 2n\tau^2\varsigma^2$. This tells us that
\[
\norm{\nabla\cL_\taumax(\vwo)}_2 \leq \varsigma\tau_{\max}^2\sqrt{40} = \bigO{\varsigma}
\]
This finishes the proof of Theorem~\ref{thm:conv-noisy-sigmoid}.

\subsection{Consistent Recovery under Low-variance Post-activation Noise: a Proof of Theorem~\ref{thm:consistent-post-noisy-sigmoid}}
\label{app:consistent-noisy}
For the post-activation noise case, the analyses in Appendices~\ref{app:elss-noisy} and \ref{app:elsc-noisy} respectively impose a temperature cap of
\begin{align*}
	\tau_{\max}^\text{ELSS} &= \frac13\exp\br{-\frac R{10\sqrt d}}\sqrt{\frac2\pi}\cdot\frac1{\varsigma}\text{ and}\\
	\tau_{\max}^\text{ELSC} &= \frac{\exp(-3R/\sqrt d)}{26}\sqrt{\frac2\pi}\cdot\frac1{\varsigma}
\end{align*}
respectively. $\tau_{\max}^\text{ELSC}$ being smaller is the stricter of the two. Suppose $\varsigma$ is small enough so that $\tau_{\max}^\text{ELSC} \geq 1$ i.e. suppose
\[
\varsigma \leq \frac{\exp(-3R/\sqrt d)}{26}\sqrt{\frac2\pi}.
\]
Then the temperature cap becomes vacuous since \alg anyway never considers temperature values $\tau > 1$. Thus, in this case, we effectively have $\taumax = 1$. However, notice that in this case, since $\sigma_1 \equiv \sigma$ we have
\begin{align*}
	\nabla\cL_\taumax(\vwo) = \nabla\cL_1(\vwo) &= \frac{2\tau}n\sum_{i=1}^n(\sigma^1_i - y^1_i)\br{\sigma^1_i - (\sigma^1_i)^2}\cdot\vx_i\\
	&= \frac{2\tau}n\sum_{i=1}^n\br{\sigma\br{\frac{\vx_i^\top\vwo}{\sqrt d}} - \sigma\br{\sigma^{-1}\br{\sigma\br{\frac{\vx_i^\top\vwo}{\sqrt d}} + \epsilon_i}}}\br{\sigma^1_i - (\sigma^1_i)^2}\cdot\vx_i\\
	&= \frac{2\tau}n\sum_{i=1}^n\epsilon_i\br{\sigma^1_i - (\sigma^1_i)^2}\cdot\vx_i
\end{align*}

Now notice that since noise is unbiased, we have $\E{\cond{\epsilon_i}\vx_i,\vwo} = 0$. By linearity of expectation, for any fixed unit vector $\vv \in S^{d-1}$, we have
\[
\E{\ip{\nabla\cL_\taumax(\vwo)}\vv} = \frac{2\tau}n\cdot\E{\sum_{i=1}^n\epsilon_i\br{\sigma^1_i - (\sigma^1_i)^2}\cdot\ip{\vx_i}\vv} = 0
\]
since noise is independently sampled and unbiased. Now, since $\sigma^1_i - (\sigma^1_i)^2 \leq 1$, the noise is $\varsigma$-bounded hence $\varsigma$-sub-Gaussian by assumption and the random variable $\ip{\vx_i}{\vv}$ is $1$-subGaussian since $\vv$ is unit norm and $\cN(\vzero, I_d)$ is a $1$-sub-Gaussian vector distribution, the random variable $\epsilon_i\br{\sigma^1_i - (\sigma^1_i)^2}\cdot\ip{\vx_i}\vv$ is $\varsigma$-subexponential. Applying the Bernstein bound \cite[Theorem 2.8.1]{Vershynin2018} tells us that
\[
\P{\abs{\sum_{i=1}^n\epsilon_i\br{\sigma^1_i - (\sigma^1_i)^2}\cdot\ip{\vx_i}\vv} \geq t} \leq 2\exp\br{-\frac{ct^2}{n\varsigma^2}},
\]
for some small universal constant $c > 0$. Taking a union bound over a $\frac12$-net $\cN_{1/2}$ over the unit sphere (that has atmost $5^d$ elements) tells us that
\[
\P{\exists \vv \in \cN_{1/2}: \abs{\sum_{i=1}^n\epsilon_i\br{\sigma^1_i - (\sigma^1_i)^2}\cdot\ip{\vx_i}\vv} \geq t} \leq 2\cdot 5^d\exp\br{-\frac{ct^2}{n\varsigma^2}}
\]
If for any $\vv \in S^{d-1}$, we let $\vv_0$ denote its closest net point i.e. $\norm{\vv - \vv_0}_2 \leq \frac12$, then for any vector $\vz \in \bR^d$ we have
\[
\norm\vz_2 = \sup_{\vv \in S^{d-1}}\ip\vz\vv = \ip\vz{\vv_0} + \ip\vz{\vv-\vv_0} \leq \sup_{\vv_0 \in \cN_{1/2}}\ \ip\vz{\vv_0} + \frac12\norm\vz_2,
\]
where in the last step we used the Cauchy-Schwartz inequality. This tells us that
\[
\norm\vz_2 \leq 2\sup_{\vv_0 \in \cN_{1/2}}\ \ip\vz{\vv_0}
\]
The above argument tells us that
\[
\P{\norm{\sum_{i=1}^n\epsilon_i\br{\sigma^1_i - (\sigma^1_i)^2}\cdot\vx_i}_2 \geq 2t} \leq 2\cdot 5^d\exp\br{-\frac{ct^2}{n\varsigma^2}}
\]
Taking $t = \sqrt{\frac{3dn\varsigma^2}c}$ and introducing the constant factors tells us that with probability at least $1-\exp(-3d)$, we have
\[
\norm{\nabla\cL_\taumax(\vwo)}_2 = \frac{2\tau}n\cdot\norm{\sum_{i=1}^n\epsilon_i\br{\sigma^1_i - (\sigma^1_i)^2}\cdot\vx_i}_2 \leq 4\tau\varsigma\sqrt{\frac d{cn}}
\]
This implies that $\norm{\tvw - \vwo}_2 \leq \bigO{\varsigma\sqrt{\frac dn}}$. Due to the linear rate of convergence offered by standard guarantees for GD-style methods on strongly convex and smooth objectives, within $\bigO{\ln\frac1\epsilon}$ iterations after reaching unit temperature (which itself takes $\bigO1$ iterations), \alg is able to reach within $\epsilon$ distance of $\tvw = \arg\min\cL(\vw)$ which the above analysis shows is at most $\bigO{\varsigma\sqrt{\frac dn}}$ away from $\vwo$. Applying the triangle inequality now finishes the proof of Theorem~\ref{thm:consistent-post-noisy-sigmoid}.

\end{document}